\documentclass[10pt,a4paper]{article}

\usepackage{cite}
\usepackage{graphicx}
\usepackage{amssymb,amsmath,amsthm}
\usepackage{esint}
\usepackage{tikz}
\usepackage{bbm,dsfont,bm}
\usepackage{color}

\theoremstyle{definition}
\newtheorem{theorem}{Theorem}[section]

\newtheorem{lemma}{Lemma}[section]

\newtheorem{definition}{Definition}[section]

\newtheorem{remark}{Remark}[section]
\newtheorem{proposition}{Proposition}[section]

\usepackage{times}
\usepackage{float}
\usepackage{graphicx}

\usepackage{amsmath,amsthm,amssymb,enumerate, fancyhdr}
\usepackage[english]{babel}
\usepackage[colorlinks=true, linkcolor=red]{hyperref}
\usepackage{thmtools}

\newtheorem*{theorem*}{Theorem}

\numberwithin{equation}{section}

\usepackage{graphicx}
\usepackage{tikz}

\usepackage[left=2cm,right=2cm,top=2cm,bottom=2cm]{geometry}
\usepackage{authblk}
\begin{document}
\title{Existence of weak solutions for incompressible fluid– Koiter shell interactions with Navier slip boundary condition}



\author[1,2]{Claudiu M\^{i}ndril\u{a}, Arnab Roy}

\maketitle
   \begin{abstract}
We study a three-dimensional fluid--structure interaction problem describing the motion of an incompressible, viscous fluid coupled with a deformable elastic shell of Koiter type that forms part of the fluid boundary. The fluid motion is governed by the incompressible Navier--Stokes equations posed on a time-dependent domain, while the shell evolution is described by a nonlinear elastic model. At the fluid--structure interface, we impose Navier slip boundary conditions, allowing for tangential slip penalized by friction. Our main result establishes the global-in-time existence of weak solutions up to the first possible self-intersection of the shell, for arbitrarily large initial data with finite energy. The analysis is carried out in a fully three-dimensional setting and addresses the major mathematical challenges arising from the moving domain, the geometric nonlinearity of the shell, and the reduced regularization induced by the slip boundary condition. The proof relies on a careful construction of suitable approximation schemes, novel compactness arguments adapted to the slip framework, and a new extension operator for divergence-free test functions compatible with the fluid--shell coupling. As a further contribution, we provide a direct approach to the strong convergence of second-order spatial derivatives of the shell displacement, which allows us to treat nonlinear Koiter shell models within the same framework. 
\end{abstract}


\section{Introduction}\label{sec:intro}

Fluid-structure interaction (FSI) problems arise in many areas of science and engineering, where the motion of a rigid or deformable structure interacts dynamically with a surrounding fluid. Such problems are inherently nonlinear and multi-physical, combining the analytical challenges of fluid dynamics and elasticity. In this work, we study the coupled evolution of an incompressible, viscous fluid and a deformable elastic cylindrical shell that constitutes part of the boundary of the fluid domain. The fluid occupies a three-dimensional time-dependent domain whose motion is induced by the deformation of the shell. The fluid dynamics are governed by the incompressible Navier-Stokes equations, while the elastic structure is modeled by a Koiter-type shell theory.

A distinguishing feature of our study is the use of \emph{Navier slip} boundary conditions at the fluid–structure interface. In contrast to the classical no-slip condition, Navier slip boundary conditions allow for tangential relative motion along the interface, penalized by a friction term. Such boundary conditions are physically relevant in several regimes, including flows near hydrophobic surfaces, microfluidic applications, and fluid–biological tissue interactions. In particular, recent in- vivo measurements reported in \cite{malek-in-vivo} indicate a slip character of blood flow in vessels; see also \cite{malek2024determining} and \cite{bulicek-malek-rajagopal-sima} for further discussion and experimental evidence.
From a mathematical viewpoint, Navier slip boundary conditions are more subtle than their no-slip counterparts, as they weaken the coupling at the interface and reduce the regularizing effect of the fluid–structure interaction. For incompressible viscous fluids in fixed domains and under suitable regularity assumptions on the boundary, the existence theory is by now well established; see, for example, \cite{bulivcek2007navier,bulivcek2009navier}. More recently, well-posedness results for the Stokes system under Navier slip boundary conditions in irregular domains were obtained in \cite{breit-schwarz-sslip}, see also \cite{gahn2025effectiveinterfacelawsnaviersliptype}.

The mathematical analysis of fluid–structure interaction problems under no-slip boundary conditions has attracted considerable attention in recent years. For viscous incompressible fluids, we refer to \cite{chambolle2005existence,LR14,lengeler2014weak,MS22,cheng2010interaction} and the references therein, while for inviscid fluids we mention, among many other works, \cite{alazard2025global,alazard2025global-arxiv}. In contrast, significantly less is known in the presence of slip boundary conditions, where the weakened coupling at the fluid–structure interface introduces substantial analytical difficulties, particularly in the derivation of compactness and strong convergence properties.
In two spatial dimensions, the existence of weak solutions for fluid–structure interaction problems with Navier slip boundary conditions was established by \v{C}ani\'{c} and Muha in \cite{muha-canic-slip}, while a stochastic counterpart was recently obtained by Tawri in \cite{tawri2024stochastic} using related techniques. In fact, the recent monograph \cite{Canic-book} offers an overview on the subject and the methodology based on the arbitrary lagriangian eulerian method. In particular we refer to \cite[Section 6.3]
{Canic-book} for results related to the slip boundary conditions.

In three spatial dimensions, to our best  knowledge, the only available existence result concerns strong solutions and is proved by Djebour and Takahashi in \cite{takeo-imene-slip}. There, local-in-time existence and uniqueness of strong solutions are obtained, together with global-in-time well-posedness for small initial data.
Recently, Mitra and Schwarzacher \cite{mitra2025thermal} investigated incompressible and heat-conducting fluids governed by the Navier–Stokes–Fourier system interacting with nonlinear Koiter shells under no-slip boundary conditions. Their analysis highlights the delicate role of the pressure in coupled fluid–structure systems and motivates further investigation of slip boundary conditions in thermodynamically consistent settings; see also \cite[Remark~1.1]{mitra2025thermal} and \cite[Subsection~3.2]{bulivcek2009navier}.

Motivated by the above developments, we establish the existence of weak solutions to a coupled incompressible fluid–elastic shell system under Navier slip boundary conditions in three spatial dimensions. The fluid occupies a time-dependent domain whose boundary is determined by the deformation of the shell. The principal analytical challenges arise from the geometric nonlinearity induced by the moving interface, the reduced dissipation caused by the slip condition, and the subtle interaction between the fluid pressure and the elastic response of the shell.
Our approach builds on the existence framework introduced in \cite{LR14}, which has proven effective for a broad class of fluid–structure interaction problems. In particular, this methodology was further developed in \cite{BS18,BS21} to treat compressible, viscous, and heat-conducting fluids, thereby complementing the incompressible setting of \cite{LR14}. However, the presence of Navier slip boundary conditions necessitates substantial new analytical ingredients, especially at the level of compactness and strong convergence.

\paragraph{Novelty and significance.}
Our main result, Theorem~\ref{thm:main}, establishes the existence of a weak solution on any time interval \( I=[0,T] \) with \( T<T_\star \), where \( T_\star\in(0,\infty] \) denotes the first possible time of self-intersection of the shell. The maximal time \( T_\star \) depends on the initial data, which are assumed to have finite energy and to satisfy natural compatibility conditions specified below. Theorem~\ref{thm:main} is proved for a linearized membrane energy (see \eqref{eqn:lame}). We subsequently extend the analysis to a class of nonlinear elastic energies, leading to Theorem~\ref{thm:main-nonlinear}, which covers nonlinear Koiter shell models.

The presence of Navier slip boundary conditions introduces several new analytical difficulties, most notably due to the reduced dissipation and weakened interface coupling. The principal novelties of this work can be summarized as follows:
\begin{itemize}
\item \textbf{New compactness arguments.}  
We develop refined \(L^2\)-compactness techniques in Proposition~\ref{prop:compactness}, extending the framework of \cite{LR14} to the setting of slip boundary conditions.

\item \textbf{Slip-adapted extension operators.}  
We construct a novel extension operator for shell test functions that yields divergence-free test functions defined on the full fluid domain, while preserving the coupling only in the normal direction along the moving interface. This construction is crucial in the presence of slip and is detailed in Proposition~\ref{prop:extension-slip}.

\item \textbf{Strong convergence of second-order shell gradients.}  
We provide a direct proof of the strong convergence of the second spatial derivatives of the shell displacement\footnote{Due to the Piola transform, spatial differentiation naturally produces second-order derivatives of the shell displacement.} in \(L^2(I;L^2)\)-- in Proposition~\ref{prop:2nd-gradients-convergence}.

To the best of our knowledge, such convergence was first obtained in the no-slip setting in \cite{MS22} using difference quotient techniques. Our approach yields a more direct argument, tailored to the slip framework. This result plays a crucial role in extending the analysis to nonlinear elastic energies and, in particular, to nonlinear Koiter shell models, as presented in Section~\ref{sec:nonlinear-koiter} and Theorem~\ref{thm:main-nonlinear}.
\end{itemize}

Let us also emphasize that our model involves \emph{purely elastic} shells, without any additional visco-elastic regularization. This feature represents a significant analytical challenge and distinguishes our approach from previous works on weak solutions with slip boundary conditions, such as \cite{muha-canic-slip} and \cite{tawri2024stochastic}, where additional regularity---typically in the form of visco-elastic damping---is required to ensure Lipschitz continuity  of the change-of-variables mapping between the reference and deformed configurations. In contrast, our analysis does not rely on such artificial regularization mechanisms. For strong solutions, a related result in the presence of Navier slip boundary conditions was obtained by Djebour and Takahashi in \cite{takeo-imene-slip} in the case of an incompressible fluid interacting with a linear elastic plate. Their analysis, however, requires highly regular and sufficiently small initial data, leading to short-time well-posedness results. In comparison, our framework accommodates arbitrarily large initial data, subject only to a finite energy condition and natural compatibility assumptions. The maximal time of existence is constrained solely by the possible onset of geometric degeneracies, namely self-intersections of the shell.

\paragraph{Overview.}The paper is organized as follows. In the remainder of Section~\ref{sec:intro}, we introduce the mathematical setting of the problem, describe the geometry of the reference and evolving domains, and formulate the coupled fluid--structure interaction model. In Section~\ref{sec:weak-form-main-res}, we define the notion of weak (distributional) solutions in Definition~\ref{def:weak-soln}, derive the fundamental a priori estimates, and state our main results, namely Theorem~\ref{thm:main} and Theorem~\ref{thm:main-nonlinear}. Section~\ref{sec:tools} collects the analytical tools and auxiliary results that are repeatedly used throughout the paper. In Section~\ref{sec:proof-main}, we present the main steps of the proof of Theorem~\ref{thm:main}. Finally, Section~\ref{sec:nonlinear-koiter} is devoted to the analysis of nonlinear Koiter shell models and the proof of Theorem~\ref{thm:main-nonlinear}.

\subsection{Geometry}

Let $\Omega \subset \mathbb{R}^{3}$ be a cylinder of radius $R$ and length $L$. We will make use of cylindrical coordinates and we can write
\begin{equation*}
\Omega:=\left\{ \left(x,y,z\right):\ x=r\cos\theta,\ y=r\sin\theta,\ r\in\left[0,R\right),\ \theta\in\left[0,2\pi\right),\ z\in\left(0,L\right)\right\}. 
\end{equation*}
We are interested in the evolution in time of  the \emph{flexible part of the boundary} which will be denoted by
\begin{equation}
\Gamma:=\left\{ \left(x,y,z\right):x=R\cos\theta,y=R\sin\theta,\ \left(\theta,z\right)\in\omega\right\},  
\end{equation}
where $\omega := (0,2\pi)\times (0,L) \subset \mathbb{R}^{2}$  and defines a parametrization of $\Gamma$ with the mapping \[\phi:\omega\mapsto\Gamma\subset\mathbb{R}^{3},\quad\phi\left(\theta,z\right)=\left(R\cos\theta,R\sin\theta,z\right),\ \forall\left(\theta,z\right)\in\omega.\] 

The displacement of the elastic structure is denoted by $\boldsymbol{\eta}:\omega \mapsto \mathbb{R}^{3}$ and we assume it happens in the outer normal/radial direction $\mathbf{e}_{r}(\theta)=(\cos \theta, \sin \theta,0)$  for each $t\in I$, where $I:=[0,T]$ denotes a time horizon of the evolution with $T>0$.
Thus, it suffices to study the radial component of $\boldsymbol{\eta}$. So we write $\boldsymbol{\eta}=\eta \mathbf{e}_{r}$ and we study  the scalar function
\[
\eta: I\times \omega \mapsto \mathbb{R}\]
which is the unknown of the elastic equation from now on. Then, the \emph{deformed boundary} at each $t\in I$ is given by 
\begin{equation}
\Gamma^{\eta}\left(t\right):=\left\{ \left(\left(R+\eta\right)\cos\theta,\left(R+\eta\right)\sin\theta,z\right):\left(\theta,z\right)\in\omega\right\}. 
\end{equation}
The moving boundary $\Gamma^{\eta}(t)$ can be parametrized by considering  
\begin{equation*}
\phi_{\eta\left(t\right)}:\omega\mapsto\mathbb{R}^{3},\ \phi_{\eta\left(t\right)}\left(\theta,z\right):=\left(\left(R+\eta\right)\cos\theta,\left(R+\eta\right)\sin\theta,z\right).
\end{equation*}
This allows us to define the  tangent vectors $\boldsymbol{\tau}_{1}^{\eta},\boldsymbol{\tau}_{2}^{\eta}$ at any point $p=\phi_{\eta\left(t\right)}\left(\theta ,z\right)\in \Gamma ^{\eta}(t)$ with $(\theta, z)\in \omega$ by 
\begin{equation*}
\boldsymbol{\tau}_{1}^{\eta}\left(\theta,z\right):=\partial_{\theta}\phi_{\eta\left(t\right)}\left(\theta,z\right),\quad\boldsymbol{\tau}_{2}^{\eta}\left(\theta,z\right):=\partial_{z}\phi_{\eta\left(t\right)}\left(\theta,z\right).
\end{equation*}
The unit outer normal vector $\boldsymbol{\nu}^{\eta}$ at $p$ is given by 
\begin{equation*}
\boldsymbol{\nu}^{\eta}\left(\theta,z\right):=\frac{\partial_{\theta}\phi_{\eta\left(t\right)}\times\partial_{z}\phi_{\eta\left(t\right)}}{\left|\partial_{\theta}\phi_{\eta\left(t\right)}\times\partial_{z}\phi_{\eta\left(t\right)}\right|}\left(\theta,z\right).
\end{equation*}
Let us now introduce  the \emph{inflow} and \emph{outflow} parts of the boundary
\begin{equation*}
  \Gamma_{in/out}:=\left\{ \left(r\cos\theta,r\sin\theta,z\right):\left(r,\theta\right)\in\left(0,R\right)\times\left(0,2\pi\right),\ z\in\left\{ 0,L\right\} \right\}.
\end{equation*}
They will remain fixed, not changing during the evolution in time. Finally, let us introduce  the moving domains by
\begin{equation*}
\Omega^{\eta}\left(t\right):=\left\{ \left(r\cos\theta,r\sin\theta,z\right)\in\mathbb{R}^{3}:\left(\theta,z\right)\in\omega,0<r<R+\eta\left(t,\theta,z\right)\right\} 
\end{equation*}
which, at all times $t\in I$ will contain an incompressible and viscous fluid. The setting is sketched in Figure~\ref{fig:moving-domains}.
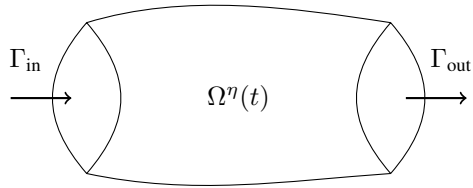
\begin{figure}[H]
    \centering
    \begin{tikzpicture}[scale=1]

\draw
  (0,-1) .. controls (0.6,-0.3) and (0.6,0.3) .. (0,1)
  .. controls (-0.6,0.3) and (-0.6,-0.3) .. (0,-1);

\draw
  (4,-1) .. controls (4.6,-0.3) and (4.6,0.3) .. (4,1)
  .. controls (3.4,0.3) and (3.4,-0.3) .. (4,-1);

\draw
  (0,1)
  .. controls (1.3,1.4) and (2.7,1.2) ..
  (4,1);

\draw
  (0,-1)
  .. controls (1.4,-1.3) and (2.6,-1.1) ..
  (4,-1);

\node at (2,0) {$\Omega^{\eta}(t)$};

\node at (-0.8,0.5) {$\Gamma_{\text{in}}$};
\node at (4.8,0.5) {$\Gamma_{\text{out}}$};

\draw[->, thick] (-1.0,0) -- (-0.2,0);

\draw[->, thick] (4.2,0) -- (5.0,0);

    \end{tikzpicture}

    \caption{The moving domains}
    \label{fig:moving-domains}
\end{figure}

\begin{remark}
    The assumption that the displacement of the shell is restricted to normal or radial direction can be relaxed to include displacements in all three directions, as long as the injectivity of $\phi_{\eta}$ can be guaranteed; see  \cite[p. 6639]{galic-canic-muha} for further details. 
\end{remark}
\subsection{Formulation of the problem}\label{ssec:problem}
Let us denote, here and throughout the work, the time-space domains by  
\[
I\times\Omega^{\eta}:=\bigcup_{t\in I}\left\{ t\right\} \times\Omega^{\eta}\left(t\right).
\]

Let $\mathbf{u}:I\times\Omega^{\eta}\mapsto\mathbb{R}^{3}$ be the velocity field of the fluid and $p:I\times \Omega^{\eta}\mapsto \mathbb{R}$ is the associated pressure field associated to $\mathbf{u}$. We assume the fluid to be homogeneous, of density $\rho_{f}>0$, incompressible, viscous with viscosity $\mu_{f}>0$ and obeying the \emph{Navier-Stokes equations}. This means that we have 

\begin{equation}\label{eqn:fluid}
  \left.\begin{array}{c}
\rho_{f}\left(\partial_{t}\mathbf{u}+\mathbf{u}\cdot\nabla\mathbf{u}\right)=\operatorname{div}\sigma\\
\text{div}\ \mathbf{u}=0
\end{array}\right\} \text{in}\ I\times \Omega^{\eta}
\end{equation}
with $\sigma$ denoting the usual \emph{Cauchy stress tensor} given by the formula 
\begin{equation}
\sigma=\sigma(x,\mathbf{u},p):=2\mu_{f}\frac{\nabla\mathbf{u}+\left(\nabla\mathbf{u}\right)^{T}}{2}-p\mathbb{I}_{3\times3}=:2\mu_{f}\mathbb{D}\mathbf{u}-p\mathbb{I}
\end{equation}
where $\mathbb{D}$ represents the symmetric gradient. 
Throughout this work we shall neglect the problem of the pressure, by dealing with test-functions with divergence zero. 
Once found a velocity field $\mathbf{u}$, the pressure can then be reconstructed by standard methods in the literature. 

The elastic shell is assumed to be a homogeneous medium of given density $\rho_{s}>0$ and of thickness $h>0$. 
We denote its displacement, measured with respect to the (fixed) lateral boundary $\Gamma$, by $\eta:I\times\omega \mapsto \mathbb{R}$ and we assume that it fulfills a Lam\'{e} type equation of the form 
\begin{equation}\label{eqn:lame}
\rho_{s}h\partial_{tt}\eta+\Delta^{2}\eta=f.
\end{equation}
Let us neglect $\rho_{s}$ and $h$ for the moment. The equation~\eqref{eqn:lame} comes as the Euler-Lagrange equation associated to the following energy:
\begin{equation}\label{eqn:energy-linear-shell}
E_{\eta}:=\frac{1}{2}\int_{\omega}\left|\partial_{t}\eta\right|^{2}\ dA+\frac{1}{2 }\int_{\omega}\left|\nabla^{2}\eta\right|^{2}\ dA-\int_{\omega}f\eta\ dA
\end{equation}
and the presence of the  second gradient of $\eta$ is related to a \emph{linearized} model of Koiter type shells, well known in the literature- see e.g. \cite{Ciarlet05} and the references therein. This is a simplified model and the presence of $\int_{\omega} \left\vert\nabla^{2}\eta\right\vert^{2}dA$ in the elastic energy $E_{\eta}$ already contains essential properties of a more general linear self-adjoint, coercive, second order differential operator (with coefficients that depend on the elastic material) denoted $\mathcal{L}_{e}:H_{0}^{2}\mapsto H^{-2}$ such that $\left\langle \mathcal{L}_{e}\left(\eta\right),\eta\right\rangle _{H^{-2}\to H^{2}}\ge c\left\Vert \eta\right\Vert_{H^2}^{2}$ for some $c>0$ and  all $\eta \in H_{0}^{2}(\omega)$ and used in many related works, see e.g. \cite{muha-canic-slip}.

\begin{remark}
Most of the paper, we focus on the linearized model derived from \eqref{eqn:energy-linear-shell} for which we present the existence result of Theorem~\ref{thm:main}. A more general version, called \emph{nonlinear Koiter model} is discussed in more details in Section~\ref{sec:nonlinear-koiter}. Therefore, in a second part of the paper, in Section~\ref{sec:nonlinear-koiter}, we also include the nonlinear model and explain several adjustments that need to be made to obtain Theorem~\ref{thm:main-nonlinear}.
\end{remark}
Regarding the forcing $f$, we assume the so-called \emph{dynamic coupling} boundary condition meaning that $f$ is given by the tension exerted by the fluid through $\sigma$ and evaluated in the direction $-\mathbf{e}_{r}$. This reads as 
\begin{equation}\label{eqn:force-f-shell}
f=-J_{\eta\left(t\right)}\left(x\right)\sigma\left(\phi_{\eta\left(t\right)}\left(x\right),\mathbf{u},p\right)\boldsymbol{\nu}^{\eta}\left(t,x\right)\cdot\mathbf{e}_{r},\quad  x\in\omega
\end{equation}
where we denote $J_{\eta\left(t\right)}:=\left|\partial_{\theta}\phi_{\eta\left(t\right)}\times\partial_{z}\phi_{\eta\left(t\right)}\right|$ the Jacobian of the transformation $\phi_{\eta}$.
By simple computations we obtain  
\begin{equation}\label{eqn:jacobian}
J_{\eta}=\sqrt{\left(R+\eta\right)^{2}\left(1+\left(\partial_{z}\eta\right)^{2}\right)+\left(\partial_{\theta}\eta\right)^{2}}
\end{equation}
and we observe that  $J_{\eta}>0$ if we ensure the condition $\min\limits_{x\in\omega}\left(R+\eta\left(x\right)\right)>0$. 

The problem of evolution after a possible contact (although clearly interesting and widely open) is beyond the aims of this paper, therefore we will focus on the following condition 
 \begin{equation}
     \left\Vert \eta\right\Vert _{L^{\infty}\left(I\times\omega\right)}\le M<R
 \end{equation}
for a certain $M>0$, which will be a consequence of the energy estimates and
which ensures that  $J_{\eta}$ and $J_{\eta}^{-1}$ remain bounded by a constant depending on $M$ and $\left|\nabla\eta\right|$.
Thus, between the area elements of $\Gamma^{\eta}$ and $\Gamma$ we have that
$dA_{\eta}=J_{\eta}dA$.


Let us now discuss the coupling conditions between the fluid and the membrane. As in the title of this work, we assume the so called \emph{Navier-slip boundary conditions}. They were introduced by Navier in \cite{navier1823memoire}. These conditions state that the normal component of the velocity and solid are equal on the moving interface $\Gamma^{\eta}$. Thus we have
\begin{equation}
\mathbf{u}\left(t,\phi_{\eta\left(t\right)}\left(x\right)\right)\cdot\boldsymbol{\nu}^{\eta\left(t\right)}\left(x\right)=\partial_{t}\eta\left(x\right)\mathbf{e}_{r}\left(\theta\right)\cdot\boldsymbol{\nu}^{\eta\left(t\right)}\left(x\right),\quad t\in I,\ x=\left(\theta,z\right)\in\omega.
\end{equation}
The second statement of the Navier's boundary condition states that the amount of slip in the tangential direction is proportional to the tangential part of the normal stress exerted by the fluid on the boundary. We express this in the form of
\begin{equation}
\left(\partial_{t}\eta\mathbf{e}_{r}-\mathbf{u}\left(t,\phi_{\eta\left(t\right)}\right)\right)\cdot\boldsymbol{\tau}_{i}^{\eta\left(t\right)}\left(x\right)=\alpha\cdot\sigma\left(\phi_{\eta\left(t\right)}\right)\boldsymbol{\nu}^{\eta\left(t\right)}\cdot\boldsymbol{\tau}_{i}^{\eta\left(t\right)}\left(x\right)\quad i=1,2
\end{equation}
where $\alpha>0$ is a parameter called \emph{slip length} and which will be fixed throughout this work. Usually $\alpha:\Gamma^{\eta}\left(t\right)\mapsto\left[0,\infty\right)$ with the case $\alpha \equiv 0$ corresponding to the case of \emph{no slip}. However, we will keep $\alpha$ constant throughout this work.
 
Finally, let us specify the boundary conditions at the inflow/outflow regions. We use the so-called \emph{dynamic boundary conditions} which involve the pressure $p$ and given by the formulas 

\begin{equation}
    \left.\begin{array}{c}
\mathbf{u}\cdot\mathbf{\boldsymbol{\tau}}_{i}=0,\ i=1,2\\
\frac{\rho_{f}}{2}\left|\mathbf{u}\right|^{2}+p=P_{in/out}\left(t\right)
\end{array}\right\} \ \text{on}\ \Gamma_{in/out},\ t\in I
\end{equation}
where $\boldsymbol{\tau}_{1,2}$ are the tangent vectors at $\Gamma_{in/out}$ which are represented by $\mathbf{e}_{r}$ (introduced before) and $\mathbf{e}_{\theta}=\left(-\sin\theta,\cos\theta,0\right)$ for  $(\theta, z)\in \omega$. Therefore on $\Gamma_{in/out}$ we have that $\mathbf{u}$ is nonzero only in the normal direction $\mathbf{e}_{z}=(0,0,1)$.

The \emph{forcing} of our system, leading its dynamics is given by prescribing the function $P_{in/out}$  and we denote
\begin{equation}
P=P(t):=P_{in/out}\left(t\right):\left(0, \infty \right)\mapsto\mathbb{R}
\end{equation}

Concerning the elastic shell, we assume that it is \emph{clamped}\footnote{This asumption can be relaxed to $\left|\eta\right|=\left|\partial_{z}\eta\right|=0$ as in \cite{muha-canic-slip}, for example.} at the endpoints, meaning that
\begin{equation}
\left|\eta\right|=\left|\nabla\eta\right|=0\quad\text{at}\ z\in\left\{ 0,L\right\} 
\end{equation}
We assume that at time $t=0$ we are provided with the quantities $\mathbf{u}_{0},\eta_{0},\eta_{1}$ satisfying the following compatibility conditions: 
\begin{equation}\label{eqn:initial-cond}
\begin{cases}
\begin{array}{c}
\mathbf{u}\left(0,\cdot\right)=\mathbf{u}_{0}\in L_{div}^{2}\left(\Omega^{\eta}\left(0\right)\right)\\
\mathbf{u}_{0}\cdot\boldsymbol{\tau}_{1,2}=0\ \text{on}\ \Gamma_{in/out}\\
\mathbf{u}_{0}\left(\phi_{\eta\left(0\right)}\right)\cdot\boldsymbol{\nu}^{\eta\left(0\right)}=\eta_{1}\mathbf{e}_{r}\cdot\boldsymbol{\nu}^{\eta\left(0\right)}\ \text{on}\ \omega\\
\eta\left(0,\cdot\right)=\eta_{0}\in H_{0}^{2}\left(\omega\right),\ \partial_{t}\eta\left(0,\cdot\right)=\eta_{1}\in L^{2}\left(\omega\right)
\end{array}\end{cases}
\end{equation}
We also assume that they have finite energy (see \eqref{eqn:energy-ineq}).

The space $L_{\text{div}}^{2}\left(\Omega^{\eta}\left(0\right)\right)$ is defined via
$L_{\text{div}}^{2}:=\overline{\left\{ \mathbf{v}\in C^{\infty}\left(\Omega^{\eta}\left(0\right);\mathbb{R}^{3}\right),\ \text{div}\ \mathbf{v}=0\right\} }^{\left\Vert \cdot\right\Vert _{2}}$.

\begin{remark}
For simplicity of the exposition we will refer to the $\texttt{data}$ of the problem given by the set 
\[\texttt{data}:=\left\{ R,L,\rho_{f},\mu_{f},\rho_{s},h,\alpha\right\}. \]
We fix $\rho_{f}=\rho_{s}=h=1,\mu_{f}=\frac{1}{2}$  and $\alpha=\alpha_{0}>0$ constant to simplify the computations.
\end{remark}

\paragraph{The fluid-structure interaction (FSI) problem.} 
In this work we aim to find a pair $\left(\mathbf{u},\eta\right)$ as described above which  solves the following system: 
\begin{equation}\label{eqn:FSI-system}
   \begin{cases}
\partial_{t}\mathbf{u}+\left(\mathbf{u}\cdot\nabla\right)\mathbf{u}=\operatorname{div}\sigma & \text{in}\ I\times\Omega^{\eta}\\
\text{div}\mathbf{u}=0 & \text{in}\ I\times\Omega^{\eta}\\
\partial_{tt}\eta+\Delta^{2}\eta=f & \text{on}\ I\times\omega\\
\left|\eta\right|=\left|\nabla\eta\right|=0 & \text{on}\ I\times\partial\omega\\
\frac{1}{2}\left|\mathbf{u}\right|^{2}+p=P(t):=P_{in/out}\left(t\right),\ \mathbf{u}\cdot\boldsymbol{\tau}_{1,2}=0 & \text{on}\ I\times\Gamma_{in/out}\\
\left(\partial_{t}\eta\mathbf{e}_{r}-\mathbf{u}\left(t,\phi_{\eta\left(t\right)}\right)\right)\cdot\boldsymbol{\nu}^{\eta\left(t\right)}=0 & \text{on}\ I\times\omega\\
\left(\partial_{t}\eta\mathbf{e}_{r}-\mathbf{u}\left(t,\phi_{\eta\left(t\right)}\right)-\alpha\ \sigma\left(\phi_{\eta\left(t\right)}\right)\boldsymbol{\nu}^{\eta}\left(t,x\right)\right)\cdot\boldsymbol{\tau}_{1,2}^{\eta}=0 & \text{on}\ I\times\omega\\
\mathbf{u}\left(0,\cdot\right)=\mathbf{u}_{0}, & \text{in \ensuremath{\Omega^{\eta_{0}}}}\\
\eta\left(0,\cdot\right)=\eta_{0},\ \partial_{t}\eta\left(0,\cdot\right)=\eta_{1} & \text{in}\ \omega
\end{cases}
 \tag{FSI}
\end{equation}

\section{Variational formulation and the main results}\label{sec:weak-form-main-res}
Wwe are interested in the study of the \emph{weak} or \emph{variational} solutions to \eqref{eqn:FSI-system}.
Let us assume that all the involved functions in the sequel are smooth.
We multiply the fluid equations \eqref{eqn:fluid} by a function $\mathbf{q}:I\times \Omega^{\eta} \mapsto \mathbb{R}^{3} $ and \eqref{eqn:lame} by $\xi:I\times \omega \mapsto \mathbb {R}$ between which we assume the following compatibility conditions:
\begin{equation}\label{eqn:compatible}
\begin{cases}
\mathbf{q}\left(t,\phi_{\eta}\right)\cdot\boldsymbol{\nu}^{\eta}=\xi\mathbf{e}_{r}\cdot\boldsymbol{\nu}^{\eta} & \text{on}\ I\times\omega\\
\text{div}\mathbf{q}=0 & \text{in}\ I\times\Omega^{\eta}\left(t\right)\\
\mathbf{q}\cdot\boldsymbol{\tau}_{1,2}=0 & \text{on}\ I\times\Gamma_{in/out}.
\end{cases}
\end{equation}

Using Reynolds' transport theorem, we obtain 
\begin{equation}
\begin{aligned}\int_{\Omega^{\eta}\left(t\right)}\partial_{t}\mathbf{u}\cdot\mathbf{q}dx= & \int_{\Omega^{\eta}\left(t\right)}\partial_{t}\left(\mathbf{u}\cdot\mathbf{q}\right)dx-\int_{\Omega^{\eta}\left(t\right)}\mathbf{u}\cdot\partial_{t}\mathbf{q}dx\\
= & \frac{d}{dt}\int_{\Omega^{\eta}\left(t\right)}\mathbf{u}\cdot\mathbf{q}dx-\int_{\Gamma^{\eta}\left(t\right)}\left(\mathbf{u}\cdot\mathbf{q}\right)\left(\partial_{t}\eta\mathbf{e}_{r}\circ\phi_{\eta\left(t\right)}^{-1}\right)\cdot\boldsymbol{\nu}^{\eta\left(t\right)}dA_{\eta\left(t\right)}+\\
 & -\int_{\Omega^{\eta}\left(t\right)}\mathbf{u}\cdot\partial_{t}\mathbf{q}dx.
\end{aligned}
\end{equation}
Concerning the convective term we write, using integration by parts
\begin{equation}
\begin{aligned}\int_{\Omega^{\eta}\left(t\right)}\left(\mathbf{u}\cdot\nabla\right)\mathbf{u}\cdot\mathbf{q}dx= & \frac{1}{2}\int_{\Omega^{\eta}\left(t\right)}\left(\mathbf{u}\cdot\nabla\right)\mathbf{u}\cdot\mathbf{q}dx+\frac{1}{2}\int_{\Omega^{\eta}\left(t\right)}\left(\mathbf{u}\cdot\nabla\right)\mathbf{u}\cdot\mathbf{q}dx\\
= & \frac{1}{2}\int_{\partial\Omega^{\eta}\left(t\right)}\left(\mathbf{u}\cdot\mathbf{q}\right)\left(\mathbf{u}\cdot\boldsymbol{\nu}^{\eta\left(t\right)}\right)dA_{\eta}+\frac{1}{2}\int_{\Omega^{\eta}\left(t\right)}\left(\mathbf{u}\cdot\nabla\right)\mathbf{u}\cdot\mathbf{q}dx\\
= & \frac{1}{2}\int_{\Gamma^{\eta}\left(t\right)}\left(\mathbf{u}\cdot\mathbf{q}\right)\left(\mathbf{u}\cdot\boldsymbol{\nu}^{\eta\left(t\right)}\right)+\frac{1}{2}\int_{\Gamma_{in/out}}\left(\mathbf{u}\cdot\mathbf{q}\right)\left(\mathbf{u}\cdot\boldsymbol{\nu}\right)+\\
 & \frac{1}{2}\int_{\Omega^{\eta}\left(t\right)}-\left(\mathbf{u}\cdot\nabla\right)\mathbf{q}\cdot\mathbf{u}dx+\frac{1}{2}\left(\mathbf{u}\cdot\nabla\right)\mathbf{u}\cdot\mathbf{q}dx\\
= & \frac{1}{2}\int_{\Gamma^{\eta}\left(t\right)}\left(\mathbf{u}\cdot\mathbf{q}\right)\left(\mathbf{u}\cdot\boldsymbol{\nu}^{\eta\left(t\right)}\right)dA_{\eta}+\int_{\Gamma_{in/out}}\frac{\left|\mathbf{u}\right|^{2}}{2}\mathbf{q}\cdot\boldsymbol{\nu}dA+b\left(t,\mathbf{u},\mathbf{u},\mathbf{q}\right)
\end{aligned}
\end{equation}
 where we have denoted 
\begin{equation}
b\left(t,\mathbf{u},\mathbf{v},\mathbf{w}\right):= \int_{\Omega^{\eta}\left(t\right)} \left(\frac{1}{2}\left(\mathbf{u}\cdot\nabla\right)\mathbf{v}\cdot\mathbf{w}- \frac{1}{2}\left(\mathbf{u}\cdot\nabla\right)\mathbf{w}\cdot\mathbf{v}\right) dx
\end{equation}
and used the fact that on $\Gamma_{in/out}$ we have that $\mathbf{q}=q_z\mathbf{e}_{z}$ and $\boldsymbol{\nu}=\pm\mathbf{e}_{z}$.
Then, we have
\begin{equation}
\begin{aligned}\int_{\Omega^{\eta}\left(t\right)}\operatorname{div}\sigma\cdot\mathbf{q}dx= & \int_{\partial\Omega^{\eta}\left(t\right)}\sigma\mathbf{q}\cdot\boldsymbol{\nu}^{\eta\left(t\right)}dx-\int_{\Omega^{\eta}\left(t\right)}\sigma:\mathbb{D}\mathbf{q}dx\\
= & \int_{\Gamma^{\eta}\left(t\right)}\sigma\mathbf{q}\cdot\boldsymbol{\nu}^{\eta\left(t\right)}dA_{\eta}+\int_{\Gamma_{in/out}}\sigma\mathbf{q}\cdot\boldsymbol{\nu}dA-\int_{\Omega^{\eta}\left(t\right)}\sigma:\mathbb{D}\mathbf{q}dx\\
= & \int_{\Gamma^{\eta}\left(t\right)}\sigma\mathbf{q}\cdot\boldsymbol{\nu}^{\eta\left(t\right)}dA_{\eta}+\int_{\Gamma_{in/out}}p\left(\mathbf{q}\cdot\boldsymbol{\nu}\right)dA-\int_{\Omega^{\eta}\left(t\right)}\sigma:\mathbb{D}\mathbf{q}dx,
\end{aligned}
\end{equation}
and further the first summand in the right hand side equals
\begin{equation}
\begin{aligned}\int_{\Gamma^{\eta}}\sigma\mathbf{q}\cdot\boldsymbol{\nu}^{\eta\left(t\right)}dA_{\eta}= & \int_{\Gamma^{\eta}}\sigma\boldsymbol{\nu}^{\eta\left(t\right)}\cdot\mathbf{\mathbf{q}}dA_{\eta}\\
= & \int_{\Gamma^{\eta}}\left(\sigma\boldsymbol{\nu}^{\eta\left(t\right)}\cdot\boldsymbol{\nu}^{\eta\left(t\right)}\right)\left(\mathbf{q}\cdot\boldsymbol{\nu}^{\eta\left(t\right)}\right)+\left(\sigma\boldsymbol{\nu}^{\eta\left(t\right)}\cdot\boldsymbol{\tau}_{1,2}^{\eta\left(t\right)}\right)\left(\mathbf{q}\cdot\boldsymbol{\tau}_{1,2}^{\eta\left(t\right)}\right)dA_{\eta}\\
= & \int_{\Gamma^{\eta}}\left(\sigma\boldsymbol{\nu}^{\eta\left(t\right)}\cdot\boldsymbol{\nu}^{\eta\left(t\right)}\right)\left(\mathbf{q}\cdot\boldsymbol{\nu}^{\eta\left(t\right)}\right)+\frac{1}{\alpha}\left(\partial_{t}\eta\mathbf{e}_{r}\circ\phi_{\eta\left(t\right)}^{-1}-\mathbf{u}\right)\cdot\boldsymbol{\tau}_{1,2}^{\eta\left(t\right)}\left(\mathbf{q}\cdot\boldsymbol{\tau}_{1,2}^{\eta\left(t\right)}\right)dA_{\eta}.
\end{aligned}
\end{equation}
Concerning the elastic shell, by multiplying \eqref{eqn:lame} by $\xi$ and using the definition of $f$ from \eqref{eqn:force-f-shell} we obtain that 
\begin{equation}
  \frac{d}{dt}\int_{\omega}\partial_{t}\eta\xi dA+\int_{\omega}-\partial_{t}\eta\partial_{t}\xi+\nabla^{2}\eta:\nabla^{2}\xi dA=-\int_{\Gamma^{\eta}\left(t\right)}\sigma\mathbf{q}\cdot\boldsymbol{\nu}^{\eta\left(t\right)}dA_{\eta}
\end{equation}
Adding all the computations we obtain formally that
\begin{equation}\label{eqn:formal-weak-formulation}
\begin{aligned}\frac{d}{dt}\int_{\Omega^{\eta}\left(t\right)}\mathbf{u}\cdot\mathbf{q}dx+\int_{\Omega^{\eta}\left(t\right)}-\mathbf{u}\cdot\partial_{t}\mathbf{q}+\mathbb{D}\mathbf{u}:\mathbb{D}\mathbf{q}dx+b\left(t,\mathbf{u},\mathbf{u},\mathbf{q}\right) & +\\
\int_{\Gamma^{\eta}\left(t\right)}-\frac{1}{2}\left(\mathbf{u}\cdot\mathbf{q}\right)\left(\partial_{t}\eta\mathbf{e}_{r}\circ\phi_{\eta\left(t\right)}^{-1}\right)\cdot\boldsymbol{\nu}^{\eta\left(t\right)}dA_{\eta} & +\\
\sum_{i=1}^{2}\int_{\omega}\frac{1}{\alpha}\left(\mathbf{u}\circ\phi_{\eta\left(t\right)}-\partial_{t}\eta\mathbf{e}_{r}\right)_{\tau_{i}^{\eta}}\left(\mathbf{q}\circ\phi_{\eta\left(t\right)}-\xi\mathbf{e}_{r}\right)_{\tau_{i}^{\eta}}J_{\eta\left(t\right)}dA & +\\
\frac{d}{dt}\int_{\omega}\partial_{t}\eta\cdot\xi dA+\int_{\omega}-\partial_{t}\eta\cdot\partial_{t}\xi+\nabla^{2}\eta:\nabla^{2}\xi dA & =\\
\left\langle F\left(t\right),\mathbf{q}\right\rangle 
\end{aligned}
\end{equation}
where we have used the notation $\boldsymbol{\xi}_{\tau_{i}^{\eta}}=\boldsymbol{\xi}\cdot\boldsymbol{\tau}_{i}^{\eta}$ for a generic\footnote{Such as $\mathbf{q}\circ\phi_{\eta\left(t\right)},\ \xi\mathbf{e}_{r}$}, sufficiently smooth function $\boldsymbol{\xi}:\omega\mapsto\mathbb{R}^{3}$  and also
\begin{equation}\label{eqn:notation-F-b}
\begin{aligned}
\left\langle F\left(t\right),\mathbf{q}\right\rangle := & P_{in}\left(t\right)\int_{\Gamma_{in}}\mathbf{q}\cdot\boldsymbol{\nu}dA-P_{out}\left(t\right)\int_{\Gamma_{out}}\mathbf{q}\cdot\boldsymbol{\nu}dA.
\end{aligned}
\end{equation}

\subsection{Formal a-priori estimates}

Let us use the basic admissible test function $(\mathbf{u},\partial_{t}\eta)$ in \eqref{eqn:formal-weak-formulation}  We get the  energy balance 
\begin{equation}\label{eqn:energy-balance}
   \frac{d}{dt}E\left(t\right)+E_{slip}\left(t\right)+D\left(t\right)=\left\langle F\left(t\right),\mathbf{u}\right\rangle \quad t\in I
\end{equation}

where
\begin{equation}\label{eqn:energy-E-D-Eslip}
\begin{aligned}E\left(t\right):= & \frac{1}{2}\int_{\Omega^{\eta}\left(t\right)}\left|\mathbf{u}\left(t,x\right)\right|^{2}dx+\frac{1}{2}\int_{\omega}\left|\partial_{t}\eta\right|^{2}dA+\frac{1}{2}\int_{\omega}\left|\nabla^{2}\eta\right|^{2}dA\\
D\left(t\right):= & \int_{\Omega^{\eta}\left(t\right)}\left|\mathbb{D}\mathbf{u}\right|^{2}dx\\
E_{slip}\left(t\right):= & \frac{1}{\alpha}\int_{\omega}\left|\mathbf{u}\circ\phi_{\eta\left(t\right)}-\partial_{t}\eta\mathbf{e}_{r}\right|^{2}J_{\eta\left(t\right)}dA.
\end{aligned}
\end{equation}
\begin{remark}\label{rmk:steady}
    Due to the fact that the shell is clamped, that is the condition $\left|\eta\right|=\left|\nabla\eta\right|=0$ at $\partial\omega$ it follows that when the $z$-coordinate is very close to $0$ or $L$ we have that the moving boundary $\Gamma ^{\eta}$ is steady in time since $\Gamma^{\eta}\left(t\right)=\left\{ \left(R+\eta\right)\mathbf{e}_{r}+z\mathbf{e}_{z}\right\} =\left\{ R\mathbf{e}_{r}+z\mathbf{e}_{z}\right\} =\Gamma$. Therefore denoting by $\Omega_{s}$ the region of $\Omega$ with $z$ very close to $0$ or $L$ we can see that $\Omega_{s}$ is Lipschitz regular and thus we can estimate 
    \begin{equation}
       \begin{aligned}\left\Vert \text{tr}_{\Gamma_{in/out}}\mathbf{u}\right\Vert _{L_{t}^{2}L_{x}^{2}\left(\Gamma_{in/out}\right)}\lesssim & \left\Vert \mathbf{u}\right\Vert _{L_{t}^{2}L_{x}^{2}\left(\Omega_{s}\right)}+\left\Vert \nabla\mathbf{u}\right\Vert _{L_{t}^{2}L_{x}^{2}\left(\Omega_{s}\right)}\\
\lesssim & \left\Vert \mathbf{u}\right\Vert _{L_{t}^{\infty}L_{x}^{2}\left(\Omega_{s}\right)}+\left\Vert \mathbb{D}\mathbf{u}\right\Vert _{L_{t}^{2}L_{x}^{2}\left(\Omega_{s}\right)}\\
\lesssim & \left(\sup_{t\in I}E\left(t\right)\right)^{1/2}+\left\Vert D\right\Vert _{L_{t}^{2}}
\end{aligned}
    \end{equation}
    where the last inequality is due to Korn's inequality (valid in Lipschitz domains, see \cite{acosta2006weighted} for more details).
    Alternatively, one could also argue using the trace operator from Lemma~\ref{lm:trace}, after extending $\eta$ by zero to $\mathbb{R}^{2}$.
\end{remark}

Now, using H\"{o}lder's inequality, Young's inequality and Remark~\ref{rmk:steady} we obtain  
\begin{equation}
\begin{aligned}E\left(t\right)+\int_{0}^{t}E_{slip}\left(s\right)+D\left(s\right)ds= & E\left(0\right)+\int_{0}^{t}\left\langle F\left(t\right),\mathbf{u}\right\rangle ds\\
\lesssim & E\left(0\right)+\int_{0}^{t}\left|P\left(s\right)\right|\left\Vert \mathbf{u}\left(s\right)\right\Vert _{L_{x}^{2}\left(\Gamma_{in/out}\right)}ds\\
\lesssim & E\left(0\right)+\varepsilon\int_{0}^{t}\left\Vert \mathbf{u}\right\Vert _{L_{x}^{2}\left(\Gamma_{in/out}\right)}^{2}+\frac{1}{4\varepsilon}\int_{0}^{t}P^{2}\left(s\right)ds\\
\lesssim & E\left(0\right)+\varepsilon\int_{0}^{t}\left\Vert \mathbf{u}\right\Vert _{L_{t}^{\infty}L_{x}^{2}\left(\Omega^{\eta}\right)}^{2}+\left\Vert \mathbb{D}\mathbf{u}\right\Vert _{L_{x}^{2}\left(\Omega^{\eta}\right)}^{2}ds+\frac{1}{4\varepsilon}\left\Vert P\right\Vert _{L_{t}^{2}}^{2}\\
\lesssim & E\left(0\right)+\varepsilon\left(\sup_{t\in I}E\left(t\right)+\int_{0}^{T}D\left(s\right)ds\right)+\frac{1}{4\varepsilon}\left\Vert P\right\Vert _{L_{t}^{2}}^{2}
\end{aligned}
\end{equation}
up to a constant depending on $\Omega$, for any $\varepsilon>0$. Since the inequality in the left-hand side is valid for all $t>0$, we may now choose $\varepsilon$ sufficiently small  to get  
\begin{equation}\label{eqn:energy-ineq}
\sup_{t\in \left(0,T\right)}E\left(t\right)+\int_{0}^{T}E_{slip}\left(s\right)+D\left(s\right)ds\lesssim E\left(0\right)+\left\Vert P\right\Vert _{L_{t}^{2}}^{2}.
\end{equation}
From \eqref{eqn:energy-ineq} it follows that, in general, $\eta$ is only H\"{o}lder continuous in both time and space with 
\begin{equation}\label{eqn:holder-eta}
    \eta\in C^{0,1-\alpha}\left(I;C^{0,2\alpha-1}\right),\quad\alpha\in\left(\frac{1}{2},1\right).
\end{equation}
So $\eta$ and  the moving interface $\Gamma^{\eta}$ are \emph{not Lipschitz continuous}, which creates some additional difficulties.

Now, we use Korn's inequality from Proposition~\ref{prop:korn} (with $p=2$). We obtain that for any $r<2$ the following holds:

\begin{equation}\label{eqn:energy}
\sup_{t\in\left(0,T\right)}E\left(t\right)+\int_{0}^{T}E_{slip}\left(t\right)dt+\int_{0}^{T}\int_{\Omega^{\eta}\left(t\right)}\left|\nabla\mathbf{u}\right|^{r}dxdt\lesssim E(0)+\left\Vert P\right\Vert _{L_{t}^{2}}^{2}<\infty 
\end{equation}
up to a constant depending of the \texttt{data}.

\subsection{Traces. Function spaces. Variational formulation.}
\paragraph{Trace operator.} Let us now introduce a notion of trace adapted to the moving domains $\Omega^{\eta}$, even though they are not Lipschitz.
Following \cite{LR14} we have the following
\begin{lemma}\label{lm:trace}
Let $\eta \in H^{2}(\omega)$,  $1<p<\infty$ and $1<r<p$.
Then there exists a linear and continuous \emph{trace operator} given by the formula 
\begin{equation}
\text{tr}_{\eta}:=\text{tr}_{\Gamma^{\eta}}:W^{1,p}\left(\Omega^{\eta}\left(t\right)\right)\mapsto W^{1-\frac{1}{r},r}\left(\omega\right),\quad \text{tr}_{\eta}:\mathbf{q}\mapsto\mathbf{q}\circ\phi_{\eta}
\end{equation}
whose continuity constant depends on $\Omega,r$ and an upper bound for $\left\Vert \eta\right\Vert _{H_{x}^{2}}$.
\end{lemma}

\paragraph{Function spaces.} Motivated by \eqref{eqn:energy}, we can introduce the Lebesgue function spaces of the type 
\[
L^{p}\left(I;L^{r}\left(\Omega^{\eta}\left(t\right)\right)\right):=\left\{ \mathbf{v}\in L^{1}\left(I\times\Omega^{\eta}\right):t\mapsto\left\Vert \mathbf{v}\left(t,\cdot\right)\right\Vert _{L^{r}\left(\Omega^{\eta}\left(t\right)\right)}\in L^{p}\left(I\right)\right\} 
\]

where we recall the notation $I\times\Omega^{\eta}:=\bigcup_{t\in I}\left\{ t\right\} \times\Omega^{\eta}\left(t\right).$
The Sobolev spaces are defined in an analogous manner. 

Due to the lack of Korn's inequality, we also need to consider the auxiliary spaces 
\[
E^{1,p}\left(\Omega^{\eta}\right):=\left\{ \mathbf{q}:\left\Vert \mathbf{q}\right\Vert _{L^{p}\left(\Omega^{\eta}\right)}+\left\Vert \mathbb{D}\mathbf{q}\right\Vert _{L^{p}\left(\Omega^{\eta}\right)}<\infty\right\} ,\quad p\in\left[1,\infty\right].
\]
\begin{remark}\label{rmk:loss}
    As already pointed out in \eqref{eqn:energy}, using Propositon~\ref{prop:korn} and the Sobolev embedding (see also \cite[Corollary 2.10]{lengeler2014weak}) we see that $E_{x}^{1,2}\hookrightarrow W_{x}^{1,2-}\hookrightarrow L_{x}^{6-}$
    where the notation "$2-$" means any $r<2$. This is why we will denote further norms of the type $\left\Vert \mathbf{u}\right\Vert _{L_{t}^{2}W^{1,2-}},\ \left\Vert \mathbf{u}\right\Vert _{L_{t}^{2}L^{6-}}$ etc. 
\end{remark}

Next, let us consider 
\[
H_{fl}^{\eta\left(t\right)}:=\left\{ \mathbf{u}\in E^{1,2}\left(\Omega^{\eta}\left(t\right)\right):\ \text{div}\ \mathbf{u}=0,\ \text{tr}_{\Gamma_{in/out}}\mathbf{u}\cdot\boldsymbol{\tau}_{1,2}=0\right\}
\]
and the energy spaces for fluid/shell respectively by 
\begin{equation}
    \begin{aligned}V_{fl}^{\eta}:= & L^{\infty}\left(I;L^{2}\left(\Omega^{\eta}\right)\right)\cap L^{2}\left(I;H_{fl}^{\eta}\right)\\
V_{s}:= & W^{1,\infty}\left(I;L^{2}\left(\omega\right)\right)\cap L^{2}\left(I;H_{0}^{2}\left(\omega\right)\right).
\end{aligned}
\end{equation} 
Note that by using the Aubin- Lions theorem we can obtain the compact embedding 
    \begin{equation}\label{eqn:comp-emd-AL}
        V_{s} \hookrightarrow \hookrightarrow L^{\infty}\left(I; W^{1,2}(\omega) \right).
    \end{equation}

We may now introduce the function spaces of \emph{solutions} and  \emph{test-functions} respectively via
\begin{equation}
\begin{aligned}\mathcal{S}^{\eta}:= & \left\{ \left(\mathbf{u},\eta\right)\in V_{fl}^{\eta}\times V_{s}:\left(\mathbf{u}\circ\phi_{\eta}-\partial_{t}\eta\mathbf{e}_{r}\right)\cdot\boldsymbol{\nu^{\eta}}=0\right\} \\
\mathcal{T}^{\eta}:= & \left\{ \left(\mathbf{q},\xi\right)\in\mathcal{S}^{\eta}:\ \partial_{t}\mathbf{q}\in L^{2}\left(I;L^{2}\left(\Omega^{\eta}\right)\right)\right\} .
\end{aligned}
\end{equation}
endowed with the norms 
\[
\left\Vert \left(\mathbf{u},\eta\right)\right\Vert _{\,\mathcal{S}^{\eta}}:=\left\Vert \mathbf{u}\right\Vert _{V_{fl}^{\eta}}+\left\Vert \eta\right\Vert _{V_{s}}
\]
and 
\[\left\Vert \left(\mathbf{q},\xi\right)\right\Vert _{\,\mathcal{T}^{\eta}}:=\left\Vert \left(\mathbf{q},\xi\right)\right\Vert _{\mathcal{S}^{\eta}}+\left\Vert \partial_{t}\mathbf{q}\right\Vert _{L^{2}\left(I;L^{2}\left(\Omega^{\eta}\right)\right)}\] respectively.
\begin{remark}
    The condition $\left(\mathbf{u}\circ\phi_{\eta}-\partial_{t}\eta\mathbf{e}_{r}\right)\cdot\boldsymbol{\nu}^{\eta}=0$ holds pointwise for almost every $(\theta,z)\in \omega$ and almost every $t\in I$. This is due to the low regularity of $\left|\nabla\eta\right|\in L^{\infty}\left(I;W^{1,2}\left(\omega\right)\right)$. 

\end{remark}
\paragraph{Variational formulation of the equations.}
\begin{definition}\label{def:weak-soln}
    We call a couple $\left(\mathbf{u},\eta\right)\in\mathcal{S}^{\eta}$ a \emph{weak-solution} for the FSI-problem \eqref{eqn:FSI-system} provided that
    \begin{enumerate}
        \item For almost every $t\in I=[0,T]$ and for all $\left(\mathbf{q},\xi\right)\in\mathcal{T}^{\eta}$, the following relation holds:
   
    \begin{equation}\label{eqn:weak-formulation}
\begin{aligned}\int_{\Omega^{\eta}\left(t\right)}\mathbf{u}\cdot\mathbf{q}\left(t\right)dx+\int_{0}^{t}\int_{\Omega^{\eta}\left(t\right)}-\mathbf{u}\cdot\partial_{t}\mathbf{q}+\mathbb{D}\mathbf{u}:\mathbb{D}\mathbf{q}dxds & +\\
\int_{0}^{t}b\left(s,\mathbf{u},\mathbf{u},\mathbf{q}\right)ds+\int_{0}^{t}\int_{\Gamma^{\eta}\left(s\right)}-\frac{1}{2}\left(\mathbf{u}\cdot\mathbf{q}\right)\left(\partial_{t}\eta\mathbf{e}_{r}\circ\phi_{\eta\left(t\right)}^{-1}\right)\cdot\boldsymbol{\nu}^{\eta\left(t\right)}dA_{\eta\left(s\right)}ds & +\\
\int_{0}^{t}\int_{\omega}\frac{1}{\alpha}\left(\mathbf{u}\circ\phi_{\eta}-\partial_{t}\eta\mathbf{e}_{r}\right)\cdot\left(\mathbf{q}\circ\phi_{\eta}-\xi\mathbf{e}_{r}\right)J_{\eta}dAds & +\\
\int_{\omega}\left(\partial_{t}\eta\right)\cdot\xi\left(t\right)dA+\int_{0}^{t}\int_{\omega}-\partial_{t}\eta\partial_{t}\xi+\nabla^{2}\eta:\nabla^{2}\xi dAds & =\\
\int_{I}\left\langle F\left(t\right),\mathbf{q}\right\rangle dt+\int_{\Omega^{\eta}\left(0\right)}\mathbf{u}_{0}\cdot\mathbf{q}\left(0\right)dx+\int_{\omega}\eta_{1}\xi dA.
\end{aligned}
\end{equation}
Recall that $J_{\eta}$ was defined in \eqref{eqn:jacobian} and $F$ and $b$ were introduced in \eqref{eqn:notation-F-b}.
\item The energy balance from \eqref{eqn:energy} holds, namely
\begin{equation}
\sup_{t\in\left(0,T\right)}E\left(t\right)+\int_{0}^{T}E_{slip}\left(t\right)dt+\int_{0}^{T}\int_{\Omega^{\eta}\left(t\right)}\left|\nabla\mathbf{u}\right|^{r}dxdt\lesssim E(0)+\left\Vert P\right\Vert _{L_{t}^{2}}^{2}<\infty 
\end{equation}
for any $r<2$. The continuity constant depends on \texttt{data}.
  \end{enumerate}
\end{definition}

\begin{remark} In the derivation of \eqref{eqn:weak-formulation} we used the elementary identity
\[
\sum_{i=1}^{2}\left(\mathbf{u}\circ\phi_{\eta}-\partial_{t}\eta\mathbf{e}_{r}\right)_{\tau_{i}^{\eta}}\left(\mathbf{q}\circ\phi_{\eta}-\xi\mathbf{e}_{r}\right)_{\tau_{i}^{\eta}}=\left(\mathbf{u}\circ\phi_{\eta}-\partial_{t}\eta\mathbf{e}_{r}\right)\cdot\left(\mathbf{q}\circ\phi_{\eta}-\xi\mathbf{e}_{r}\right).
\]
\end{remark}
\subsection{Main results}\label{ssec:main-results}
The main result of this work is the following 
\begin{theorem}\label{thm:main}
Let $P=P(t)\in L^2(0,\infty)$ be an inflow/outflow prescribed source, and let the initial data $\left(\mathbf{u}_{0},\eta_{0},\eta_{1}\right)$ fulfill the conditions \eqref{eqn:initial-cond}. Then there exists a  positive time $T_{\star}=T_{\star}\left(\mathbf{u}_{0},\eta_{0},\eta_{1}\right)$ with $0<T_{\star}\le \infty$ and such that, in any  time interval $I=[0,T]$ with $T<T_{\star}$ the problem \eqref{eqn:FSI-system} admits at least one weak solution in the sense of Definition~\ref{def:weak-soln}. In case no self-contact of the shell occurs, we can take $T_{\star}=\infty$, thus obtaining the global existence of weak solutions.\footnote{With the same proof we can obtain the same result adapted to the case of a 2-dimensional fluid and 1-dimensional elastic beam, recovering the result of \cite{muha-canic-slip}.}
\end{theorem}

The proof of Theorem~\ref{thm:main} will be carried out in several steps, in Section~\ref{sec:proof-main}.

Once this is achieved, we discuss the possibility to include the \emph{nonlinear Koiter} elasticity model as introduced in \cite{MS22}. 
We refer to Section~\ref{sec:nonlinear-koiter} for the precise formulation of the nonlinear Koiter model, which provides a new, nonlinear FSI problem, \eqref{eqn:nonl-system-fsi} -- which is the counterpart of \eqref{eqn:FSI-system}.
In this case, we are able to obtain the analogue of Theorem~\ref{thm:main}. This means:   assume  the same hypothesis as in Theorem~\ref{thm:main} and replace the linear elastic model by the nonlinear Koiter model presented in Section~\ref{sec:nonlinear-koiter}. This results in the new problem \eqref{eqn:nonl-system-fsi}. 
  Then, the same conclusion holds. 

\begin{theorem}\label{thm:main-nonlinear}

 There exists a  positive time $T_{\star}=T_{\star}\left(\mathbf{u}_{0},\eta_{0},\eta_{1}\right)$ with $0<T_{\star}\le \infty$ such that, in any  time interval $I=[0,T]$ with $T<T_{\star}$ the problem \eqref{eqn:nonl-system-fsi} admits at least one weak solution in the sense of Definition~\ref{def:weak-soln}. In case no self-contact of the shell occurs we can take $T_{\star}=+\infty$.
\end{theorem}

We refer to Section~\ref{sec:nonlinear-koiter} for additional setting and the proof.

\section{Auxiliary tools}\label{sec:tools}
We collect here some results needed in the proof of Theorem~\ref{thm:main}.
\subsection{Korn's inequality in moving H\"{o}lder domains}

One important difference between the \emph{no-slip} and the  \emph{Navier's slip} in moving domains condition concerns the Korn's (in)equality which plays a prominent role in \eqref{eqn:weak-formulation}. 

Let us recall \cite[Lemma A.5]{LR14} which proves that in the case of \emph{no-slip} one can even obtain a \emph{Korn identity}. For all smooth $\mathbf{q}$ which have the form $\text{tr}_{\eta}\mathbf{q}=\xi\mathbf{e}_{r}$ for a $\xi:\omega\mapsto\mathbb{R}$, it holds that
\begin{equation}
\int_{\Omega^{\eta}\left(t\right)}\mathbb{D}\mathbf{u}:\mathbb{D}\mathbf{q}dx=2\int_{\Omega^{\eta}\left(t\right)}\nabla\mathbf{u}:\nabla\mathbf{q}dx.
\end{equation}

This is not the case in the present work. Usual Korn inequalities are formulated in steady and Lipschitz domains, see e.g. \cite{velvcic2012nonlinear} and \cite[Lemma 1]{muha-canic-slip}.
Our domains $\Omega^{\eta}$ are merely of type $C^{0,\alpha}_{x}$ for any $\alpha \in [0,1)$. 

To circumvent this, let us recall the following seminal result of \cite[Theorem 3.1]{acosta2006weighted}:

\begin{theorem}\label{thm:Duran}
Let $\Omega$ be an open, bounded domain, of $\alpha$-H\"{o}lder regularity with $\alpha \in (0,1]$ and $1<p<\infty$. Then it holds that 
\begin{equation}
  \left\Vert d^{1-\alpha}\nabla\mathbf{u}\right\Vert _{L^{p}\left(\Omega\right)}\le C\left(\Omega,p\right)\left(\left\Vert \mathbb{D}\mathbf{u}\right\Vert _{L^{p}\left(\Omega\right)}+\left\Vert \mathbf{u}\right\Vert _{L^{p}\left(\Omega\right)}\right)
\end{equation}
where $d:\Omega\mapsto\mathbb{R}$ represents the distance function from a point $x\in \Omega$ to $\partial\Omega$.
\end{theorem}
Now with the aid of Theorem~\ref{thm:Duran}
one can obtain the result of \cite[Proposition 2.9]{lengeler2014weak}. For the sake of completeness, we include it here:
\begin{proposition}\label{prop:korn}
    Let $1<p<\infty$ and $1\le r<p$. For any $t\in I$ and for all the smooth functions $\mathbf{q}\in C^{\infty}\left(\Omega^{\eta}\left(t\right)\right)$ it holds that 
    \begin{equation}
        \left\Vert \nabla\mathbf{q}\right\Vert _{L_{x}^{r}}\lesssim\left\Vert \mathbb{D}\mathbf{q}\right\Vert _{L_{x}^{p}}+\left\Vert \mathbf{q}\right\Vert _{L_{x}^{p}}.
    \end{equation}
    The continuity constant depends on $\Omega,p,r$.
\end{proposition}
\begin{proof}
    Since according to \eqref{eqn:holder-eta} we have that $\partial\Omega^{\eta}\left(t\right)\in C^{0,\beta}$ for \emph{all} $\beta \in (0,1)$ we can use Theorem~\ref{thm:Duran} and H\"{o}lder's inequality for $s$ with $\frac{1}{s}=\frac{1}{r}-\frac{1}{p}$ to get that 
    $\left\Vert \nabla\mathbf{q}\right\Vert _{L_{x}^{r}}\le\left\Vert d^{1-\beta}\nabla\mathbf{q}\right\Vert _{L_{x}^{p}}\left\Vert d^{\beta-1}\right\Vert _{L_{x}^{q}}$. Now we only need to bound $\left\Vert d^{\beta-1}\right\Vert _{L_{x}^{s}}$. For this let us assume that $\left\Vert \eta\right\Vert _{L_{x}^{\infty}}\le R/2$ and we write 
    \[
    \int_{\Omega^{\eta}\left(t\right)}d^{\left(\beta-1\right)s}=\int_{\left\{ x\in\Omega^{\eta}\left(t\right):d\left(x\right)<R\right\} }d^{\left(\beta-1\right)s}+\int_{\left\{ x\in\Omega^{\eta}\left(t\right):d\left(x\right)\ge R\right\} }d^{\left(\beta-1\right)s}.
    \]
The upper bound of the second integral is obvious, while for the first one we have 
\[
\int_{0}^{R}\int_{\left\{ x\in\Omega^{\eta}\left(t\right):d\left(x\right)=\theta\right\} }d^{\left(\beta-1\right)s}dxd\theta\lesssim\int_{0}^{R}\theta^{\left(\beta-1\right)s}d\theta<\infty
\]
    provided that, $\left(\beta-1\right)s+1>0$ or that $s<\frac{1}{1-\beta}$ which can be ensured.
\end{proof}

\subsection{The Piola transform}
Let us now introduce the well known \emph{Piola transform}, denoted by $\mathcal{J}_{\eta}$. It allows us to extend functions from $f:\Omega\mapsto\mathbb{R}^{3}$ to $\mathcal{J}_{\eta}f:\Omega^{\eta}\mapsto\mathbb{R}^{3}$ such that $\text{div}\ f=0$ implies $\text{div}\ \mathcal{J}_{\eta}f=0$. Moreover, corresponding zero boundary conditions on the trace are also preserved, we present the details below.

Let us consider an arbitrary displacement, say  $\eta \in C^{2}(\omega)$ with 
\begin{equation}\label{eqn:cond-eta-M}
    \left\Vert \eta\right\Vert _{L_{x}^{\infty}}\le M <R
\end{equation}
and extend it in radial direction by setting 
$\tilde{\eta}=\rho\left(r\right)\eta$
where $\rho\in C^{\infty}\left(0,R\right)$ is a fixed function such that 
\begin{equation}\label{eqn:cond-rho}
\rho=\rho_{a,b}\left(r\right)=\begin{cases}
0 & r\in\left(0,a\right)\\
1 & r\in\left(R-b,R\right)
\end{cases}, \quad 0\le\rho^{\prime}<\frac{1}{M}
\end{equation}
for some $a,b>0$ with $a<R-b$.
Note that  although all the further estimates will depend on $\rho$, we will not mention this dependence once $\rho$ is fixed.

This allows us now consider the following mapping from the reference configuration $\Omega$ to the deformed configuration $\Omega^{\eta}=\Omega^{\eta}\left(t\right)$ defined via
\[
\psi_{\eta}:\Omega\mapsto\Omega^{\eta}\left(t\right),\quad\psi_{\eta}\left(r,\theta,z\right)=\left(r+\tilde{\eta}\right)\mathbf{e}_{r}+z\mathbf{e}_{z}, \quad r\in\left(0,R\right),\ \left(\theta,z\right)\in\omega.
\]
One can check easily that $\psi_{\eta}$ is a homeomorphism and even a diffeomorphism.

Then, we can define the \emph{Piola transform}, denoted by $\mathcal{J}_{\eta}$, which assigns to each $\phi:\Omega\mapsto\mathbb{R}^{3}$  the mapping
\begin{equation}\label{eqn:Piola}
\mathcal{J}_{\eta}\phi:=\left(\frac{\nabla\psi_{\eta}}{\det\nabla\psi_{\eta}}\phi\right)\circ\psi_{\eta}^{-1}:\Omega^{\eta}\mapsto\mathbb{R}^{3}
\end{equation}

For each $1\le p \le \infty$ and $1\le r<p$ we have that 
$\mathcal{J}_{\eta}:W^{1,p}\left(\Omega\right)\mapsto W^{1,r}\left(\Omega^{\eta}\right)$
is linear, continuous  and one-to-one.
\begin{lemma}\label{lm:Piola-est}
    We have the following pointwise estimates
    \begin{equation}
\begin{aligned}\left|\mathcal{J}_{\eta}\phi\right|\lesssim & \left(1+\left|\eta\right|+\left|\nabla\eta\right|\right)\left|\phi\right|\\
\left|\nabla\mathcal{J}_{\eta}\phi\right|\lesssim & \left(1+\left|\eta\right|+\left|\nabla\eta\right|\right)^{3}\left|\phi\right|+\left|\nabla^{2}\eta\right|\left|\phi\right|+\left(1+\left|\eta\right|+\left|\nabla\eta\right|\right)\left|\nabla\phi\right|
\end{aligned}
    \end{equation}
    with a constant depending on $\Omega,M$.
    
    In particular, the Piola mapping is a continuous operator from $L^{p}\left(\Omega\right)$ to $L^{r}\left(\Omega^{\eta}\right)$ and from $W^{1,p}\left(\Omega\right)$ to $W^{1,r}\left(\Omega^{\eta}\right)$ for any $1<p<\infty $ and $1\le r<p$. The loss of regularity comes from the fact that $\nabla\eta\notin L_{x}^{\infty}$. 
\end{lemma}
\begin{proof}
    Elementary computations and triangle's inequality. 
\end{proof}

It is obvious that the Piola mapping preserves the zero trace of functions $\phi$ as above. Let us now prove that the Piola mapping preserves the zero boundary values in the normal direction only. More precisely, we have 
\begin{lemma}\label{lm:Piola}
For each 
$\phi\in C^{\infty}\left(\Omega;\mathbb{R}^{3}\right)$ with $\left.\phi\right|_{r=R}\cdot\mathbf{e}_{r}=0$ we have that $\left.\mathcal{J}_{\eta}\phi\right|_{r=R+\eta}\cdot\boldsymbol{\nu}^{\eta}=0$. 
\end{lemma}
\begin{proof}
Let us recall the unit tangent vectors and unit outer normal at $\Gamma$ by 
\begin{equation}\label{eqn:vectors-tan-normal}
\mathbf{e}_{r}=\left(\cos\theta,\sin\theta,0\right),\ \mathbf{e}_{\theta}=\left(-\sin\theta,\cos\theta,0\right),\mathbf{e}_{z}=\left(0,0,1\right)
\end{equation}

We compute now 
\[
\nabla\psi_{\eta}=\left[\begin{array}{ccc}
1+\partial_{r}\tilde{\eta} & \partial_{\theta}\tilde{\eta} & \partial_{z}\tilde{\eta}\\
0 & r+\tilde{\eta} & 0\\
0 & 0 & 1
\end{array}\right],\quad\det\nabla\psi_{\eta}=\left(1+\partial_{r}\tilde{\eta}\right)\left(r+\tilde{\eta}\right)
\]
with $\det \nabla \psi_{\eta}>0$ due to conditions \eqref{eqn:cond-eta-M} and 
\eqref{eqn:cond-rho}.

Next, on $\Gamma^{\eta}$ that is at $r=R+\eta$ we have that the tangent vectors $\boldsymbol{\tau}^{\eta}_{1,2}$ and the unit outer normal vector  $\boldsymbol{\nu}^{\eta}$ are given by:
\[
\boldsymbol{\tau}_{1}^{\eta}:=\partial_{\theta}\eta\mathbf{e}_{r}+\left(R+\eta\right)\mathbf{e}_{\theta},\quad  \boldsymbol{\tau}_{2}^{\eta}:=\partial_{z}\eta\mathbf{e}_{r}+\mathbf{e}_{z}
\]
with the outer normal direction vector 
\[
\mathbf{n}\left(\eta\right):=\boldsymbol{\tau}_{1}^{\eta}\times\boldsymbol{\tau}_{2}^{\eta}=\left(R+\eta\right)\mathbf{e}_{r}-\partial_{\theta}\eta\mathbf{e}_{\theta}-\partial_{z}\eta\left(R+\eta\right)\mathbf{e}_{z}.
\]

We obtain, using the fact that when $r$ is close to $R$ we have $\tilde{\eta }=\eta$, that:
\begin{equation}\label{eqn:Piola-zero-normal}
\left[\nabla\psi_{\eta}\right]\phi\left(R\right)\cdot\mathbf{n}\left(\eta\right)=\left[\begin{array}{c}
\left(1+\partial_{r}\tilde{\eta}\right)\phi_{r}+\partial_{\theta}\tilde{\eta}\phi_{\theta}+\partial_{z}\tilde{\eta}\phi_{z}\\
\left(R+\tilde{\eta}\right)\phi_{\theta}\\
\phi_{z}
\end{array}\right]^{T}\cdot\left[\begin{array}{c}
R+\eta\\
-\partial_{\theta}\eta\\
-\partial_{z}\eta\left(R+\eta\right)
\end{array}\right]=(R+\eta)\phi_{r}
\end{equation}
and the conclusion follows. The Piola transform \emph{preserves  the zero boundary values in the normal direction}.
\end{proof}

In the rest of the section, we address the problem of extending test functions $\xi: I\times \omega \mapsto \mathbb{R}$ for the shell equation. More precisely, we would like to pair such a $\xi$ with a function defined on $\Omega ^{\eta}(t)$ with values in $\mathbb{R}^{3}$ which has  divergence zero and the correct boundary conditions. This task is not trivial, particularly if the extension operator needs to enjoy also suitable estimates. For this, we introduce the term \emph{extension operator}.
For example,  in \cite{mindrila-roy-multuilayered} 
we constructed such an extension operator $\mathcal{F}_{\eta}$,  adapted to the \emph{no-slip} case.
We will recall the construction below.
The pair $\left(\mathcal{F}_{\eta}\left(\xi\right),\xi\right)$ belongs to $\mathcal{T}^{\eta}$.
We refer to Subsection~\ref{ssec:extension-no-slip}. Note that the resulting test functions will make the slip term vanish.

To overcome this,  we construct a new extension operator $\mathcal{F}^{s}_{\eta}$, adapted to the  Navier slip. That is, in  needs to fulfill $\left.\mathcal{F}_{\eta}^{s}\left(\xi\right)\right|_{r=R+\eta}\cdot\boldsymbol{\nu}^{\eta}=\xi\mathbf{e}_{r}\cdot\boldsymbol{\nu}^{\eta}$ and $\left.\mathcal{F}_{\eta}^{s}\left(\xi\right)\right|_{r=R+\eta}\neq\xi\mathbf{e}_{r}$, to keep the slip term nonzero.
We present the details in Subsection~\ref{ssec:extension}.
\subsection{An extension operator adapted to the no-slip condition}\label{ssec:extension-no-slip}

We recall the extension operator adapted to the no-slip context from 
\cite[Proposition 3.1]{mindrila-roy-multuilayered}.
We can simplify the construction and re-define the operator $\mathcal{F}_{\delta}$, for a sufficiently smooth $\delta:I\times \omega \mapsto \mathbb{R}$ such that $\left\Vert \delta\right\Vert _{L_{t,x}^{\infty}}\le M<R$.
Here $\delta$ is a generic function prescribing some moving domain $\Omega^{\delta}$.

Consider first a smooth function
\begin{equation}\label{eqn:alpha-1}
    \alpha_{1}\in C^{\infty}\left(0,R+M\right),\ \alpha_{1}=\begin{cases}
1 & r\in\left[R-M,R+M\right]\\
0 & r\in\left[0,\frac{R-M}{2}\right]\\
C^{\infty} & r\in\left(\frac{R-M}{2},R-M\right)
\end{cases}
\end{equation}
and set, for any $\xi \in V_{s}$ (admissible test function for the shell equation) 
\begin{equation}\label{eqn:defn-ext-op-noslip}
\mathcal{F}_{\delta}\left(\xi\right):=\frac{R+\delta}{r}\alpha_{1}\left(r\right)\xi\mathbf{e}_{r}+\left(1-\alpha_{1}\left(r\right)\right)\beta\mathbf{e}_{\theta},\quad\beta\left(r,\theta,z\right):=-\frac{\alpha_{1}^{\prime}\left(r\right)}{1-\alpha_{1}\left(r\right)}\int_{0}^{\theta}\left(R+\delta\right)\xi d\tilde{\theta}
\end{equation}
We have 
\begin{proposition}\label{prop:extension-no-slip}
Let $\delta\in C^{\infty}\left(I;C^{\infty}\left(\omega\right)\right)$ be any prescribed motion of a moving domains $\Omega^{\delta}\left(t\right)$ with $t\in I$. The operator $\mathcal{F}_{\delta}$ from \eqref{eqn:defn-ext-op-noslip} maps $\mathcal{F}_{\delta}:C_{0}^{\infty}\left(\omega\right)\mapsto C_{\text{div}}^{\infty}\left(\Omega^{\delta}\right)$  and is a linear extension operator (of test functions defined on $\omega$ to $\Omega^{\delta}$)  such that $\text{tr}_{\Gamma^{\delta}}\mathcal{F}_{\delta}\left(\xi\right)=\xi\mathbf{e}_{r}$ and $\mathcal{F}_{\delta}\left(\xi\right) \times \mathbf{e}_{z}=\mathbf{0}$ on $\Gamma_{in/out}$.

Furthermore, for all $p,q\in\left[1,\infty\right]$ and $k\in\left[0,1\right]$ we have that 
\begin{equation}\label{eqn:cons-ext-1}
\left\Vert \mathcal{F}_{\delta}\left(\xi\right)\right\Vert _{L_{t}^{p}W_{x}^{k,q}}\lesssim\left\Vert \delta\xi\right\Vert _{L_{t}^{p}W_{x}^{k,q}}
\end{equation}
and  
\begin{equation}\label{eqn:cons-ext-2}
\left\Vert \partial_{t}\mathcal{F}_{\delta}\left(\xi\right)\right\Vert _{L_{t}^{p}L_{x}^{q}}\lesssim\left\Vert \partial_{t}\left(\delta\xi\right)\right\Vert _{L_{t}^{p}L_{x}^{q}}.
\end{equation}
The continuity constants depend only on $\Omega$.
\end{proposition}
\begin{proof}
    The estimates can be checked straightforwardly. Regarding the boundary values, since \[\mathcal{F}_{\delta}\left(\xi\right):=\frac{R+\delta}{r}\alpha_{1}\left(r\right)\xi\mathbf{e}_{r}+\left(-\alpha_{1}^{\prime}\left(r\right)\right)\int_{0}^{\theta}\left(R+\delta\right)\xi d\tilde{\theta}\mathbf{e}_{\theta}\]
we see that due to \eqref{eqn:alpha-1} and $\xi\left(\theta, z \right)=0$ for $z \in \left \{0, L\right\}$ we have
\[\left.\mathcal{F}_{\delta}\left(\xi\right)\right|_{r=R+\delta}=\xi\mathbf{e}_{r},\quad\left.\mathcal{F}_{\delta}\left(\xi\right)\right|_{z\in\left\{ 0,L\right\} }=\mathbf{0}.
\]
Using the formula for divergence in cylindrical coordinates 
\[
\text{div}\ \mathcal{F}_{\delta}\left(\xi\right)=\frac{1}{r}\partial_{r}\left(r\cdot\frac{R+\delta}{r}\alpha_{1}\left(r\right)\xi\right)+\frac{1}{r}\partial_{\theta}\left(\left(-\alpha_{1}^{\prime}\left(r\right)\right)\int_{0}^{\theta}\left(R+\delta\right)\xi d\tilde{\theta}\right)=0.
\]
\end{proof}
We will use this operator in Subsection~\ref{ssec:compactness}.

\subsection{An extension operator adapted to the Navier slip condition}\label{ssec:extension}
For the case of \emph{Navier-slip} the operator considered in Proposition~\ref{prop:extension-no-slip} contains in fact too much information. In the sense that the slip term in \eqref{eqn:weak-formulation} vanishes, and this  makes it not suitable for the construction of a solution, for example as in Subsection~\ref{ssec:construction-weak}. Therefore, we need to construct a new extension operator $\mathcal{F}_{\eta}^{s}\left(\xi\right)$ that extends an arbitrary function $\xi\in H^{2}\left(\omega\right)$ from $\omega$ to $\Omega^{\eta}$ with the property that $\left(\mathcal{F}_{\eta}^{s}\left(\xi\right),\xi\right)\in\mathcal{T}^{\eta}$ for any $\xi$ such that $(\mathbf{q},\xi)\in\mathcal{T}^{\eta}$ for some suitable $\mathbf{q}$.

We emphasize that the $s$ in the definition of $\mathcal{F}_{\eta}^{s}$ is by no means related to raising $\mathcal{F}_{\eta}$ to the power $s$; the $s$ comes from "slip".

First, let introduce the \emph{inverse Piola transform} of a function $\varphi:\Omega^{\eta}\mapsto\mathbb{R}^{3}$ by the formula
\begin{equation}\label{eqn:piola-inverse}
\mathcal{J}_{\eta}^{-1}\varphi:=\left(\det\nabla\psi_{\eta}\left(\nabla\psi_{\eta}\right)^{-1}\varphi\right)\circ\psi_{\eta}:\Omega\mapsto\mathbb{R}^{3}
\end{equation}
which enjoys the same (adapted) properties of $\mathcal{J}_{\eta}$ from Lemma~\ref{lm:Piola-est}.

We  assume that
$\left\Vert \eta\right\Vert _{L_{x}^{\infty}}\le M<R$.

Observing now that the Piola transform preserves the zero boundary values in normal direction, we would like to find first a function $\mathbf{q}:\Omega\mapsto \mathbb{R}^{3}$ such that
$\left.\left(\mathbf{q}-\mathcal{J}_{\eta}^{-1}\left(\xi\mathbf{e}_{r}\right)\right)\right|_{r=R}\cdot\mathbf{e}_{r}=0$, using \eqref{eqn:piola-inverse}.
The function $\xi\mathbf{e}_{r}$ is defined for all $r \in (0, R+\eta)$ and constant in the $r$ variable. Now, we want to extend $\mathbf{q}$ from $r=R$ to all $r\in (0,R)$ in a divergence free way. For this, since
$\mathcal{J}_{\eta}^{-1}\left(\xi\mathbf{e}_{r}\right)\cdot\mathbf{e}_{r}=\left(r+\tilde{\eta}\right)\xi$, 
 we may define
\[
\mathbf{q}=\mathbf{q}\left(\xi\right):=\left(r+\tilde{\eta}\right)\xi\mathbf{e}_{r}+\beta\mathbf{e}_{\theta}\]
with $\beta$ chosen such that
\[\beta=\beta_{\eta}\left(r,\theta,z\right)=-\partial_{r}\int_{0}^{\theta}r\left(r+\tilde{\eta}\left(r,s,z\right)\right)\xi\left(s,z\right)ds.\]
We compute
 \[
\text{div}\ \mathbf{q}=\frac{1}{r}\partial_{r}\left(r\left(r+\tilde{\eta}\right)\xi\right)+\frac{1}{r}\partial_{\theta}\beta=0\] due to the particular choice of $\beta$.

Now, since $\mathbf{q}:\Omega\mapsto\mathbb{R}$ we need to define it on $\Omega^{\eta}$ preserving its essential properties. Consequently we use the Piola transform $\mathcal{J}_{\eta}$ (from Lemma~\ref{lm:Piola}) and
we define 
\begin{equation}
\begin{aligned}\mathcal{F}_{\eta}^{s}\left(\xi\right): & =\mathcal{J}_{\eta}\left(\mathbf{q}\left(\xi\right)\right)=\frac{1}{\left(1+\partial_{r}\tilde{\eta}\right)\left(r+\tilde{\eta}\right)}\left[\begin{array}{c}
\left(1+\partial_{r}\tilde{\eta}\right)\left(r+\tilde{\eta}\right)\xi+\left(\partial_{\theta}\tilde{\eta}\right)\beta\\
\left(r+\tilde{\eta}\right)\beta\\
0
\end{array}\right]^{T}.
\end{aligned}
\end{equation}
Using Lemma~\ref{lm:Piola} we have that 
\begin{equation}
\text{div}\mathcal{F}^{s}_{\eta}\left(\xi\right)=0,\quad\text{tr}_{\Gamma^{\eta}}\mathcal{F}_{\eta}\left(\xi\right)\cdot\boldsymbol{\nu}^{\eta}=\xi\mathbf{e}_{r}\cdot\boldsymbol{\nu}^{\eta}.
\end{equation}
Regarding the boundary condition at $z\in\left\{ 0,L\right\}$, since  $\xi(z)=0$ we get $\beta=0$ and thus $\mathcal{F}^{s}_{\eta}\left(\xi\right)=0$.

\paragraph{Estimates}
Let $\eta\in L^{\infty}_{t}H^{2}_{x}$. 
First, by straight-forward computations we get that
\begin{lemma}\label{lm:q-ests} For all $p\in[1,\infty]$ and $q\in[1,\infty)$ we have that
\begin{equation}
\begin{aligned}\left\Vert \beta\right\Vert _{L_{x}^{p}}\lesssim & \left\Vert \eta\xi\right\Vert _{L_{x}^{p}}\\
\left\Vert \nabla\beta\right\Vert _{L_{x}^{q}}\lesssim & \left\Vert \eta\xi\right\Vert _{L_{x}^{q}}+\left\Vert \nabla\xi\right\Vert _{L_{x}^{q}}+\left\Vert \left|\nabla\eta\right|\xi\right\Vert _{L_{x}^{q}}+\left\Vert \eta\left|\nabla\xi\right|\right\Vert _{L_{x}^{q}}\\
\left\Vert \mathbf{q}\right\Vert _{L_{x}^{p}}\lesssim & \left\Vert \xi\right\Vert _{L_{x}^{p}}+\left\Vert \eta\xi\right\Vert _{L_{x}^{p}}\\
\left\Vert \nabla\mathbf{q}\right\Vert _{L_{x}^{q}}\lesssim & \left\Vert \eta\xi\right\Vert _{L_{x}^{q}}+\left\Vert \nabla\xi\right\Vert _{L_{x}^{q}}+\left\Vert \left|\nabla\eta\right|\xi\right\Vert _{L_{x}^{q}}+\left\Vert \eta\left|\nabla\xi\right|\right\Vert _{L_{x}^{q}}
\end{aligned}
\end{equation}

All continuity constants depend only on $\Omega$ and $M$.
\end{lemma}



\begin{proof}
The estimates for $\beta$ are straight-forward and use only the triangle inequality.  For the estimates of $\mathbf{q}$,
we combine Lemma~\ref{lm:Piola-est} and Lemma~\ref{lm:q-ests}..
    We also need to make use of the Sobolev embeddings $H^{2}\left(\omega\right)\hookrightarrow W^{1,q}\left(\omega\right)$ for any $q\in [1,\infty)$ and $H^{2}\left(\omega\right)\hookrightarrow L^{\infty}\left(\omega\right)$.

\end{proof}

Next, let us present some pointwise estimates for $\mathcal{F}^{s}_{\eta}$. We have
\begin{lemma}\label{lm:pointwise-F-s-eta} It holds, up to a constant depending on $\Omega$ and $M$, that
\begin{equation}
\begin{aligned}\left|\mathcal{F}_{\eta}^{s}\left(\xi\right)\right|\lesssim & \left(1+\left|\eta\right|+\left|\nabla\eta\right|\right)\left(\left|\xi\right|+\left|\eta\xi\right|\right)\\
\left|\nabla\mathcal{F}_{\eta}^{s}\left(\xi\right)\right|\lesssim & \left(1+\left|\eta\right|^{3}+\left|\nabla\eta\right|^{3}+\left|\nabla^{2}\eta\right|\right)\left(\left|\xi\right|+\left|\eta\xi\right|\right)+\\
 & \left(1+\left|\eta\right|+\left|\nabla\eta\right|\right)\left(\left|\eta\xi\right|+\left|\left|\nabla\eta\right|\xi\right|+\left|\eta\right|\left|\nabla\xi\right|\right)
\end{aligned}
\end{equation}
\end{lemma}
\begin{proof}
We use  Lemma~\ref{lm:q-ests} and we combine it with the Piola estimates from Lemma~\ref{lm:Piola-est} 
\end{proof}
\begin{proposition}\label{prop:extension-slip}
Let $\eta \in L^{\infty}\left(I;H_{0}^{2}\left(\omega\right)\right)$ and assume without loss of generality that $\left\Vert \eta\right\Vert _{L_{t}^{\infty}H_{x}^{2}}\le1$. Then,
for any $\xi \in H_{0}^{2}(\omega)$ such that there exists $\mathbf{q}$ for which $\left(\mathbf{q},\xi\right)\in\mathcal{T}^{\eta}$, the extension operator $\mathcal{F}_{\eta}^{s}\left(\xi\right)$ constructed above has the property $\left(\mathcal{F}_{\eta}^{s}\left(\xi\right),\xi\right)\in\mathcal{T}^{\eta}$.
Furthermore, for any $1\le p < p_1\le \infty$ and $1\le q\le 2$ and any $q_1>q$ 
it enjoys the estimates 
\begin{equation}\label{eqn:extension-estimates}
\begin{aligned}\left\Vert \mathcal{F}_{\eta}^{s}\left(\xi\right)\right\Vert _{L_{x}^{p}}\lesssim & \left\Vert \xi\right\Vert _{L_{x}^{p_{1}}}\\
\left\Vert \nabla\mathcal{F}_{\eta}^{s}\left(\xi\right)\right\Vert _{L_{x}^{q}}\lesssim & \left\Vert \xi\right\Vert _{L_{x}^{q_{1}}}+\left\Vert \nabla\xi\right\Vert _{L_{x}^{q_{1}}}+\left\Vert \xi\right\Vert _{L_{x}^{2}}
\end{aligned}
\end{equation}
    
\end{proposition}

\subsection{A regularizing operator}\label{ssec:regularization-operator}
We recall a regularizing operator from \cite[p. 234]{LR14} and \cite[Lemma 4.1]{BS18} in the following:
\begin{lemma}
Let $\varepsilon>0$. There exists an operator 
\[
\mathcal{R}_{\varepsilon}:C^{0}\left(\overline{I};C^{0}\left(\overline{\omega}\right)\right)\mapsto C^{4}\left(I;C^{4}\left(\omega\right)\right)
\]
such that 
\begin{enumerate}[(a)]
    \item $\mathcal{R}_{\varepsilon} \to \delta$ uniformly, for every $\delta \in C^{0}\left(\overline{I};C^{0}\left(\overline{\omega}\right)\right)$.
\item For any $\delta$ as above we have $\mathcal{R}_{\varepsilon}\delta\to\delta$ in $L^{p}\left(I;X\right)$ for any $p \in [1, \infty]$ and any $X$ of the following: $L^{q}\left(\omega\right)$, $W^{1,q}(\omega)$, $C^{\gamma}(\omega)$ with $q\in [1,\infty]$ and $\gamma \in (0,1)$.
\item If $\partial_{t}\delta\in L^{p}\left(I\times\omega\right)$ it holds that $\partial_{t}\mathcal{R}_{\varepsilon}\delta=\mathcal{R}_{\varepsilon}\partial_{t}\delta\to\partial_{t}\delta$.

\item For any $\delta$ as above we have $\mathcal{R}_{\varepsilon}\delta\ge\delta$ and $\left\Vert \mathcal{R}_{\varepsilon}\delta\right\Vert _{L_{t,x}^{\infty}}\le\left\Vert \delta\right\Vert _{L_{t,x}^{\infty}}+\varepsilon$ 
\end{enumerate}
\end{lemma}

\begin{remark}
    For any $\mathbf{v}\in L^{2}\left(I;L^{2}\left(\mathbb{R}^{3}\right)\right)$ we  define $\mathcal{R}_{\varepsilon}\mathbf{v}$ as the usual regularization in time and space, obtained by convolution with a mollifier.
\end{remark}

\section{Proof of the main result}\label{sec:proof-main}
\paragraph{Plan of the proof of Theorem~\ref{thm:main}}
\begin{enumerate}
\item In Subsection~\ref{ssec:compactness} we establish the $L^{2}$ compactness arguments needed in the proof. They follow closely the ones from \cite[Subsection 3.1]{LR14}, but they are slightly modified in order to be adapted to our Navier slip situation. 
    \item In order to eliminate the double role of $\eta$, as an unknown in the shell equation but in the same time prescribing (unknown) moving domains $\Omega^{\eta}$, we \emph{decouple} the problem. By this, we mean that we replace $\Omega^{\delta}$ for \emph{any} given and sufficiently smooth mapping $\delta : I \times \omega\mapsto \mathbb{R}$. The details are presented in Subsection~\ref{ssec:decoupled}.
    \item We solve the corresponding system by means of Galerkin approximations in Subsection~\ref{ssec:construction-weak}
    \item To recover the moving domains $\Omega^{\eta}$ we perform a set-valued fixed-point procedure, explained in Subsection~\ref{ssec:set-v-fixed-point}.
    \item Finally we perform a limit passage for the sequence of suitable approximations of the original problem. 
\end{enumerate}
\subsection{The compactness argument}\label{ssec:compactness}
We present here an $L^{2}$ in time and space compactness argument which will be used several times in the proofs. 
As usual when dealing with incompressible Navier-Stokes equations we need to treat the $L^{2}$ compactness of an approximating sequence $\mathbf{u}_{n} \to \mathbf{u}$ in order to pass the limit in the convective term. 
For this, we proceed in a similar manner as in \cite[Proposition 3.8]{lengeler2014weak}. We have also presented the details in 
\cite[Section 4]{mindrila-roy-multuilayered}. We will therefore insist only on the main differences.
First, let us recall the extension operator $\mathcal{F}_{\eta}$  from Proposition~\ref{prop:extension-no-slip}. Although it will make the slip term (containing $\frac{1}{\alpha}$) vanish, its use suffices to prove the  compactness argument and simplify the estimates that will appear.

    With the aid of Proposition~\ref{prop:extension-no-slip} we may now prove the following:

\begin{proposition}\label{prop:compactness}
    Let  $\left(\mathbf{u}_{n},\eta_{n}\right)\in\mathcal{S}^{\eta_{n}}$ be a sequence of solutions of \eqref{eqn:weak-formulation} with initial data 
    \[\left(\mathbf{u}_{n}\left(0,\cdot\right),\eta_{n}\left(0,\cdot\right),\partial_{t}\eta_{n}\left(0,\cdot\right)\right)=:\left(\mathbf{u}_{0n},\eta_{0n},\eta_{1n}\right)\in\mathcal{S}^{\eta_{0n}}\]
    having finite energy, that is  
    \begin{equation}
    \sup_{n\ge1}E_{n}\left(0\right)<\infty .
    \end{equation}

  Then there exists a subsequence of the sequence $\left(\mathbf{u}_{n},\partial_{t}\eta_{n}\right)_{n\ge1}$ which we do not relabel and a pair $\left(\mathbf{u},\eta\right)\in\mathcal{S}^{\eta}$ such that 
  \begin{equation}\label{eqn:weak-conv-comp}
\left(\mathbf{u}_{n}\chi_{\Omega^{\eta_{n}}},\partial_{t}\eta_{n}\right)\to\left(\mathbf{u},\partial_{t}\eta\right)\quad\text{in}\ L^{2}\left(I;L^{2}\left(\mathbb{R}^{3}\right)\times L^{2}\left(\omega\right)\right)
  \end{equation}
 where  $\chi_{A}$ denotes the characteristic function of a set $A$.
\end{proposition}
\begin{proof}
Let us start by noticing that from \eqref{eqn:energy-ineq}
we obtain that 
\begin{equation}\label{eqn:sup-n-energies}
\sup_{n\ge1}\left(\sup_{t\in\left(0,T\right)}E_{n}\left(t\right)+\int_{0}^{T}E_{slip,n}\left(t\right)dt+\int_{0}^{T}\int_{\Omega^{\eta_{n}}\left(t\right)}\left|\nabla\mathbf{u}_{n}\right|^{r}dxdt\right)<\infty
\end{equation}
for all $r<2$. This justifies the existence of a weak limit $\left(\mathbf{u}, \eta \right)$ as in \eqref{eqn:weak-conv-comp} and of a subsequence converging weakly to it.
    In the sequel we shall follow the same plan that was detailed in \cite[Section 4]{mindrila-roy-multuilayered}, with the appropriate adjustments.  
    \begin{enumerate}
        \item 
     First, by considering, for any $\xi \in H_{0}^{2}(\omega)$, the functionals \begin{equation}
        \begin{aligned}c_{\xi,n}\left(t\right):= & \int_{\Omega^{\eta_{n}}\left(t\right)}\mathbf{u}_{n}\cdot\mathcal{F}_{\eta_{n}}\left(\xi\right)dx+\int_{\omega}\partial_{t}\eta_{n}\xi dA\\
c_{\xi}\left(t\right):= & \int_{\Omega^{\eta}\left(t\right)}\mathbf{u}\cdot\mathcal{F}_{\eta}\left(\xi\right)dx+\int_{\omega}\partial_{t}\eta\xi dA
\end{aligned}
    \end{equation}
we  prove that 
\begin{equation}\label{eqn:compact-1}
    \sup_{t\in I}\sup_{\left\Vert \xi\right\Vert _{H_{x}^{2}}\le1}\left|c_{\xi,n}-c_{\xi}\right|\xrightarrow{n\to\infty}0
\end{equation}

To justify \eqref{eqn:compact-1} the main effort is to prove that there exists $\alpha \in (0,1)$ for which 
\begin{equation}
    \sup_{n\ge1}\sup_{t\neq s}\frac{\left|c_{\xi,n}\left(t\right)-c_{\xi,n}\left(s\right)\right|}{\left|t-s\right|^{\alpha}}<\infty 
\end{equation}
and to use the Arzela- Ascoli theorem.
The situation is analogue to  \cite[Section 4]{mindrila-roy-multuilayered}. The boundary term appearing from the slip condition vanishes, namely   for $\mathbf{q}=\mathcal{F}_{\eta_{n}}\left(\xi\right)$.
  
Let us recall the properties of the extension operator from Proposition~\ref{prop:extension-no-slip}.
To estimate the terms, we exemplify our argument on the diffusive term.
We have
\begin{equation}
\begin{aligned}
\int_{s}^{t}\int_{\Omega^{\eta_{n}\left(\tau\right)}}\mathbb{D}\mathbf{u}_{n}:\mathbb{D}\mathcal{F}_{\eta_{n}}\left(\xi\right)\le & \left|t-s\right|^{1/2}\left\Vert \mathbb{D}\mathbf{u}_{n}\right\Vert _{L_{t}^{2}L_{x}^{2}}\left\Vert \mathbb{D}\mathcal{F}_{\eta_{n}}\left(\xi\right)\right\Vert _{L_{t}^{\infty}L_{x}^{2}}\\
\lesssim & \left|t-s\right|^{1/2}\left\Vert \nabla\mathcal{F}_{\eta_{n}}\left(\xi\right)\right\Vert _{L_{t}^{\infty}L_{x}^{2}}\\
\lesssim & \left|t-s\right|^{1/2}\left\Vert \nabla\left(\eta_{n}\xi\right)\right\Vert _{L_{t}^{\infty}L_{x}^{2}}\\
\lesssim & \left|t-s\right|^{1/2}
\end{aligned}
\end{equation}
The other estimates follow similarly, see  \cite[Section 4.1]{mindrila-roy-multuilayered} for more details.
\item Next, having obtained that
\begin{equation}\label{eqn:unif-conv-H2-c-xi}
    h_{n}\left(t\right):=\sup_{\left\Vert \xi\right\Vert _{H_{x}^{2}}\le1}\left|c_{\xi,n}\left(t\right)-c_{\xi}\left(t\right)\right|\to0\quad\text{in}\ C\left(\overline{I}\right)
\end{equation}
 we extend it to functions $\xi \in L^{2}$, namely 
    \begin{equation}\label{eqn:unif-conv-xi-l2}
        g_{n}\left(t\right):=\sup_{\left\Vert \xi\right\Vert _{L_{x}^{2}}\le1}\left|c_{\xi,n}\left(t\right)-c_{\xi}\left(t\right)\right|\to0\quad\text{in}\ C\left(\overline{I}\right).
    \end{equation}
    
From \eqref{eqn:unif-conv-xi-l2}, we can choose $\xi:=\partial_{t}\eta_{n}\left(t,\cdot\right)$ in \eqref{eqn:unif-conv-xi-l2} and obtain that 
\begin{equation}\label{eqn:comp-u-eta-1}
\int_{I}\int_{\Omega^{\eta_{n}}\left(t\right)}\mathbf{u}_{n}\cdot\mathcal{F}_{\eta_{n}}\left(\partial_{t}\eta_{n}\right)+\int_{I}\int_{\omega}\left(\partial_{t}\eta_{n}\right)^{2}\to\int_{I}\int_{\Omega^{\eta}\left(t\right)}\mathbf{u}\cdot\mathcal{F}_{\eta}\left(\partial_{t}\eta\right)+\int_{I}\int_{\omega}\left(\partial_{t}\eta\right)^{2}
\end{equation}

\item  The remaining part of the compactness argument, complementary to \eqref{eqn:comp-u-eta-1}, is to prove that

\begin{equation}\label{eqn:comp-2}
\int_{I}\int_{\Omega^{\eta_{n}}\left(t\right)}\mathbf{u}_{n}\cdot\left(\mathbf{u}_{n}-\mathcal{F}_{\eta_{n}}\left(\partial_{t}\eta_{n}\right)\right)\to\int_{I}\int_{\Omega^{\eta}\left(t\right)}\mathbf{u}\cdot\left(\mathbf{u}-\mathcal{F}_{\eta}\left(\partial_{t}\eta\right)\right)
\end{equation}

For this,  similarly to \eqref{eqn:unif-conv-H2-c-xi}, we first prove that 
\begin{equation}\label{eqn:comp-unif-W1n}
\int_{\Omega^{\eta_{n}}\left(t\right)}\mathbf{u}_{n}\cdot\mathcal{J}_{\eta_{n}}\left[\mathbf{q}\right]\left(t\right)dx\to\int_{\Omega^{\eta}\left(t\right)}\mathbf{u}\cdot\mathcal{J}_{\eta}\left[\mathbf{q}\right]\left(t\right)dx\quad\text{in}\ C\left(\overline{I}\right)
\end{equation}
uniformly for all $\mathbf{q}\in H_{N}:=\left\{ \mathbf{f}:\mathbf{f}\in W_{\text{div}}^{1,N}\left(\Omega\right),\ \text{tr}_{\Gamma}\mathbf{f}\cdot\mathbf{e}_{r}=0= \text{tr}_{\Gamma_{in/out}}\mathbf{f}\cdot\mathbf{\tau}_{1,2}\right\} $ for a sufficiently large $N>1$.
We briefly sketch the argument leading to \eqref{eqn:comp-unif-W1n}. For $\mathbf{q}\in H_{N}$ let us denote 
\begin{equation}
d_{n,\mathbf{q}}\left(t\right):=\int_{\Omega^{\eta_{n}}\left(t\right)}\mathbf{u}_{n}\cdot\mathcal{J}_{\eta_{n}}\mathbf{q}dx,\quad d_{\mathbf{q}}\left(t\right):=\int_{\Omega^{\eta}\left(t\right)}\mathbf{u}\cdot\mathcal{J}_{\eta}\mathbf{q}dx
\end{equation}

We show that there exists $\lambda \in [0,1)$ for which 
\begin{equation}
\left|d_{n,\mathbf{q}}\left(t\right)-d_{n,\mathbf{q}}\left(s\right)\right|\lesssim\left|t-s\right|^{\lambda}.
\end{equation}
We use the test function $\left(\mathcal{J}_{\eta_{n}}\mathbf{q},0\right)\in\mathcal{T}^{\eta_{n}}$.

By integrating \eqref{eqn:formal-weak-formulation} in time on $(s,t)$ we obtain that for all $0<s<t<T$ 
\begin{equation}
\begin{aligned}d_{\mathbf{q},n}\left(t\right)-d_{\mathbf{q},n}\left(s\right)= & \int_{s}^{t}\int_{\Omega^{\eta_{n}}\left(\tau\right)}\mathbf{u}_{n}\cdot\partial_{t}\mathcal{J}_{\eta_{n}}\mathbf{q}-\mathbb{D}\mathbf{u}_{n}:\mathbb{D}\mathcal{J}_{\eta_{n}}\mathbf{q}-\int_{s}^{t}b\left(\tau,\mathbf{u}_{n},\mathbf{u}_{n},\mathcal{J}_{\eta_{n}}\mathbf{q}\right)+\\
 & \frac{1}{2}\int_{s}^{t}\int_{\Gamma^{\eta}\left(\tau\right)}\left(\mathbf{u}_{n}\cdot\mathcal{J}_{\eta_{n}}\mathbf{q}\right)\left(\partial_{t}\eta_{n}\mathbf{e}_{r}\circ\phi_{\eta_{n}\left(t\right)}^{-1}\right)\cdot\boldsymbol{\nu}^{\eta_{n}\left(t\right)}+\\
 & \frac{1}{\alpha}\int_{s}^{t}\int_{\omega}\left(\mathbf{u}_{n}\circ\phi_{\eta_{n}\left(t\right)}-\partial_{t}\eta_{n}\mathbf{e}_{r}\right)\cdot\mathcal{J}_{\eta_{n}}\mathbf{q}\circ\phi_{\eta_{n}\left(t\right)}J_{\eta_{n}\left(t\right)}+\\
 & \left\langle F\left(t\right),\mathcal{J}_{\eta_{n}}\mathbf{q}\right\rangle +\int_{\Omega^{\eta_{n}\left(0\right)}}\mathbf{u}_{0n}\cdot\mathcal{J}_{\eta_{n}}\mathbf{q}\left(0\right)
\end{aligned}
\end{equation}
And we proceed to  estimate each term, as follows: 
\begin{equation}
\begin{aligned}\int_{s}^{t}\int_{\Omega^{\eta_{n}}\left(\tau\right)}\mathbb{D}\mathbf{u}_{n}:\mathbb{D}\mathcal{J}_{\eta_{n}}\mathbf{q}\le & \left\Vert \mathbb{D}\mathbf{u}_{n}\right\Vert _{L_{t}^{2}L_{x}^{2}}\left\Vert \mathbb{D}\mathcal{J}_{\eta_{n}}\mathbf{q}\right\Vert _{L_{t}^{\infty}L_{x}^{2}}\left|t-s\right|^{1/2}\\
\lesssim & \left\Vert \mathcal{J}_{\eta_{n}}\mathbf{q}\right\Vert _{L_{t}^{\infty}W_{x}^{1,2}}\left|t-s\right|^{1/2}\\
\lesssim & \left|t-s\right|^{1/2}.
\end{aligned}
\end{equation}
Concerning the convective term, we use Remark~\ref{rmk:loss} to get

\begin{equation}
\begin{aligned}\left|\int_{s}^{t}b\left(\tau,\mathbf{u}_{n},\mathbf{u}_{n},\mathcal{J}_{\eta_{n}}\mathbf{q}\right)\right|\le & \int_{s}^{t}\int_{\Omega^{\eta_{n}}\left(\tau\right)}\left|\mathbf{u}_{n}\right|\left|\nabla\mathbf{u}_{n}\right|\left|\mathcal{J}_{\eta_{n}}\mathbf{q}\right|+\left|\mathbf{u}_{n}\right|^{2}\left|\nabla\mathcal{J}_{\eta_{n}}\mathbf{q}\right|\\
\lesssim & \left\Vert \mathbf{u}_{n}\right\Vert _{L_{t}^{10/3}L_{x}^{10/3-}}\left\Vert \nabla\mathbf{u}_{n}\right\Vert _{L_{t}^{2}L_{x}^{2-}}\left\Vert \mathcal{J}_{\eta_{n}}\mathbf{q}\right\Vert _{L_{t}^{\infty}L_{x}^{5+}}\left|t-s\right|^{1/5}+\\
 & \left\Vert \mathbf{u}_{n}\right\Vert _{L_{t}^{10/3}L_{x}^{10/3-}}\left\Vert \mathbf{u}_{n}\right\Vert _{L_{t}^{2}L_{x}^{6-}}\left\Vert \nabla\mathcal{J}_{\eta_{n}}\mathbf{q}\right\Vert _{L_{t}^{\infty}L_{x}^{2+}}\left|t-s\right|^{1/30}\\
\lesssim & \left\Vert \mathbf{q}\right\Vert _{L_{t}^{\infty}L_{x}^{5+}}\left|t-s\right|^{1/5}+\left\Vert \nabla\mathbf{q}\right\Vert _{L_{t}^{\infty}L_{x}^{2+}}\left|t-s\right|^{1/30}\\
\lesssim & \left|t-s\right|^{1/30}
\end{aligned}
\end{equation}
for $\mathbf{q}\in H_{N}$ with $N>2$.
The boundary terms are estimated in a similar manner.

The most tedious estimate is the one of the time-derivative term. Let us  denote, for $\mathbf{q}\in H_{N}$  $\mathbf{Q}:=\frac{\nabla\psi_{\eta_{n}}}{\det\nabla\psi_{\eta_{n}}}\mathbf{q}$. Thus, $\mathcal{J}_{\eta_{n}}\mathbf{q}=\mathbf{Q}\left(t,\psi_{\eta_{n}\left(t\right)}^{-1}\right)$.
By the chain rule we have 
\[
\partial_{t}\mathcal{J}_{\eta_{n}}\mathbf{q}\left(x\right)=\partial_{t}\mathbf{Q}\left(t,\psi_{\eta_{n}\left(t\right)}^{-1}\left(x\right)\right)+\nabla\mathbf{Q}\left(t,\psi_{\eta_{n}\left(t\right)}^{-1}\left(x\right)\right)\partial_{t}\psi_{\eta_{n}\left(t\right)}^{-1}\left(t,x\right),\ x\in\Omega^{\eta_{n}}\left(t\right)
\]
which for $x=\psi_{\eta_{n}\left(t\right)}\left(X\right)$ with $X\in \Omega$ becomes
\[
\partial_{t}\mathcal{J}_{\eta_{n}}\mathbf{q}\left(\psi_{\eta_{n}\left(t\right)}\left(X\right)\right)=\partial_{t}\mathbf{Q}\left(t,X\right)+\nabla\mathbf{Q}\left(t,X\right)\partial_{t}\psi_{\eta_{n}\left(t\right)}^{-1}\left(t,\psi_{\eta_{n}\left(t\right)}\left(X\right)\right),\ X\in\Omega
\]
Now, since $\psi_{\eta_{n}}^{-1}\left(t,\psi_{\eta_{n}}\left(X\right)\right)=X$ by taking time-derivatives we get
\[
\partial_{t}\psi_{\eta_{n}\left(t\right)}^{-1}\left(t,\psi_{\eta_{n}\left(t\right)}\left(X\right)\right)+\nabla\psi_{\eta_{n}\left(t\right)}^{-1}\left(t,\psi_{\eta_{n}\left(t\right)}\left(X\right)\right)\partial_{t}\psi_{\eta_{n}\left(t\right)}\left(X\right)=0
\]
and $
\nabla\psi_{\eta_{n}\left(t\right)}^{-1}\left(t,\psi_{\eta_{n}\left(t\right)}\left(X\right)\right)\nabla\psi_{\eta_{n}\left(t\right)}\left(X\right)=\mathbb{I}_{3}
$
so we get
\begin{equation}
\partial_{t}\mathcal{J}_{\eta_{n}}\mathbf{q}\left(\psi_{\eta_{n}\left(t\right)}\left(X\right)\right)=\partial_{t}\mathbf{Q}\left(t,X\right)-\nabla\mathbf{Q}\left(t,X\right)\left(\nabla\psi_{\eta_{n}\left(t\right)}\left(X\right)\right)^{-1}\partial_{t}\psi_{\eta_{n}},\ X\in\Omega.
\end{equation}
We denote $\overline{\mathbf{u}}_{n}:=\mathbf{u}_{n}\circ\psi_{\eta_{n}}$ and we obtain that 
\begin{equation}   \begin{aligned}\int_{s}^{t}\int_{\Omega^{\eta_{n}}\left(\tau\right)}\mathbf{u}_{n}\cdot\partial_{t}\mathcal{J}_{\eta_{n}}\mathbf{q} & =\\
\int_{s}^{t}\int_{\Omega}\overline{\mathbf{u}}_{n}\cdot\left(\partial_{t}\mathbf{Q}\left(t,X\right)-\nabla\mathbf{Q}\left(t,X\right)\left(\nabla\psi_{\eta_{n}\left(t\right)}\left(X\right)\right)^{-1}\partial_{t}\psi_{\eta_{n}}\right)\det\nabla\psi_{\eta_{n}\left(t\right)} & =:J_{1}-J_{2}
\end{aligned}
\end{equation}
Since $0<\det\nabla\psi_{\eta_{n}\left(t\right)}\lesssim1$ we write 
\begin{equation}
\begin{aligned}J_{1}\lesssim & \int_{s}^{t}\int_{\Omega}\overline{\mathbf{u}}_{n}\cdot\frac{\partial_{t}\nabla\psi_{\eta_{n}}\det\nabla\psi_{\eta_{n}}-\nabla\psi_{\eta_{n}}\partial_{t}\det\nabla\psi_{\eta_{n}}}{\det\nabla\psi_{\eta_{n}}}\mathbf{q}\\
\lesssim & \int_{s}^{t}\int_{\Omega}\overline{\mathbf{u}}_{n}\cdot\partial_{t}\nabla\psi_{\eta_{n}}\mathbf{q}-\frac{\partial_{t}\det\nabla\psi_{\eta_{n}}}{\det\nabla\psi_{\eta_{n}}}\overline{\mathbf{u}}_{n}\cdot\nabla\eta_{n}\mathbf{q}=:J_{1,1}-J_{1,2}
\end{aligned}
\end{equation}
and let us focus on the most tedious term, namely $J_{1,1}$. For this, we  make use of the estimate~\eqref{eqn:sup-n-energies}, we note that 
\[
\left|\partial_{t}\eta_{n}\left(R+\eta_{n}\right)\right|=\left|\partial_{t}\eta_{n}\mathbf{e}_{r}\cdot\boldsymbol{\nu}^{\eta_{n}}\right|=\left|\text{tr}_{\Gamma^{\eta_{n}}}\mathbf{u}_{n}\cdot\boldsymbol{\nu}^{\eta_{n}}\right|\le\left|\text{tr}_{\Gamma^{\eta_{n}}}\mathbf{u}_{n}\right|
\]
and hence $\partial_{t}\eta_{n}\in L_{t}^{2}L_{x}^{4-}$ 
and therefore by using linear interpolation we have
\[
\partial_{t}\eta_{n}\in L_{t}^{\infty}L_{x}^{2}\cap L_{t}^{2}L_{x}^{4-}\hookrightarrow L_{t}^{3}L_{x}^{3}
\]
Now, the important observation is that we integrate by parts to move the space derivatives from $\nabla\psi_{\eta_{n}}$ to $\overline{\mathbf{u}}_{n}$ and $\mathbf{q}$ and we obtain again the H\"{o}lder continuous estimates as before.  

The estimates concerning $J_{1,2}$ and $J_{2}$ can be obtained again in a (by now) straight-forward way via the energy estimates \eqref{eqn:sup-n-energies}.

Thus, we obtain that
\begin{equation}\label{eqn:unif-conv-HN}
    \sup_{\mathbf{q}\in H_{N}}\left|\int_{I}\int_{\Omega^{\eta_{n}}\left(t\right)}\mathbf{u}_{n}\cdot\mathcal{J}_{\eta_{n}}\mathbf{q}dxdt-\int_{I}\int_{\Omega^{\eta}\left(t\right)}\mathbf{u}\cdot\mathcal{J}_{\eta}\mathbf{q}dxdt\right|\xrightarrow{n\to\infty}0
\end{equation}

\item In order to cast $\mathbf{u}_{n}-\mathcal{F}_{\eta_{n}}\left(\partial_{t}\eta_{n}\right)$
 as  $\mathcal{J}_{\eta_{n}}\mathbf{q}$ for some $\mathbf{q}\in H_{N}$ (in order to get \eqref{eqn:comp-2})
we set 
\begin{equation}\label{eqn:choice-q}
\mathbf{q}:=\mathcal{M}_{\rho}\mathcal{J}_{\eta_{n}}^{-1}\left(\mathbf{u}_{n}-\mathcal{F}_{\eta_{n}}\left(\partial_{t}\eta_{n}\right)\right)
\end{equation} where $\mathcal{M}_{\rho}$ is a standard mollification operator\footnote{It can be defined as follows: given a  mollifier $m\in C_{0}^{\infty}\left(\mathbb{R}^{3};\mathbb{R}_{+}\right),\ \int m=1$ for any $\rho>0$ we define the mollification operator of a function $\mathbf{f}:\Omega\mapsto\mathbb{R}^{3}$ as follows:
$\mathcal{M}_{\rho}\mathbf{f}\left(x\right):=\rho^{-3}\int_{\mathbb{R}^{3}}m\left(\frac{x-y}{\rho}\right)\mathbf{f}\chi_{\Omega}\left(y\right)dy$.}

We notice that on $\Gamma^{\eta_n}$, at $r=R+\eta_{n}$, we have
$
\left(\mathbf{u}_{n}-\mathcal{F}_{\eta_{n}}\left(\partial_{t}\eta_{n}\right)\right)\cdot\boldsymbol{\nu}^{\eta_{n}}=0$. 
Using Lemma~\ref{lm:Piola} we get that at $r=R$ we have $\mathcal{J}_{\eta_{n}}^{-1}\left(\mathbf{u}_{n}-\mathcal{F}_{\eta_{n}}\left(\partial_{t}\eta_{n}\right)\right)\cdot\mathbf{e}_{r}=0$ so 
\[\mathcal{M}_{\rho}\mathcal{J}_{\eta_{n}}^{-1}\left(\mathbf{u}_{n}-\mathcal{F}_{\eta_{n}}\left(\partial_{t}\eta_{n}\right)\right)\cdot\mathbf{e}_{r}=0\] on $\Gamma$.
This is because around $\Gamma$, at $r\sim R$, the map $r \mapsto\mathcal{J}_{\eta_{n}}^{-1}\left(\mathbf{u}_{n}-\mathcal{F}_{\eta_{n}}\left(\partial_{t}\eta_{n}\right)\right)\cdot \mathbf{e}_{r}$ is a Sobolev function with zero trace and it can be approximated by a smooth function with compact support.  On the other hand for $z\in\left\{ 0,L\right\}$ (that is on $\Gamma_{in/out}$) we have that \[\mathcal{J}_{\eta_{n}}^{-1}\left(\mathbf{u}_{n}-\mathcal{F}_{\eta_{n}}\left(\partial_{t}\eta_{n}\right)\right)\times\mathbf{e}_{z}=0\]
and thus $\mathcal{M}_{\rho}\mathcal{J}_{\eta_{n}}^{-1}\left(\mathbf{u}_{n}-\mathcal{F}_{\eta_{n}}\left(\partial_{t}\eta_{n}\right)\right)\in H_{N}$ for any $N$, due to the smoothness of $\mathcal{M}_{\rho}$.

We now choose in \eqref{eqn:unif-conv-HN} a  $\mathbf{q}$ as in \eqref{eqn:choice-q}.
Let $\varepsilon>0$. There is $n(\varepsilon) \ge 1$ for which 

\begin{equation}
\left|\int_{I}\int_{\Omega^{\eta_{n}}\left(t\right)}\mathbf{u}_{n}\cdot\mathcal{J}_{\eta_{n}}\mathbf{q}-\int_{I}\int_{\Omega^{\eta}\left(t\right)}\mathbf{u}\cdot\mathcal{J}_{\eta}\mathbf{q}\right|<\varepsilon\left\Vert \mathbf{q}\right\Vert _{W_{x}^{1,N}},\ \forall \  n\ge n\left(\varepsilon\right)
\end{equation}

where $\mathbf{q}=\mathbf{q}_{\rho,n}:=\mathcal{M}_{\rho}\mathcal{J}_{\eta_{n}}^{-1}\left(\mathbf{u}_{n}-\mathcal{F}_{\eta_{n}}\left(\partial_{t}\eta_{n}\right)\right)\in H_{N}$. We see that \[\mathbf{q}_{\rho,n}\xrightarrow{\rho\to0}\mathbf{q}_{n}:=\mathcal{J}_{\eta_{n}}^{-1}\left(\mathbf{u}_{n}-\mathcal{F}_{\eta_{n}}\left(\partial_{t}\eta_{n}\right)\right) \quad \text{in} \ L^2(\Omega).\] 

From 
\[
\left|\int_{I}\int_{\Omega^{\eta_{n}}\left(t\right)}\mathbf{u}_{n}\cdot\mathcal{J}_{\eta_{n}}\mathbf{q}_{\rho,n}-\int_{I}\int_{\Omega^{\eta}\left(t\right)}\mathbf{u}\cdot\mathcal{J}_{\eta}\mathbf{q}_{\rho,n}\right|<\varepsilon\left\Vert \mathbf{q}_{\rho,n}\right\Vert _{W_{x}^{1,N}}
\]
since 
 $
 \left\Vert \mathbf{q}_{\rho,n}\right\Vert _{W_{x}^{1,N}}\lesssim\rho^{-K}\left\Vert \mathbf{q}_{n}\right\Vert _{L_{x}^{2}}\lesssim\rho^{-K}
 $
for a sufficiently large $K=K(N)$, we may choose $\rho=\rho_{\varepsilon}$ such that $\rho_{\varepsilon}^{K}=\sqrt{\varepsilon}$ to obtain that 
\[
\left|\int_{I}\int_{\Omega^{\eta_{n}}\left(t\right)}\mathbf{u}_{n}\cdot\mathcal{J}_{\eta_{n}}\mathbf{q}_{\rho_{\varepsilon},n}dxdt-\int_{I}\int_{\Omega^{\eta}\left(t\right)}\mathbf{u}\cdot\mathcal{J}_{\eta}\mathbf{q}_{\rho_{\varepsilon},n}dxdt\right|\lesssim \sqrt{\varepsilon}\ , \quad \forall \ n\ge n(\varepsilon) .
\]

For $\varepsilon$ small enough (and hence $\rho_{\varepsilon}$ small) we get that 
\[
\left|\int_{I}\int_{\Omega^{\eta_{n}}\left(t\right)}\mathbf{u}_{n}\cdot\mathcal{J}_{\eta_{n}}\mathbf{q}_{n}-\int_{I}\int_{\Omega^{\eta}\left(t\right)}\mathbf{u}\cdot\mathcal{J}_{\eta}\mathbf{q}_{n}\right|\lesssim\sqrt{\varepsilon},\quad \forall \  n\ge n\left(\varepsilon\right).
\]
Since $\mathcal{J}_{\eta_{n}}\mathbf{q}_{n}=\mathbf{u}_{n}-\mathcal{F}_{\eta_{n}}\left(\partial_{t}\eta_{n}\right)$
we are only left to prove that $\mathcal{J}_{\eta}\mathbf{q}_{n}\rightharpoonup\mathbf{u}-\mathcal{F}_{\eta}\left(\partial_{t}\eta\right)$ in an appropriate sense. But
$\mathbf{u}_{n}-\mathcal{F}_{\eta_{n}}\left(\partial_{t}\eta_{n}\right)\rightharpoonup\mathbf{u}-\mathcal{F}_{\eta}\left(\partial_{t}\eta\right)$ in $L^{2}_{t}L^{2}_{x}$ and $\nabla\psi_{\eta_{n}}\to\nabla\psi_{\eta}$ in $L_{t}^{\infty}L_{x}^{r}$ for any $1 \le r <\infty$ and $\det\nabla\psi_{\eta_{n}}\to\det\nabla\psi_{\eta}$ uniformly in $(t,x)$ we get that  $\mathcal{J}_{\eta_{n}}^{-1}\left(\mathbf{u}_{n}-\mathcal{F}_{\eta_{n}}\left(\partial_{t}\eta_{n}\right)\right)\rightharpoonup\mathcal{J}_{\eta}^{-1}\left(\mathbf{u}-\mathcal{F}_{\eta}\left(\partial_{t}\eta\right)\right)$ in $L^{2}_{t} L^{2}_{x}$. This means that $\mathcal{J}_{\eta}\mathbf{q}_{n}\rightharpoonup\mathbf{u}-\mathcal{F}_{\eta}\left(\partial_{t}\eta\right)$ in the same space. This allows as to establish \eqref{eqn:comp-2}.
And by adding \eqref{eqn:comp-u-eta-1} and \eqref{eqn:comp-2} we obtain the norm convergence in the Hilbert space $L^{2}\left(I;L^{2}\left(\Omega^{\eta}\right)\times L^{2}\left(\omega\right)\right)$. Which improves the weak convergence from \eqref{eqn:weak-conv-comp} to the strong convergence announced in Proposition~\ref{prop:compactness}.
\end{enumerate}
 \end{proof}

\subsection{The decoupled, linearized and regularized problem}\label{ssec:decoupled}
In order to construct a solution for our FSI problem, since $\eta$ plays the role of an unknown in the elastic equation and also prescribes the moving domains $\Omega^{\eta}$, it seems necessary to \emph{decouple} the problem. 
Suppose we are given a smooth function $\delta\in C^{\infty}\left(I\times\omega\right)$ which prescribes a domain $\Omega^{\delta}\left(t\right)$ (with moving boundary $\Gamma^{\delta}(t)$) we can solve the FSI problem in this new domain, thus obtaining a solution dependent of $\delta$, denoted $\left(\mathbf{u}_{\delta},\eta_{\delta}\right)\in\mathcal{S}^{\delta}$. We observe that if we find a $\delta$ for which $\delta=\eta_{\delta}$, then this would recover a solution of the original problem. This will be achieved throughout the set-valued fixed-point Theorem~\ref{thm:kakutani}. Since Theorem~\ref{thm:kakutani} requires convexity of some sets, we 
will \emph{linearize} the problem, and more precisely the convective term (the only nonlinearity of this problem) around a smooth vector field $\mathbf{v}\in C^{\infty}\left(I\times\mathbb{R}^{3};\mathbb{R}^{3}\right)$. This is why we will replace $b\left(t,\mathbf{u},\mathbf{u},\mathbf{q}\right)$ by $b\left(t,\mathbf{u},\mathbf{v},\mathbf{q}\right)$.  By decoupling we write 
\[
b\left(t,\mathbf{u},\mathbf{v},\mathbf{q}\right):=\int_{\Omega^{\delta}\left(t\right)}\frac{1}{2}\left(\mathbf{u}\cdot\nabla\right)\mathbf{v}\cdot\mathbf{q}-\frac{1}{2}\left(\mathbf{u}\cdot\nabla\right)\mathbf{q}\cdot\mathbf{v}dx.
\]
In general the domain of definition of the initial value $\mathbf{u}_{0}$ is $\Omega^{\eta}\left(0\right)$ and we fix this by requiring that $\delta\left(0,\cdot\right)=\eta_{0}$ and by assuming that $\eta_{0}$ is smooth; for this one would have to regularize $\eta_{0}$ by writing $\eta_{0}=\mathcal{R}_{\varepsilon}\eta_{0}$. This can be achieved without issues in Subsection~\ref{ssec:set-v-fixed-point}. See also \cite[p. 234]{LR14}.

We introduce now an adapted notion of the solution.
\begin{definition} Given a pair
$
\left(\delta,\mathbf{v}\right)\in C^{\infty}\left(I\times\omega\right)\times L^{2}\left(I\times\mathbb{R}^{3}\right)$
  we call  a couple $\left(\mathbf{u},\eta\right)\in\mathcal{S}^{\delta}$  a  solution for the \emph{decoupled and linearized} problem  provided that the following variational formulation holds:
  \begin{enumerate}
      \item For almost all $t\in I$ and for all $\left(\mathbf{q},\xi\right)\in\mathcal{T}^{\delta}$ the following holds: 
\begin{equation}
\begin{aligned}\int_{\Omega^{\delta}\left(t\right)}\mathbf{u}\cdot\mathbf{q}\left(t\right)dx+\int_{0}^{t}\int_{\Omega^{\delta}\left(s\right)}-\mathbf{u}\cdot\partial_{t}\mathbf{q}+\mathbb{D}\mathbf{u}:\mathbb{D}\mathbf{q}dxds & +\\
\int_{0}^{t}b\left(s,\mathbf{u},\mathbf{v},\mathbf{q}\right)ds+\int_{0}^{t}\int_{\Gamma^{\delta}\left(s\right)}-\frac{1}{2}\left(\mathbf{u}\cdot\mathbf{q}\right)\left(\partial_{t}\delta\mathbf{e}_{r}\circ\phi_{\delta\left(s\right)}^{-1}\right)\cdot\boldsymbol{\nu}^{\delta\left(s\right)}dA_{\delta}ds & +\\
\frac{1}{\alpha}\int_{0}^{t}\int_{\omega}\left(\mathbf{u}\circ\phi_{\delta\left(s\right)}-\partial_{t}\eta\mathbf{e}_{r}\right)\cdot\left(\mathbf{q}\circ\phi_{\delta\left(s\right)}-\xi\mathbf{e}_{r}\right)J_{\delta\left(s\right)}dAds & +\\
\int_{\omega}\partial_{t}\eta\cdot\xi\left(t\right)dA+\int_{0}^{t}\int_{\omega}-\partial_{t}\eta\cdot\partial_{t}\xi+\nabla^{2}\eta:\nabla^{2}\xi dAds & =\\
\int_{0}^{t}\left\langle F\left(t\right),\mathbf{q}\right\rangle ds+\int_{\Omega^{\delta}\left(0\right)}\mathbf{u}_{0}\cdot\mathbf{q}\left(0\right)dx+\int_{\omega}\eta_{1}\xi dA.
\end{aligned}
\end{equation}
\item The following energy balance holds 
\begin{equation}\label{eqn:delta-energy-ineq}
\sup_{t\in\left(0,T\right)}E_{\delta}\left(t\right)+\int_{0}^{T}E_{slip,\delta}\left(t\right)dt+\int_{0}^{T}\int_{\Omega^{\delta}\left(t\right)}\left|\nabla\mathbf{u}\right|^{r}dxdt\lesssim E(0)+\left\Vert P\right\Vert _{L_{t}^{2}}^{2}<\infty
\end{equation}
for all $r<2$.
  \end{enumerate}

\end{definition}
\begin{remark} The energy balance \eqref{eqn:delta-energy-ineq} is the analogue of \eqref{eqn:energy}. The corresponding $E_{\delta}$ energy is defined by \eqref{eqn:notation-energy-delta}. In order to obtain \eqref{eqn:delta-energy-ineq} by testing with $\left(\mathbf{u},\partial_{t}\eta\right)$, we have modified the term involving $\int_{\Gamma^{\delta}\left(t\right)}$ accordingly.
\end{remark}

\subsubsection{The construction of a decoupled and linearized solution}\label{ssec:construction-weak}
We will now construct the Galerkin approximations $\left(\mathbf{u}_{n},\eta_{n}\right)_{n\ge1}$ of a decoupled and linearized solution $\left(\mathbf{u},\eta\right)\in\mathcal{S}^{\delta}$.
For this, we recall the extension operator $\mathcal{F}^{s}_{\delta}$ from Proposition~\ref{prop:extension-slip}.
The Galerkin basis will consist of two parts: 
\begin{itemize}
    \item Let $\left(Y_{k}\right)_{k\ge1}$ be a smooth basis of the space $H_{0}^{2}(\omega)$. Set $\mathbf{Y}_{k}\left(t,x\right):=\mathcal{F}^{s}_{\delta}\left(Y_{k}\right)$ for $k\ge 1$. Note that $\mathbf{Y}_{k}\left(t,\cdot\right):\Omega^{\delta}\left(t\right)\mapsto\mathbb{R}^{3}$ with $\text{div} \mathbf{Y}_{k}\left(t,\cdot\right)=0$.
   By construction we also have that $\text{tr}_{\Gamma^{\delta}}\left(\mathbf{Y}_{k}\right)\cdot\boldsymbol{\nu}^{\delta}=Y_{k}\mathbf{e}_{r}\cdot\boldsymbol{\nu}^{\delta}$. 
\item Let $\left(\mathbf{Z}_{k}\right)_{k\ge1}$ be a basis of the space $\mathcal{H}:=\left\{ \mathbf{v}\in H_{\text{div}}^{1}\left(\Omega\right):\text{tr}_{\Gamma_{in/out}}\left(\mathbf{v}\right)\cdot\boldsymbol{\tau}_{1,2}=0,\ \text{tr}_{\Gamma}\left(\mathbf{v}\right)\cdot\mathbf{e}_{r}=0\right\}$. 
Such a basis can be constructed by using the Stokes operator in a similar manner as in \cite[p. 15]{mindrila-roy-multuilayered}.
By applying now the Piola mapping $\mathcal{J}_{\delta}$ to $\mathcal{H}$ we will obtain an $H^1$ basis of the Hilbert space
\[
\mathcal{H}^{\delta\left(t\right)}:=\left\{ \mathbf{v}\in H_{\text{div}}^{1}\left(\Omega^{\delta}\left(t\right)\right):\text{tr}_{\Gamma_{in/out}}\left(\mathbf{v}\right)\cdot\boldsymbol{\tau}_{1,2}=0,\ \text{tr}_{\Gamma^{\delta\left(t\right)}}\left(\mathbf{v}\right)\cdot\boldsymbol{\nu}^{\delta}=0\right\}. 
\]
For this we use again Lemma~\ref{lm:Piola} to ensure that the zero normal component is preserved.

\item We join the families $\left\{ \mathbf{Y}_{k}\right\} _{k\ge1}$ and $\left\{ \mathbf{Z}_{k}\right\} _{k\ge1}$ by setting 
\begin{equation}
\mathbf{X}_{k}\left(t,x\right):=\begin{cases}
\mathbf{Y}_{k}\left(t,x\right) & k\ \text{odd}\\
\mathbf{Z}_{k}\left(t,x\right) & k\ \text{even}
\end{cases},\quad X_{k}\left(x\right)=\begin{cases}
Y_{k} & k\ \text{odd}\\
0 & k\ \text{even}
\end{cases}
\end{equation}
 and by construction we have that
\begin{equation}\label{eqn:couple-basis}
\text{tr}_{\Gamma^{\delta}}\mathbf{X}_{k}\cdot\boldsymbol{\nu}^{\delta}=X_{k}\mathbf{e}_{r}\cdot\boldsymbol{\nu}^{\delta}
\end{equation}
\end{itemize}

We may now consider the Galerkin approximations: for each $n\ge 1$ and each $k\in\left\{ 1,2,\ldots,n\right\}$ we aim to find $n$ functions  denoted by $a_{n}^{k}\in C^{1}\left(I;\mathbb{R}\right)$ such that the approximations $\left(\mathbf{u}_{n},\eta_{n}\right)$ of  defined via
\begin{equation}
\begin{aligned}\mathbf{u}_{n}\left(t,x\right):= & \sum_{k=1}^{n}\left(a_{n}^{k}\right)^{\prime}\left(t\right)\mathbf{X}_{k}\left(t,x\right), & \left(t,x\right)\in I\times\Omega^{\delta}\\
\eta_{n}\left(t,x\right):= & \sum_{k=1}^{n}a_{n}^{k}\left(t\right)X_{k}\left(x\right), & \left(t,x\right)\in I\times\omega
\end{aligned}
\end{equation}
From \eqref{eqn:couple-basis} we have ensured  that \begin{equation}
\text{tr}_{\Gamma^{\delta}}\mathbf{u}_{n}\cdot\boldsymbol{\nu}^{\delta}=\partial_{t}\eta_{n}\mathbf{e}_{r}\cdot\boldsymbol{\nu}^{\delta}
\end{equation}
solve the system of $n$ differential equations given, for any $k=\overline{1,n}$ by 
\begin{equation}\label{eqn:system-ode}
\begin{aligned}\int_{\Omega^{\delta}\left(t\right)}\partial_{t}\mathbf{u}_{n}\cdot\mathbf{X}_{k}+\mathbb{D}\mathbf{u}_{n}:\mathbf{\mathbb{D}X}_{k}dx+b\left(t,\mathbf{u}_{n},\mathbf{v},\mathbf{X}_{k}\right) & +\\
\int_{\Gamma^{\delta}\left(t\right)}\frac{1}{2}\left(\mathbf{u}\cdot\mathbf{X}_{k}\right)\left(\partial_{t}\delta\mathbf{e}_{r}\circ\phi_{\delta\left(t\right)}^{-1}\right)\cdot\boldsymbol{\nu}^{\delta\left(t\right)}dA_{\delta} & +\\
\frac{1}{\alpha}\int_{\omega}\left(\mathbf{u}_{n}\circ\phi_{\delta}-\partial_{t}\eta_{n}\mathbf{e}_{r}\right)\cdot\left(\mathbf{X}_{k}\circ\phi_{\delta}-X_{k}\mathbf{e}_{r}\right)J_{\delta}dA & +\\
\int_{\omega}\partial_{tt}\eta_{n}X_{k}+\nabla^{2}\eta_{n}:\nabla^{2}X_{k}dA & =\\
\left\langle F\left(t\right),\mathbf{X}_{k}\right\rangle. 
\end{aligned}
\end{equation}
Denoting $\mathbf{a}_{n}\left(t\right):=\left(a_{n}^{k}\left(t\right)\right)_{k=\overline{1,n}}$  we see that  \eqref{eqn:system-ode} is a system of differential equations of second order in the unknown $\mathbf{a}_{n}$
with the initial conditions $\mathbf{a}_{n}^{\prime}\left(0\right),\mathbf{a}_{n}\left(0\right)$ chosen such that $(\mathbf{u}_{n}\left(0\right),\partial_{t}\eta_{n}\left(0\right))\to (\mathbf{u}_{0},\eta_{1})$ in $L^{2}_{x}$ and $\eta_{n}\left(0,\cdot\right)\to\eta_{0}$ in $H^{2}_{x}$.

The system  \eqref{eqn:system-ode} can be solved locally in a time interval $[0,T_{0})$ by a Picard-Lindel\"{o}f argument,  in a similar manner as presented in \cite[p. 239]{LR14} and \cite[Section 5]{mindrila-roy-multuilayered}.  The mass matrix $M(t)$ of the system is given via \[M_{ik}\left(t\right):=\int_{\Omega^{\delta}\left(t\right)}\mathbf{X}_{i}\cdot\mathbf{X}_{k}dx+\int_{\omega}X_{i}\cdot X_{k}dA\] and we see that for an arbitrary $\mathbf{a}\in \mathbb{R}^{n}$ we have that
\[
M\mathbf{a}\cdot\mathbf{a}=\int_{\Omega^{\delta}\left(t\right)}\left|\sum_{i=1}^{n}a_{i}\mathbf{X}_{i}\right|^{2}dx+\int_{\omega}\left|\sum_{i=1}^{n}a_{i}X_{i}\right|^{2}dA\ge0
\]
with equality when $\mathbf{a}=\mathbf{0}$ due to the construction of the basis.  So $M(t)$ is positive definite.

Now, we can prove that the system is solvable (at least) until $T_0=T$ for any $0<T\le \infty$. This follows from the energy balance, which ensures no blow-up occurs in $I=[0,T]$.
Indeed, by taking the scalar product between \eqref{eqn:system-ode} and $\left(a_{n}^{k}\right)^{\prime}$ for $k=\overline{1,n}$ we obtain that 
\begin{equation}\label{eqn:discrete-energy-balance}
   \frac{d}{dt}E_{n}\left(t\right)+\int_{\Omega^{\delta}\left(t\right)}\left|\mathbb{D}\mathbf{u}_{n}\right|^{2}dx+\frac{1}{\alpha}\int_{\omega}\left|\mathbf{u}_{n}\circ\phi_{\delta}-\partial_{t}\eta_{n}\mathbf{e}_{r}\right|^{2}J_{\delta}dA=\left\langle F\left(t\right),\mathbf{u}_{n}\right\rangle  
\end{equation}
where $E_n$ is the analogue of $E$ from \eqref{eqn:energy}.
And from \eqref{eqn:discrete-energy-balance} we conclude that
\begin{equation}\label{eqn:sup-n-En}
\sup_{n\ge1}\sup_{t\in\left(0,T\right)}E_{n}\left(t\right)<\infty.
\end{equation}

We multiply each of the $k$ equation from\eqref{eqn:system-ode} by a test function $\varphi\in C^{1}[0,T]$ and then integrate in time on $I=[0,T]$. By integration by parts we obtain that for every $1\le k \le n$ the following holds 
\begin{equation}
\begin{aligned}\int_{\Omega^{\delta}\left(t\right)}\mathbf{u}_{n}\cdot\left(\varphi\mathbf{X}_{k}\right)\left(t\right)dx+\int_{0}^{t}\int_{\Omega^{\delta}\left(s\right)}-\mathbf{u}_{n}\cdot\partial_{t}\left(\varphi\mathbf{X}_{k}\right)+\mathbb{D}\mathbf{u}_{n}:\mathbb{D}\left(\varphi\mathbf{X}_{k}\right)dxds & +\\
\int_{0}^{t}b\left(s,\mathbf{u}_{n},\mathbf{v},\varphi\mathbf{X}_{k}\right)ds+\frac{1}{2}\int_{0}^{t}\int_{\Gamma^{\delta}\left(s\right)}\left(\mathbf{u}_{n}\cdot\varphi\mathbf{X}_{k}\right)\left(\partial_{t}\delta\mathbf{e}_{r}\circ\phi_{\delta}^{-1}\right)\cdot\boldsymbol{\nu}^{\delta}dA_{\delta}ds & +\\
\frac{1}{\alpha}\int_{0}^{t}\int_{\omega}\left(\mathbf{u}_{n}\circ\phi_{\delta}-\partial_{t}\eta_{n}\mathbf{e}_{r}\right)\cdot\left(\varphi\mathbf{X}_{k}\circ\phi_{\delta}-\varphi X_{k}\mathbf{e}_{r}\right)J_{\delta}dAds & +\\
\int_{\omega}\partial_{t}\eta_{n}\cdot\left(\varphi X_{k}\right)\left(t\right)dA+\int_{0}^{t}\int_{\omega}-\partial_{t}\eta_{n}\cdot\partial_{t}\left(\varphi X_{k}\right)+\nabla^{2}\eta_{n}:\nabla^{2}\left(\varphi X_{k}\right)dAds & =\\
\int_{0}^{t}\left\langle F\left(s\right),\varphi\mathbf{X}_{k}\right\rangle ds+\int_{\Omega^{\delta}\left(0\right)}\mathbf{u}_{n}\left(0\right)\cdot\left(\varphi\mathbf{X}_{k}\left(0\right)\right)dx+\int_{\omega}\eta_{n}\left(0\right)\left(\varphi\left(0\right)X_{k}\right)dA.
\end{aligned}
\end{equation}

Then, for fixed $k$, we  pass to the limit as $n\to \infty$. 
Indeed, from \eqref{eqn:sup-n-En}  there exists a pair $\left(u,\eta \right)\in \mathcal{S}^{\delta}$ such that 
\begin{equation}
\left(\mathbf{u}_{n},\eta_{n}\right)\rightharpoonup\left(\mathbf{u},\eta\right)\quad\text{in}\ \mathcal{S^{\delta}}.
\end{equation}
This means we have, for a subsequence that we do not relabel, that
\begin{equation}
    \begin{aligned}\mathbf{u}_{n}\rightharpoonup\mathbf{u} & \quad\text{in}\ L^{\infty}\left(I;L^{2}\left(\Omega^{\delta}\left(t\right)\right)\right)\\
\mathbb{D}\mathbf{u}_{n}\rightharpoonup\overline{\mathbb{D}\mathbf{u}} & \quad\text{in}\ L^{2}\left(I;L^{2}\left(\Omega^{\delta}\left(t\right)\right)\right)\\
\eta_{n}\rightharpoonup\eta & \quad\text{in}\ L^{\infty}\left(I;H_{0}^{2}\left(\omega\right)\right)\\
\partial_{t}\eta_{n}\rightharpoonup\partial_{t}\eta & \quad\text{in}\ L^{\infty}\left(I;L^{2}\left(\omega\right)\right)
\end{aligned}
\end{equation}

 By $\overline{\mathbb{D}\mathbf{u}}$ we understand a weak limit for which we would like to have $\overline{\mathbb{D}\mathbf{u}}=\mathbb{D}\mathbf{u}$. However, with similar arguments as in \eqref{eqn:energy} we also obtain that 
 \begin{equation}
     \nabla\mathbf{u}_{n}\rightharpoonup\nabla\mathbf{u}\quad\text{in}\ L^{2}\left(I;L^{2-}\left(\Omega^{\delta}\left(t\right)\right)\right)
 \end{equation}
and hence $\overline{\mathbb{D}\mathbf{u}}=\mathbb{D}\mathbf{u}$ and also $\mathbf{u}_{n}\circ\phi_{\delta}\to\mathbf{u}\circ\phi_{\delta}$ in $L^{2}\left(I;L^{2}\left(\Gamma^{\delta}\left(t\right)\right)\right)$.

By weak lower semicontinuity we find that the analogue quantities of \eqref{eqn:energy}, namely $E_{\delta}$, $E_{slip,\delta}$, $D_{\delta}$
given by 
\begin{equation}\label{eqn:notation-energy-delta}
\begin{aligned}E_{\delta}\left(t\right):= & \frac{1}{2}\int_{\Omega^{\delta}\left(t\right)}\left|\mathbf{u}\left(t,x\right)\right|^{2}dx+\frac{1}{2}\int_{\omega}\left|\partial_{t}\eta\right|^{2}dA+\frac{1}{2}\int_{\omega}\left|\nabla^{2}\eta\right|^{2}dA\\
D_{\delta}\left(t\right):= & \int_{\Omega^{\delta}\left(t\right)}\left|\mathbb{D}\mathbf{u}\right|^{2}+\left|\nabla\mathbf{u}\right|^{2-}dx\\
E_{slip,\delta}\left(t\right):= & \frac{1}{\alpha}\int_{\omega}\left|\mathbf{u}\circ\phi_{\delta\left(t\right)}-\partial_{t}\eta\mathbf{e}_{r}\right|^{2}J_{\delta}\ dA.
\end{aligned}
\end{equation}
fulfill the relation
\begin{equation}\label{eqn:decoupled-energy-ineq}
    \sup_{t\in (0,T)}E_{\delta}\left(t\right)+\int_{0}^{T}E_{slip,\delta}\left(s\right)+D_{\delta}\left(s\right)ds\lesssim E\left(0\right)+\left\Vert P\right\Vert _{L_{t}^{2}}^{2}.
\end{equation}
which is similar to \ref{eqn:energy}.
With similar arguments as in \cite[p. 237]{LR14} we obtain that 
\[\text{span}\left\{ \left(\varphi\mathbf{X}_{k},\varphi X_{k}\right):\varphi\in C^{1}[0,T]),\ k\in\mathbb{N}_{\ge 1}\right\} \ \text{is dense in} \ \mathcal{T}^{\delta}.\] 

All these facts enable us to conclude that we proved the following:
\begin{proposition}\label{prop:dec-lin-problem}
For any given pair $
\left(\delta,\mathbf{v}\right)\in C^{\infty}\left(I\times\omega\right)\times C^{\infty}\left(I\times\mathbb{R}^{3}\right)$, there exists at least a corresponding  solution for the decoupled and linearized problem denoted by $\left(\mathbf{u},\eta\right)\in\mathcal{S}^{\delta}$.    Moreover, it fulfills the corresponding energy inequality \eqref{eqn:energy-ineq}.
\end{proposition}

Now, in order to relax the assumptions on $(\delta,\mathbf{v})$, let us considering the regularizing operator introduced in Subsection~\ref{ssec:regularization-operator} and  
denoted by $\mathcal{R}_{\varepsilon}$ for any $\varepsilon>0$ which maps 
\[
\mathcal{R}_{\varepsilon}:C^{0}\left(I\times\overline{\omega}\right)\times L^{2}\left(I\times\mathbb{R}^{3}\right)\mapsto C^{\infty}\left(I\times\omega\right)\times C^{\infty}\left(I\times\mathbb{R}^{3}\right).
\]
Then, from Proposition~\ref{prop:dec-lin-problem} we obtain the following
\begin{proposition}\label{prop:exist-dec-eps-reg}
Let $\left(\delta,\mathbf{v}\right)\in C^{0}\left(I\times\overline{\omega}\right)\times L^{2}\left(I\times\mathbb{R}^{3}\right)$.
For any $\varepsilon>0$ there exists at least one pair $\left(\mathbf{u}_{\varepsilon},\eta_{\varepsilon}\right)\in\mathcal{S}^{\mathcal{R}_{\varepsilon}\delta}$ which is a solution for the decoupled and linearized problem around $\left(\mathcal{R}_{\varepsilon}\delta,\mathcal{R}_{\varepsilon}\mathbf{v}\right)$. Furthermore, it fulfills the corresponding energy estimate
\begin{equation}\label{eqn:eps-energy-ineq}
\sup_{t\in(0,T)}E_{\mathcal{R}_{\varepsilon}\delta}\left(t\right)+\int_{0}^{T}E_{slip,\mathcal{R}_{\varepsilon}\delta}\left(s\right)+D_{\mathcal{R}_{\varepsilon}\delta}\left(s\right)ds\lesssim E\left(0\right)+\left\Vert P\right\Vert _{L_{t}^{2}}^{2}.
\end{equation}
The quantities in \eqref{eqn:eps-energy-ineq} are obtained from \eqref{eqn:notation-energy-delta}.
\end{proposition}

\subsection{The (set-valued) fixed point argument}\label{ssec:set-v-fixed-point}
We now aim to obtain a fixed-point for the mapping $\left(\delta,\mathbf{v}\right)\mapsto\left(\eta,\mathbf{u}\right)$. Since it may be multi-valued, we employ the following well-known result due to Kakutani-Gliksberg-Fan
which can be found in \cite{Granas-Dug}:
\begin{theorem}\label{thm:kakutani}
    Let $C$ be a convex subset of a normed vector space $Z$ and $F:C \mapsto \mathcal{P}(C)$ a set-valued mapping, where $\mathcal{P}(C)$ denotes the subsets of $C$. We assume the following:
    \begin{enumerate}[(i)]
        \item $F$ is upper-semicontinuous; or, equivalently, it has the \emph{closed-graph} property: if $c_n \to c$ and $z_n \in F(c_n)$ with $z_n \to z$ then $z \in F(c)$
        \item $F(C)$ is contained in a compact subset of $C$
        \item For all $z\in C$, $F(z)$ is non-empty, convex and compact.
    \end{enumerate}
 Then $F$ has a fixed point, that is there exists $c_0 \in C$ with $c_0 \in F(c_0)$.  
\end{theorem} 
Let $\varepsilon>0$ be a regularizing parameter.
In order to apply Theorem~\ref{thm:kakutani} we set the normed space 
\[Z:=C^{0}\left(\overline{I};C\left(\overline{\omega}\right)\right)\times L^{2}\left(I\times\mathbb{R}^{3}\right).\]

Assuming that $\left\Vert \eta_{0}\right\Vert _{L_{t,x}^{\infty}}<R/2$, we choose the  convex subset  \[
D:=\left\{ \left(\delta,\mathbf{v}\right)\in Z:\delta\left(0,\cdot\right)=\eta_{0},\ \left\Vert \delta\right\Vert _{L_{t,x}^{\infty}}\le\alpha:=\frac{R+\left\Vert \eta_{0}\right\Vert _{L_{x}^{\infty}}}{2},\ \left\Vert \mathbf{v}\right\Vert _{L_{t}^{2}L_{x}^{2}}\le C_{1}\right\} 
\]
where the constant $C_1$ can be chosen so that $\left\Vert \mathbf{u}\right\Vert _{L_{t}^{2}L_{x}^{2}} \le C_1$ in accordance to \eqref{eqn:energy}.

We define the mapping 
\[F:D \subset Z \mapsto \mathcal{P}(Z)\]
as follows: 
for $\left(\delta,\mathbf{v}\right)\in D$ let $F\left(\delta,\mathbf{v}\right)$ be the set of all $\left(\eta,\mathbf{u}\right)\in\mathcal{S}^{\mathcal{R}_{\varepsilon}\delta}$ which are solutions to the decoupled and linearized problem around $\left(\mathcal{R}_{\varepsilon}\delta,\mathcal{R}_{\varepsilon}\mathbf{v}\right)$. Due to Proposition~\ref{prop:existence-eps-reg-soln} we see that this definition is meaningful. 
The existence of the fixed-point follows the same lines as of \cite[Section 3.3]{LR14} or \cite[Section 5.2]{mindrila-roy-multuilayered}.
Let us observe that due to \eqref{eqn:decoupled-energy-ineq} and the fact that $V_{s}\hookrightarrow C^{0,1-\theta}\left(\left(0,T\right);C^{0,2\theta-1}\left(\omega\right)\right)$ for any $\theta \in (1/2,1)$ we can obtain  that \begin{equation}
\begin{aligned}\left|\eta\left(t,x\right)\right| & \le\left|\eta\left(t,x\right)-\eta_{0}\left(x\right)\right|+\left|\eta_{0}\left(x\right)\right|\\
 & \lesssim T^{1-\theta}+\left\Vert \eta_{0}\right\Vert _{L_{x}^{\infty}}<T^{1-\theta}+R/2<R
\end{aligned}
\end{equation} 
provided we choose $T$ sufficiently small. 
Let us check the conditions of Theorem~\ref{thm:kakutani}:
\begin{itemize}
    \item For any $\left(\delta,\mathbf{v}\right)\in D$ we see that $F(\delta, \mathbf{v})$ is non-empty, convex (since the problem is linearized), and compact due to Proposition~\ref{prop:compactness}.
    \item The fact that $F\left(D\right)\subset D$ is compact follows again from Propositon~\ref{prop:compactness}, with the same arguments, even simplified due to the regularized geometry $\Omega^{\mathcal{R}_{\varepsilon}\delta}(t)$.
    \item In order to prove that $F$ is upper-semicontinuous, let $\left(\delta_{n},\mathbf{v}_{n}\right)\to\left(\delta,\mathbf{v}\right)$ in $Z$ and let $\left(\eta_{n},\mathbf{u}_{n}\right)\in F\left(\delta_{n},\mathbf{v}_{n}\right)$ with $\left(\eta_{n},\mathbf{u}_{n}\right)\to\left(\eta,\mathbf{u}\right)$ in $Z$. We have to prove that $\left(\eta,\mathbf{u}\right)\in F\left(\delta,\mathbf{v}\right)$. Due to the energy estimates \eqref{eqn:decoupled-energy-ineq} and the compact embedding $F(D)\subset D$ we have that 
    \begin{equation}
        \begin{aligned}\eta_{n}\rightharpoonup\eta & \quad\text{in}\ L^{\infty}\left(I;H_{0}^{2}\left(\omega\right)\right)\\
\partial_{t}\eta_{n}\to\partial_{t}\eta & \quad \text{in}\ L^{2}\left(I;L^{2}\left(\omega\right)\right)\\
\mathbf{u}_{n}\chi_{\mathcal{R}_{\varepsilon}\delta_{n}}\to\mathbf{u}\chi_{\mathcal{R}_{\varepsilon}\delta} & \quad \text{in}\ L^{2}\left(I;L^{2}\left(\mathbb{R}^{3}\right)\right)\\
\mathbf{u}_{n}\chi_{\mathcal{R}_{\varepsilon}\delta_{n}}\rightharpoonup\mathbf{u}\chi_{\mathcal{R}_{\varepsilon}\delta} &\quad \text{in}\ L^{2}\left(I;W^{1,2-}\left(\mathbb{R}^{3}\right)\cap E^{1,2}\left(\mathbb{R}^{3}\right)\right)\\
\mathbf{u}_{n}\circ\phi_{\mathcal{R}_{\varepsilon}\delta_{n}}-\partial_{t}\eta_{n}\mathbf{e}_{r}\to L&\quad\text{in}\ L^{2}\left(I;L^{2}\left(\omega\right)\right)
\end{aligned}
    \end{equation}
\end{itemize}

In order to prove that $\mathbf{u}\circ\phi_{\mathcal{R}_{\varepsilon}\delta}\cdot\nu^{\mathcal{R}_{\varepsilon}\delta}=\partial_{t}\eta\mathbf{e}_{r}\cdot\nu^{\mathcal{R}_{\varepsilon}\delta}$, since $\mathbf{u}_{n}\circ\phi_{\mathcal{R}_{\varepsilon}\delta_{n}}\cdot\nu^{\mathcal{R}_{\varepsilon}\delta_{n}}=\partial_{t}\eta_{n}\mathbf{e}_{r}\cdot\nu^{\mathcal{R}_{\varepsilon}\delta_{n}}$ and, due to $\mathcal{R}_{\varepsilon}$ we have that $\boldsymbol{\nu}^{\mathcal{R}_{\varepsilon}\delta_{n}}\to\boldsymbol{\nu}^{\mathcal{R}_{\varepsilon}\delta}$ pointwisely and uniformly, it suffices to prove that $\mathbf{u}_{n}\circ\phi_{\mathcal{R}_{\varepsilon}\delta_{n}}\to\mathbf{u}\circ\phi_{\mathcal{R}_{\varepsilon}\delta}$ in $L^{2}_{t,x}$. For this, we write 
\[
\mathbf{u}_{n}\circ\phi_{\mathcal{R}_{\varepsilon}\delta_{n}}-\mathbf{u}\circ\phi_{\mathcal{R}_{\varepsilon}\delta}=\mathbf{u}_{n}\circ\phi_{\mathcal{R}_{\varepsilon}\delta_{n}}-\mathbf{u}_{n}\circ\phi_{\mathcal{R}_{\varepsilon}\delta}+\mathbf{u}_{n}\circ\phi_{\mathcal{R}_{\varepsilon}\delta}-\mathbf{u}\circ\phi_{\mathcal{R}_{\varepsilon}\delta}.
\] The first summand is $\int_{0}^{1}\frac{d}{ds}\mathbf{u}_{n}\circ\left(s\phi_{\mathcal{R}_{\varepsilon}\delta_{n}}+\left(1-s\right)\phi_{\mathcal{R}_{\varepsilon}\delta}\right)ds$ which can be bounded with the estimate $\left\Vert \mathbf{u}_{n}\right\Vert _{L_{t}^{2}W_{x}^{1,2-}}$ and converges to zero, while for the second we use that $\phi_{\mathcal{R}_{\varepsilon}\delta_{n}}\to\phi_{\mathcal{R}_{\varepsilon}\delta}$ pointwisely and uniformly. In particular one can conclude that $L=\mathbf{u}\circ\phi_{\mathcal{R}_{\varepsilon}\delta}-\partial_{t}\eta\mathbf{e}_{r}$.

Let us observe that the weak formulation 
\begin{equation}\label{eqn:weak-form-approx-n}
\begin{aligned}\int_{\Omega^{\mathcal{R}_{\varepsilon}\delta_{n}}\left(t\right)}\mathbf{u}_{n}\cdot\mathbf{q}\left(t\right)dx+\int_{0}^{t}\int_{\Omega^{\mathcal{R}_{\varepsilon}\delta_{n}}\left(s\right)}-\mathbf{u}_{n}\cdot\partial_{t}\mathbf{q}+\mathbb{D}\mathbf{u}_{n}:\mathbb{D}\mathbf{q}dxds & +\\
\int_{0}^{t}b\left(s,\mathbf{u}_{n},\mathcal{R}_{\varepsilon}\mathbf{v}_{n},\mathbf{q}\right)ds-\frac{1}{2}\int_{0}^{t}\int_{\Gamma^{\mathcal{R}_{\varepsilon}\delta_{n}}\left(s\right)}\left(\mathbf{u}_{n}\cdot\mathbf{q}\right)\left(\partial_{t}\mathcal{R}_{\varepsilon}\delta_{n}\mathbf{e}_{r}\circ\phi_{\mathcal{R}_{\varepsilon}\delta_{n}}^{-1}\right)\cdot\boldsymbol{\nu}^{\mathcal{R}_{\varepsilon}\delta_{n}}dA_{\delta}ds & +\\
\frac{1}{\alpha}\int_{0}^{t}\int_{\omega}\left(\mathbf{u}_{n}\circ\phi_{\mathcal{R}_{\varepsilon}\delta_{n}}-\partial_{t}\eta\mathbf{e}_{r}\right)\cdot\left(\mathbf{q}\circ\phi_{\mathcal{R}_{\varepsilon}\delta_{n}}-\xi\mathbf{e}_{r}\right)J_{\mathcal{R}_{\varepsilon}\delta_{n}}dAds & +\\
\int_{\omega}\partial_{t}\eta_{n}\cdot\partial_{t}\xi\left(t\right)dA+\int_{0}^{t}\int_{\omega}-\partial_{t}\eta_{n}\cdot\partial_{t}\xi+\nabla^{2}\eta_{n}:\nabla^{2}\xi dAds & =\\
\int_{0}^{t}\left\langle F\left(t\right),\mathbf{q}\right\rangle ds+\int_{\Omega^{\mathcal{R}_{\varepsilon}\delta_{n}}\left(0\right)}\mathbf{u}_{0}\cdot\mathbf{q}\left(0\right)dx+\int_{\omega}\eta_{1}\xi dA.
\end{aligned}
\end{equation}
is valid for all test functions $\left(\mathbf{q},\xi\right)=\left(\mathbf{q}_{n},\xi_{n}\right)\in\mathcal{T}^{\mathcal{R}_{\varepsilon}\delta_{n}}$. The only problem in letting $n\to \infty$ is that  the limiting weak formulation has to be valid for all the test functions $\left(\mathbf{q},\xi\right)\in\mathcal{T}^{\mathcal{R}_{\varepsilon}\delta}$. 
To this end, let us use in \eqref{eqn:weak-form-approx-n} the pair $\left(\mathcal{F}_{\mathcal{R}_{\varepsilon}\delta_{n}}^{s}\left(\xi\right),\xi\right)$, where the operator $\mathcal{F}^{s}$ was introduced in Proposition~\ref{prop:extension-slip}. To observe that $\mathcal{F}_{\mathcal{R}_{\varepsilon}\delta_{n}}^{s}\left(\xi\right)\to\mathcal{F}_{\mathcal{R}_{\varepsilon}\delta}^{s}\left(\xi\right)$ as $n\to \infty$, in the appropriate spaces, we see that the regularizing effect $\mathcal{R}_{\varepsilon}$ simplifies the argument. Thus, for $\left(\mathbf{q},\xi\right)=\left(\mathcal{F}_{\mathcal{R}_{\varepsilon}\delta}^{s}\left(\xi\right),\xi\right)\in\mathcal{T^{\mathcal{R}_{\varepsilon}\delta}}$, we obtain that
\begin{equation}\label{eqn:weak-form-approx-limit}
\begin{aligned}\int_{\Omega^{\mathcal{R}_{\varepsilon}\delta}\left(t\right)}\mathbf{u}\cdot\mathbf{q}\left(t\right)dx+\int_{0}^{t}\int_{\Omega^{\mathcal{R}_{\varepsilon}\delta}\left(s\right)}-\mathbf{u}\cdot\partial_{t}\mathbf{q}+\mathbb{D}\mathbf{u}:\mathbb{D}\mathbf{q}dxds & +\\
\int_{0}^{t}b\left(t,\mathbf{u},\mathcal{R}_{\varepsilon}\mathbf{v},\mathbf{q}\right)ds & +\\
-\frac{1}{2}\int_{0}^{t}\int_{\Gamma^{\mathcal{R}_{\varepsilon}\delta}\left(s\right)}\left(\mathbf{u}_{n}\cdot\mathbf{q}\right)\left(\partial_{t}\mathcal{R}_{\varepsilon}\delta\mathbf{e}_{r}\circ\phi_{\mathcal{R}_{\varepsilon}\delta}^{-1}\right)\cdot\boldsymbol{\nu}^{\mathcal{R}_{\varepsilon}\delta}dA_{\delta}ds & +\\
\frac{1}{\alpha}\int_{0}^{t}\int_{\omega}\left(\mathbf{u}\circ\phi_{\mathcal{R}_{\varepsilon}\delta}-\partial_{t}\eta\mathbf{e}_{r}\right)\cdot\left(\mathbf{q}\circ\phi_{\mathcal{R}_{\varepsilon}\delta}-\xi\mathbf{e}_{r}\right)J_{\mathcal{R}_{\varepsilon}\delta}dAds & +\\
\int_{\omega}\partial_{t}\eta\cdot\partial_{t}\xi\left(t\right)dA+\int_{0}^{t}\int_{\omega}-\partial_{t}\eta\cdot\partial_{t}\xi+\nabla^{2}\eta_{n}:\nabla^{2}\xi dAds & =\\
\int_{0}^{t}\left\langle F\left(t\right),\mathbf{q}\right\rangle ds+\int_{\Omega^{\mathcal{R}_{\varepsilon}\delta}\left(0\right)}\mathbf{u}_{0}\cdot\mathbf{q}\left(0\right)dx+\int_{\omega}\eta_{1}\xi dA.
\end{aligned}
\end{equation}

We are left to prove that we may take $\left(\mathbf{q}-\mathcal{F}_{\mathcal{R}_{\varepsilon}\delta}^{s}\left(\xi\right),0\right)\in\mathcal{T^{\mathcal{R}_{\varepsilon}\delta}}$ in \eqref{eqn:weak-form-approx-limit},  for an arbitrary $\left(\mathbf{q},\xi\right)\in\mathcal{T^{\mathcal{R}_{\varepsilon}\delta}}$.
To this end approximate $\mathbf{q}$ by $\mathbf{q}_{n}:=\mathcal{J}_{\mathcal{R}_{\varepsilon}\delta_{n}}\mathcal{J}_{\mathcal{R}_{\varepsilon}\delta}^{-1}\left(\mathbf{q}\right)$ and use $\left(\mathbf{q}_{n},0\right)\in\mathcal{T}^{\mathcal{R}_{\varepsilon}\delta_{n}}$ in \eqref{eqn:weak-form-approx-limit} and let $n\to \infty$. In this way, the mapping $F$ has also closed graph.

Thus, we can  apply Theorem~\ref{thm:kakutani} and we obtain the following:

\begin{proposition}\label{prop:existence-eps-reg-soln} 
There exists a time $T_{\star}>0$ such that for any $0<T<T_{\star}$ and for any $\varepsilon>0$, there exists at least one pair $\left(\mathbf{u}_{\varepsilon},\eta_{\varepsilon}\right)\in\mathcal{S}^{\mathcal{R}_{\varepsilon}\eta_{\varepsilon}}$ with the following property: for almost every $t\in I=[0,T]$ and for all $\left(\mathbf{q},\xi\right)\in\mathcal{T}^{\mathcal{R}_{\varepsilon}\eta_{\varepsilon}}$, it holds that
 \begin{equation}\label{eqn:eps-weak-formulation}
\begin{aligned}\int_{\Omega^{\mathcal{R}_{\varepsilon}\eta_{\varepsilon}}\left(t\right)}\mathbf{u}_{\varepsilon}\cdot\mathbf{q}\left(t\right)+\int_{0}^{t}\int_{\Omega^{\mathcal{R}_{\varepsilon}\eta_{\varepsilon}}\left(t\right)}-\mathbf{u}_{\varepsilon}\cdot\partial_{t}\mathbf{q}+\mathbb{D}\mathbf{u}_{\varepsilon}:\mathbb{D}\mathbf{q}dxds & +\\
\int_{0}^{t}b\left(s,\mathbf{u}_{\varepsilon},\mathcal{R}_{\varepsilon}\mathbf{u}_{\varepsilon},\mathbf{q}\right)ds & +\\
\int_{0}^{t}\int_{\Gamma^{\mathcal{R}_{\varepsilon}\eta_{\varepsilon}}\left(s\right)}-\frac{1}{2}\left(\mathbf{u}_{\varepsilon}\cdot\mathbf{q}\right)\left(\partial_{t}\mathcal{R}_{\varepsilon}\eta_{\varepsilon}\mathbf{e}_{r}\circ\phi_{\mathcal{R}_{\varepsilon}\eta_{\varepsilon}}^{-1}\right)\cdot\boldsymbol{\nu}^{\mathcal{R}_{\varepsilon}\eta_{\varepsilon}}dA_{\mathcal{R}_{\varepsilon}\eta_{\varepsilon}}ds & +\\
\frac{1}{\alpha}\int_{0}^{t}\int_{\omega}\left(\mathbf{u}_{\varepsilon}\circ\phi_{\mathcal{R}_{\varepsilon}\eta_{\varepsilon}}-\partial_{t}\eta_{\varepsilon}\mathbf{e}_{r}\right)\cdot\left(\mathbf{q}\circ\phi_{\mathcal{R}_{\varepsilon}\eta_{\varepsilon}}-\xi\mathbf{e}_{r}\right)J_{\mathcal{R}_{\varepsilon}\eta_{\varepsilon}}dAds & +\\
\int_{\omega}\partial_{t}\eta_{\varepsilon}\cdot\xi\left(t\right)dA+\int_{0}^{t}\int_{\omega}-\partial_{t}\eta_{\varepsilon}\cdot\partial_{t}\xi+\nabla^{2}\eta_{\varepsilon}:\nabla^{2}\xi dAds & =\\
\int_{0}^{t}\left\langle F\left(t\right),\mathbf{q}\right\rangle ds+\int_{\Omega^{\mathcal{R}_{\varepsilon}\eta_{\varepsilon}}\left(0\right)}\mathbf{u}_{0}\cdot\mathbf{q}\left(0\right)dx+\int_{\omega}\eta_{1}\xi\left(0\right)dA.
\end{aligned}
    \end{equation}

In addition, we have the energy estimate 
\begin{equation}\label{eqn:eps-en-estimate}
\sup_{t\in(0,T)}E_{\mathcal{R}_{\varepsilon}\eta_{\varepsilon}}\left(t\right)+\int_{0}^{T}E_{slip,\mathcal{R}_{\varepsilon}\eta_{\varepsilon}}\left(s\right)+D_{\mathcal{R}_{\varepsilon}\eta_{\varepsilon}}\left(s\right)ds\lesssim E\left(0\right)+\left\Vert P\right\Vert _{L_{t}^{2}}^{2}
\end{equation}
where the quantities are the analogues of \eqref{eqn:notation-energy-delta}.
\end{proposition}

\subsection{Limit passage}\label{ssec:eps-limit-passage}

The final step towards the proof of Theorem~\ref{thm:main} is to perform the limit passage $\varepsilon \to 0$ in Proposition~\ref{prop:existence-eps-reg-soln}.
From \eqref{eqn:eps-en-estimate} we infer the existence of a pair $\left(\mathbf{u},\eta\right)$ and of a subsequence of $\left(\mathbf{u}_{\varepsilon},\eta_{\varepsilon}\right)_{\varepsilon>0}$ (which we do not relabel) such that 
\begin{equation}\label{eqn:weak-limit-eps}
\begin{aligned}\mathbf{u}_{\varepsilon}\chi_{\Omega^{\mathcal{R}_{\varepsilon}\eta_{\varepsilon}}}\rightharpoonup\  & \mathbf{u}\quad\text{in}\ L^{\infty}\left(I;L^{2}\left(\mathbb{R}^{3}\right)\right)\\
\mathbb{D}\mathbf{u}_{\varepsilon}\chi_{\Omega^{\mathcal{R}_{\varepsilon}\eta_{\varepsilon}}}\rightharpoonup\  & \overline{\mathbb{D}\mathbf{u}}\quad\text{in}\ L^{2}\left(I;L^{2}\left(\mathbb{R}^{3}\right)\right)\\
\nabla\mathbf{u}_{\varepsilon}\chi_{\Omega^{\mathcal{R}_{\varepsilon}\eta_{\varepsilon}}}\rightharpoonup\  & \nabla\mathbf{u}\quad\text{in}\ L^{2}\left(I;L^{2-}\left(\mathbb{R}^{3}\right)\right)\\
\eta_{\varepsilon},\mathcal{R}_{\varepsilon}\eta_{\varepsilon}\rightharpoonup\  & \eta\quad\text{in}\ L^{\infty}\left(I;H_{0}^{2}\left(\omega\right)\right)\\
\partial_{t}\eta_{\varepsilon}\rightharpoonup\  & \partial_{t}\eta\quad\text{in}\ W^{1,\infty}\left(I;L^{2}\left(\omega\right)\right)\\
\mathbf{u}_{\varepsilon}\circ\phi_{\mathcal{R}_{\varepsilon}\eta_{\varepsilon}}-\partial_{t}\eta_{\varepsilon}\mathbf{e}_{r}\rightharpoonup\  & \mathbf{u}_{}\circ\phi_{\eta}-\partial_{t}\eta\mathbf{e}_{r}\quad\text{in}\ L^{2}\left(I;L^{2}\left(\omega\right)\right)
\end{aligned}
\end{equation}
where $\overline{\mathbb{D}\mathbf{u}}$ represents only a notation for the weak limit. To prove that $\overline{\mathbb{D}\mathbf{u}}=\mathbb{D}\mathbf{u}$ we use the weak convergence of the gradients from \eqref{eqn:weak-limit-eps} to obtain  $\mathbb{D}\mathbf{u}_{\varepsilon}\chi_{\Omega^{\mathcal{R}_{\varepsilon}\eta_{\varepsilon}}}\rightharpoonup\mathbb{D}\mathbf{u}\quad\text{in}\ L_{t}^{2}L^{2-}\left(\mathbb{R}^{3}\right)$.

The convergence of the traces is due to Lemma~\ref{lm:trace}.

We now recall Proposition~\ref{prop:compactness} to see that with an analogous argument we can conclude that 
\begin{equation}\label{eqn:strong-limit-eps}
  \begin{aligned}\mathbf{u}_{\varepsilon}\chi_{\Omega^{\mathcal{R}_{\varepsilon}\eta_{\varepsilon}}}\to & \ \mathbf{u}\quad\text{in}\ L^{2}\left(I;L^{2}\left(\mathbb{R}^{3}\right)\right)\\
\partial_{t}\eta_{\varepsilon}\to & \ \partial_{t}\eta\quad\text{in}\ L^{2}\left(I;L^{2}\left(\omega\right)\right)
\end{aligned}
\end{equation}

Since $\boldsymbol{\nu}^{\mathcal{R}_{\varepsilon}\eta_{\varepsilon}}\to\boldsymbol{\nu}^{\eta}$ in $L^{\infty}\left(I;W^{1,2}\left(\omega\right)\right)$, we can also pass to the weak limits in
\[\left(\mathbf{u}_{\varepsilon}\circ\phi_{\mathcal{R}_{\varepsilon}\eta_{\varepsilon}}\right)\cdot\boldsymbol{\nu}^{\mathcal{R}_{\varepsilon}\eta_{\varepsilon}}=\partial_{t}\eta_{\varepsilon}\mathbf{e}_{r}\cdot\boldsymbol{\nu}^{\mathcal{R}_{\varepsilon}\eta_{\varepsilon}}\]

to obtain that 
\begin{equation}
\left(\mathbf{u}\circ\phi_{\eta}\right)\cdot\boldsymbol{\nu}^{\eta}=\partial_{t}\eta\mathbf{e}_{r}\cdot\boldsymbol{\nu}^{\eta}\quad\text{a.e. in }I\times\omega.
\end{equation}

By letting $\varepsilon \to 0$ in \eqref{eqn:eps-weak-formulation} we see that we can identify the correct limits, but, as in the proof of Proposition~\ref{prop:existence-eps-reg-soln}, \emph{the limit formulation has to be valid for test functions $\left(\mathbf{q},\xi\right)\in\mathcal{T}^{\eta}$}. To this end, we use in \eqref{eqn:eps-weak-formulation} the test function
\[
\left(\mathbf{q},\xi\right):=\left(\mathcal{F}_{\mathcal{R}_{\varepsilon}\eta_{\varepsilon}}^{s}\left(\xi\right),\xi\right)\in\mathcal{T}^{\mathcal{R}_{\varepsilon}\eta_{\varepsilon}}. 
\]
However, we observe that for the convergence \[\int_{I}\int_{\Omega^{\mathcal{R}_{\varepsilon}\eta_{\varepsilon}}}\mathbb{D}\mathbf{u}_{\varepsilon}:\mathbb{D}\mathcal{F}_{\mathcal{R}_{\varepsilon}\eta_{\varepsilon}}^{s}\left(\xi\right)\to\int_{I}\int_{\Omega^{\eta}}\mathbb{D}\mathbf{u}_{}:\mathbb{D}\mathcal{F}_{\eta}^{s}\left(\xi\right)\]
since we merely have $\mathbb{D}\mathbf{u}_{\varepsilon}\rightharpoonup\mathbb{D}\mathbf{u}$ in $L^{2}_{t,x}$ we would need the strong convergence $\mathbb{D}\mathcal{F}_{\mathcal{R}_{\varepsilon}\eta_{\varepsilon}}^{s}\left(\xi\right)\to\mathbb{D}\mathcal{F}_{\eta}^{s}\left(\xi\right)$ in $L^{2}_{t,x}$. This suggests the strong convergence of the second gradients of $\eta_{\varepsilon}$, which is not guaranteed by using simply energy estimates. We need to use the equation.

In order to complete the proof, we need to establish

\begin{proposition}\label{prop:2nd-gradients-convergence}
    It holds that
\begin{equation}\label{eqn:claim-2nd-gradients-eta}
    \nabla^{2}\eta_{\varepsilon}\to\nabla^{2}\eta\quad\text{in}\ L^{2}\left(I;L^{2}\left(\omega\right)\right)
\end{equation}
\end{proposition} 
\begin{proof}
Since we already enjoy the weak convergence and we work in a Hilbert space, it suffices to prove that we have norm convergence, that is 
\begin{equation}\label{eqn:norm-2nd-grad}
\int_{I}\int_{\omega}\left|\nabla^{2}\eta_{\varepsilon}\right|dAds\to\int_{I}\int_{\omega}\left|\nabla^{2}\eta\right|dAds.
\end{equation}

To prove \eqref{eqn:norm-2nd-grad} let us observe that we can use the extension operator $\mathcal{F}_{\eta}$ from Proposition~\ref{prop:extension-no-slip} and use 
the test function
$\left(\mathbf{q},\xi\right)=\left(\mathcal{F}_{\mathcal{R}_{\varepsilon}\eta_{\varepsilon}}\left(\xi\right),\xi\right)
$ in \eqref{eqn:eps-weak-formulation}. 
This time, due to the better estimates of $\mathcal{F}$ compared to $\mathcal{F}^{s}$ we can let $\varepsilon \to 0$ to obtain
\begin{equation}\label{eqn:weak-limit-eps-1}
\begin{aligned}\int_{\Omega^{\eta}\left(t\right)}\mathbf{u}\cdot\mathcal{F}_{\eta}\left(\xi\right)\left(t\right)dx+\int_{0}^{t}\int_{\Omega^{\eta}\left(s\right)}-\mathbf{u}\cdot\partial_{t}\mathbf{\mathcal{F}_{\eta}\left(\xi\right)}+\mathbb{D}\mathbf{u}:\mathbb{D}\mathcal{F}_{\eta}\left(\xi\right)dxds & +\\
\int_{0}^{t}b\left(t,\mathbf{u},\mathbf{u},\mathcal{F}_{\eta}\left(\xi\right)\right)dt-\frac{1}{2}\int_{0}^{t}\int_{\Gamma^{\eta}\left(s\right)}\left(\mathbf{u}\cdot\mathcal{F}_{\eta}\left(\xi\right)\right)\left(\partial_{t}\eta\mathbf{e}_{r}\circ\phi_{\eta\left(t\right)}^{-1}\right)\cdot\boldsymbol{\nu}^{\eta}dA_{\eta}ds & +\\
\int_{\omega}\partial_{t}\eta\cdot\xi\left(t\right)dA+\int_{0}^{t}\int_{\omega}-\partial_{t}\eta\cdot\partial_{t}\xi+\nabla^{2}\eta:\nabla^{2}\xi dAds & =\\
\int_{0}^{t}\left\langle F\left(s\right),\mathcal{F}_{\eta}\left(\xi\right)\right\rangle ds+\int_{\Omega^{\eta}\left(0\right)}\mathbf{u}_{0}\cdot\mathcal{F}_{\eta}\left(\xi\right)\left(0\right)dx+\int_{\omega}\eta_{1}\xi\left(0\right)dA.
\end{aligned}
\end{equation}

In particular, taking  $\xi =\eta$ in \eqref{eqn:weak-limit-eps-1} provides
\begin{equation}\label{eqn:2nd-grad-eta}
\begin{aligned}-\int_{0}^{t}\int_{\omega}\left|\nabla^{2}\eta\right|^{2}= & \int_{\Omega^{\eta}\left(t\right)}\mathbf{u}\cdot\mathcal{F}_{\eta}\left(\eta\right)\left(t\right)+\int_{0}^{t}\int_{\Omega^{\eta}\left(s\right)}-\mathbf{u}\cdot\partial_{t}\mathbf{\mathcal{F}_{\eta}\left(\eta\right)}+\mathbb{D}\mathbf{u}:\mathbb{D}\mathcal{F}_{\eta}\left(\eta\right)+\\
 & \int_{0}^{t}b\left(s,\mathbf{u},\mathbf{u},\mathcal{F}_{\eta}\left(\eta\right)\right)-\frac{1}{2}\int_{0}^{t}\int_{\Gamma^{\eta}\left(s\right)}\left(\mathbf{u}\cdot\mathcal{F}_{\eta}\left(\eta\right)\right)\left(\partial_{t}\eta\mathbf{e}_{r}\circ\phi_{\eta}^{-1}\right)\cdot\boldsymbol{\nu}^{\eta}+\\
 & \int_{\omega}\partial_{t}\eta\cdot\eta\left(t\right)-\int_{0}^{t}\int_{\omega}\left(\partial_{t}\eta\right)^{2}-\int_{0}^{t}\left\langle F\left(s\right),\mathcal{F}_{\eta}\left(\eta\right)\right\rangle -\int_{\Omega^{\eta}\left(0\right)}\mathbf{u}_{0}\cdot\mathcal{F}_{\eta}\left(\eta\right)\left(0\right)-\int_{\omega}\eta_{1}\eta_{0}
\end{aligned}
\end{equation}
An analogous equality for $-\int_{0}^{t}\int_{\omega}\left|\nabla^{2}\eta_{\varepsilon}\right|$  is obtained 
by taking $\left(\mathbf{q},\xi\right)=\left(\mathcal{F}_{\mathcal{R}_{\varepsilon}\eta_{\varepsilon}}\left(\eta_{\varepsilon}\right),\eta_{\varepsilon}\right)$ in \eqref{eqn:eps-weak-formulation}. We need to prove the convergence of each corresponding term. We see that 
\begin{equation}
\int_{\Omega^{\mathcal{R}_{\varepsilon}\eta_{\varepsilon}}\left(t\right)}\mathbf{u}_{\varepsilon}\cdot\mathcal{F}_{\mathcal{R}_{\varepsilon}\eta_{\varepsilon}}\left(\eta_{\varepsilon}\right)\left(t\right)dx\to\int_{\Omega^{\eta}\left(t\right)}\mathbf{u}\cdot\mathcal{F}_{\eta}\left(\eta\right)\left(t\right)dx
\end{equation}
for a.e. $t\in I$ due to $\mathbf{u}_{\varepsilon}\rightharpoonup\mathbf{u}$ in $L_{t}^{\infty}L_{x}^{2}$ and $\mathcal{F}_{\mathcal{R}_{\varepsilon}\eta_{\varepsilon}}\left(\eta_{\varepsilon}\right)\to\mathcal{F}_{\eta}\left(\eta\right)$ in $C^{0}_{t,x}$.
Then \begin{equation}
\int_{I}\int_{\Omega^{\mathcal{R}_{\varepsilon}\eta_{\varepsilon}}\left(t\right)}\mathbf{u}_{\varepsilon}\cdot\partial_{t}\mathcal{F}_{\mathcal{R}_{\varepsilon}\eta_{\varepsilon}}\left(\eta_{\varepsilon}\right)dxdt\to\int_{I}\int_{\Omega^{\eta}\left(t\right)}\mathbf{u}\cdot\partial_{t}\mathcal{F}_{\eta}\left(\eta\right)dxdt
\end{equation}
can also be obtained by recalling Proposition~\ref{prop:extension-no-slip}. We also have that 
$\nabla\mathcal{F}_{\mathcal{R}_{\varepsilon}\eta_{\varepsilon}}\left(\eta_{\varepsilon}\right)\to\nabla\mathcal{F}_{\eta}\left(\eta\right)$ in $L_{t}^{\infty}W_{x}^{1,2}$. We can pass to the limit in each $\varepsilon$- term. This proves \eqref{eqn:norm-2nd-grad} and Proposition~\ref{prop:2nd-gradients-convergence}.
\end{proof}

Resuming now to the limiting procedure
we take
$\left(\mathbf{q},\xi\right):=\left(\mathcal{F}_{\mathcal{R}_{\varepsilon}\eta_{\varepsilon}}^{s}\left(\xi\right),\xi\right)\in\mathcal{T}^{\mathcal{R}_{\varepsilon}\eta_{\varepsilon}}$
in \eqref{eqn:eps-weak-formulation} and let $\varepsilon \to 0$
to obtain 
\begin{equation}\label{eqn:lim-weak-form}
\begin{aligned}\int_{\Omega^{\eta}\left(t\right)}\mathbf{u}\cdot\mathcal{F}_{\eta}^{s}\left(\xi\right)\left(t\right)dx+\int_{0}^{t}\int_{\Omega^{\eta}\left(s\right)}-\mathbf{u}\cdot\partial_{t}\mathcal{F}_{\eta}^{s}\left(\xi\right)+\mathbb{D}\mathbf{u}:\mathbb{D}\mathcal{F}_{\eta}^{s}\left(\xi\right)dxds & +\\
\int_{0}^{t}b\left(s,\mathbf{u},\mathbf{u},\mathcal{F}_{\eta}^{s}\left(\xi\right)\right)ds-\frac{1}{2}\int_{0}^{t}\int_{\Gamma^{\eta}\left(t\right)}\left(\mathbf{u}\cdot\mathcal{F}_{\eta}^{s}\left(\xi\right)\right)\left(\partial_{t}\eta\mathbf{e}_{r}\circ\phi_{\eta_{\varepsilon}\left(t\right)}^{-1}\right)\cdot\boldsymbol{\nu}^{\eta}dA_{\eta}ds & +\\
\frac{1}{\alpha}\int_{0}^{t}\int_{\omega}\left(\mathbf{u}_{\varepsilon}\circ\phi_{\eta}-\partial_{t}\eta\mathbf{e}_{r}\right)\cdot\left(\mathcal{F}_{\eta}^{s}\left(\xi\right)\circ\phi_{\eta}-\xi\mathbf{e}_{r}\right)J_{\eta}dAds & +\\
\int_{\omega}\partial_{t}\eta\cdot\xi\left(t\right)dA+\int_{0}^{t}\int_{\omega}-\partial_{t}\eta\cdot\partial_{t}\xi+\nabla^{2}\eta:\nabla^{2}\xi dAds & =\\
\int_{0}^{t}\left\langle F\left(t\right),\mathcal{F}_{\eta}^{s}\left(\xi\right)\right\rangle ds+\int_{\Omega^{\mathcal{R}_{\varepsilon}\eta_{\varepsilon}}\left(0\right)}\mathbf{u}_{0}\cdot\mathcal{F}_{\eta}^{s}\left(\xi\right)\left(0\right)dx+\int_{\omega}\eta_{1}\xi\left(0\right)dA.
\end{aligned}
\end{equation}
This means \eqref{eqn:lim-weak-form} is valid for all test functions of the form $\left(\mathcal{F}_{\eta}^{s}\left(\xi\right),\xi\right)\in\mathcal{T}^{\eta}$. 
We need to show that \eqref{eqn:lim-weak-form} is valid for  arbitrary test function $\left(\mathbf{q},\xi\right)\in\mathcal{T}^{\eta}$ . Let us assume without loss of generality that $\mathbf{q}$ is smooth in space. We need to prove the validity of \eqref{eqn:lim-weak-form} for $\left(\mathbf{q}-\mathcal{F_{\eta}}\left(\xi\right),0\right)\in\mathcal{T}^{\eta}$. Since $\text{tr}_{\eta}\left(\mathbf{q}-\mathcal{F_{\eta}}\left(\xi\right)\right)\cdot\boldsymbol{\nu}^{\eta}=0$ we consider $\mathbf{q}_{\varepsilon}:=\mathcal{J}_{\mathcal{R}_{\varepsilon}\eta_{\varepsilon}}\mathcal{J}_{\eta}^{-1}\left(\mathbf{q}-\mathcal{F_{\eta}}\left(\xi\right)\right)$
 We can see that $\text{div}\ \mathbf{q}_{\varepsilon}=0$ and that $\left(\mathbf{q}_{\varepsilon},0\right)\in\mathcal{T^{\mathcal{R}_{\varepsilon}\eta_{\varepsilon}}}$, by recalling the properties of the Piola transform from Lemma~\ref{lm:Piola}.
 Using \eqref{eqn:2nd-grad-eta} we can obtain that 
 \begin{equation}
\mathcal{J}_{\mathcal{R}_{\varepsilon}\eta_{\varepsilon}}\mathcal{J}_{\eta}^{-1}\left(\mathbf{q}-\mathcal{F_{\eta}}\left(\xi\right)\right)\to\mathcal{J}_{\eta}\mathcal{J}_{\eta}^{-1}\left(\mathbf{q}-\mathcal{F_{\eta}}\left(\xi\right)\right)=\mathbf{q}-\mathcal{F_{\eta}}\left(\xi\right)\quad\text{in}\ L_{t}^{2}H_{x}^{1}
 \end{equation}
which enables us to let $\varepsilon \to 0$ in \eqref{eqn:eps-weak-formulation} for the test function $\left(\mathbf{q}_{\varepsilon},0\right)$. We obtain that 
\begin{equation}\label{eqn:weak-limit-eps-2}
\begin{aligned}\int_{\Omega^{\eta}\left(t\right)}\mathbf{u}\cdot\left(\mathbf{q}-\mathcal{F}_{\eta}^{s}\left(\xi\right)\right)dx+\int_{0}^{t}\int_{\Omega^{\eta}\left(s\right)}-\mathbf{u}\cdot\partial_{t}\left(\mathbf{q}-\mathcal{F}_{\eta}^{s}\left(\xi\right)\right)dxds & +\\
\int_{0}^{t}\int_{\Omega^{\eta}\left(s\right)}\mathbb{D}\mathbf{u}:\mathbb{D}\left(\mathbf{q}-\mathcal{F}_{\eta}^{s}\left(\xi\right)\right)dxds+\int_{0}^{t}b\left(s,\mathbf{u},\mathbf{u},\mathbf{q}-\mathcal{F}_{\eta}^{s}\left(\xi\right)\right)ds & +\\
\int_{0}^{t}\int_{\Gamma^{\eta}\left(s\right)}-\frac{1}{2}\left(\mathbf{u}\cdot\left(\mathbf{q}-\mathcal{F}_{\eta}^{s}\left(\xi\right)\right)\right)\left(\partial_{t}\eta\mathbf{e}_{r}\circ\phi_{\eta_{\varepsilon}\left(t\right)}^{-1}\right)\cdot\boldsymbol{\nu}^{\eta\left(t\right)}dA_{\eta}ds & +\\
\frac{1}{\alpha}\int_{0}^{t}\int_{\omega}\left(\mathbf{u}_{\varepsilon}\circ\phi_{\eta}-\partial_{t}\eta\mathbf{e}_{r}\right)\cdot\left(\left(\mathbf{q}-\mathcal{F}_{\eta}^{s}\left(\xi\right)\right)\circ\phi_{\eta}-\xi\mathbf{e}_{r}\right)J_{\eta}dAds & +\\
\int_{0}^{t}\left\langle F\left(t\right),\mathbf{q}-\mathcal{F}_{\eta}^{s}\left(\xi\right)\right\rangle ds+\int_{\Omega^{\mathcal{R}_{\varepsilon}\eta_{\varepsilon}}\left(0\right)}\mathbf{u}_{0}\cdot\left(\mathbf{q}-\mathcal{F}_{\eta}^{s}\left(\xi\right)\right)\left(0\right)dx.
\end{aligned}
\end{equation}
By adding up \eqref{eqn:lim-weak-form} and \eqref{eqn:weak-limit-eps-2} we arrive at the conclusion: we can let $\varepsilon\to 0$ in Proposition~\ref{prop:existence-eps-reg-soln}. 

\subsection{Proof of Theorem~\ref{thm:main}}\label{ssec:proof-main-thm}
\begin{proof}
    We start by decoupling the problem as in Subsection~\ref{ssec:decoupled} and we prove the existence of a solution for the decoupled and linearized problem in Proposition~\ref{prop:dec-lin-problem}. We recover the coupling (and the nonlinearity) by the fixed argument from Subsection~\ref{ssec:set-v-fixed-point}. We obtain, therefore, that for every $\varepsilon>0$ there exists at least one solution for the regularized problem, as in Proposition~\ref{prop:existence-eps-reg-soln}. Finally, we let $\varepsilon\to0$ as explained in Subsection~\ref{ssec:eps-limit-passage} to conclude the proof.
\end{proof}

\section{Nonlinear Koiter shells}\label{sec:nonlinear-koiter}
The more general nonlinear Koiter elastic energies that we  present below, represents a popular and frequently used model in the literature on elasticity theory, see e.g. \cite{ciarlet2018nonlinear}, \cite[Part B]{ciarlet2000theory} and the references therein. This model generalizes the linear and simplified model presented in Subsection~\ref{ssec:problem}.
We now follow the presentation and notations of \cite[Section 2.1]{MS22}\footnote{This work was the first to study incompressible fluids interacting with nonlinear Koiter shells.}

With the same notations as above,  let us introduce:
\begin{itemize}
    \item The \emph{change of metric tensor} as 
    \begin{equation}
        \mathbb{G}\left(\eta\right):=\left[\begin{array}{cc}
\left(R+\eta\right)^{2}+\left(\partial_{\theta}\eta\right)^{2}-R^{2} & \partial_{\theta}\eta\partial_{z}\eta\\
\partial_{\theta}\eta\partial_{z}\eta & 1+\left(\partial_{z}\eta\right)^{2}
\end{array}\right]
    \end{equation}
\item The \emph{change of curvature tensor} as
\begin{equation}
    \mathbb{R}^{\sharp}\left(\eta\right):=\left[\begin{array}{cc}
\left(1+\frac{\eta}{R}\right)\partial_{\theta\theta}\eta-\frac{1}{R}\left(R+\eta\right)^{2}-\frac{2}{R}\left(\partial_{\theta}\eta\right)^{2}+R & \left(1+\frac{\eta}{R}\right)\partial_{\theta z}\eta-\frac{1}{R}\partial_{\theta}\eta\partial_{z}\eta\\
\left(1+\frac{\eta}{R}\right)\partial_{\theta z}\eta-\frac{1}{R}\partial_{\theta}\eta\partial_{z}\eta & \left(1+\frac{\eta}{R}\right)\partial_{zz}\eta
\end{array}\right]
\end{equation}
\item The \emph{nonlinear Koiter elastic energy} as
\begin{equation}\label{eqn:nonl-koiter}
K\left(\eta\right):=K\left(\eta, \eta \right):= \frac{h}{6}\int_{\omega}\mathcal{A}\mathbb{G}\left(\eta\right):\mathbb{G}\left(\eta\right)dA+\frac{h^{3}}{48}\int_{\omega}\mathcal{A}\mathbb{R}^{\sharp}\left(\eta\right):\mathbb{R}^{\sharp}\left(\eta\right)dA
\end{equation}
\end{itemize}
In \eqref{eqn:nonl-koiter} $h$ denotes the thickness of the shell,
while $\mathcal{A}$ denotes a fourth-order tensor whose entries are the contravariant components of the shell elasticity.
We refer to \cite[p.162]{Ciarlet05} and \cite{ciarlet-roquefort-2001} for further details about $\mathcal{A}$ and the model, where in particular it is shown that that $\mathcal{A}$ is coercive. By denoting \[\gamma\left(\eta \right):=1+\frac{\eta}{R}\] we see that
\begin{equation}
\left\Vert \eta\right\Vert _{L_{x}^{4}}^{4}+\left\Vert \nabla\eta\right\Vert _{L_{x}^{4}}^{4}+\left\Vert \gamma\left(\eta\right)\nabla^{2}\eta\right\Vert _{L_{x}^{2}}^{2}\lesssim K\left(\eta\right)
\end{equation}
which ensures the $H^{2}$ coercivity of $K$ provided that $\gamma\left(\eta \right) \neq 0$ which we ensure throughout the condition $\left\Vert \eta\right\Vert _{L_{x}^{\infty}}<R$. The $H^{2}$ estimates will be our main focus since in two dimensions because the first-order derivatives term can be estimated due to the embeddings $H^{2}\left(\omega\right)\hookrightarrow\hookrightarrow W^{1,4}\left(\omega\right)\hookrightarrow\hookrightarrow L^{\infty}\left(\omega\right)$.

The corresponding Euler-Lagrange equation follows \eqref{eqn:lame} and  reads
\begin{equation}\label{eqn:shell-koiter}
    \partial_{tt}\eta+K^{\prime}\left(\eta\right)=f
\end{equation}
with $K^{\prime}$ the corresponding Fr\'{e}chet derivative. The equation \eqref{eqn:shell-koiter} replaces \eqref{eqn:lame} in \eqref{eqn:FSI-system} and updates it to the following  system
\begin{equation}\label{eqn:nonl-system-fsi}
    \begin{cases}
\partial_{t}\mathbf{u}+\left(\mathbf{u}\cdot\nabla\right)\mathbf{u}=\operatorname{div}\sigma & \text{in}\ I\times\Omega^{\eta}\\
\text{div}\mathbf{u}=0 & \text{in}\ I\times\Omega^{\eta}\\
\partial_{tt}\eta+K^{\prime}\left(\eta\right)=f & \text{on}\ I\times\omega\\
\left|\eta\right|=\left|\nabla\eta\right|=0 & \text{on}\ I\times\partial\omega\\
\frac{1}{2}\left|\mathbf{u}\right|^{2}+p=P(t):=P_{in/out}\left(t\right),\ \mathbf{u}\cdot\boldsymbol{\tau}_{1,2}=0 & \text{on}\ I\times\Gamma_{in/out}\\
\left(\partial_{t}\eta\mathbf{e}_{r}-\mathbf{u}\left(t,\phi_{\eta\left(t\right)}\right)\right)\cdot\boldsymbol{\nu}^{\eta\left(t\right)}=0 & \text{on}\ I\times\omega\\
\left(\partial_{t}\eta\mathbf{e}_{r}-\mathbf{u}\left(t,\phi_{\eta\left(t\right)}\right)-\alpha\ \sigma\left(\phi_{\eta\left(t\right)}\right)\boldsymbol{\nu}^{\eta}\left(t,x\right)\right)\cdot\boldsymbol{\tau}_{1,2}^{\eta}=0 & \text{on}\ I\times\omega\\
\mathbf{u}\left(0,\cdot\right)=\mathbf{u}_{0}, & \text{in \ensuremath{\Omega^{\eta_{0}}}}\\
\eta\left(0,\cdot\right)=\eta_{0},\ \partial_{t}\eta\left(0,\cdot\right)=\eta_{1} & \text{in}\ \omega
\end{cases}
\end{equation}
In order to study the existence of weak solution for \eqref{eqn:nonl-system-fsi} several additional challenges appear, compared to its linearized version \eqref{eqn:FSI-system}, due to the nonlinear effect of $K(\eta)$.

We list below its most important changes:
\begin{itemize}
    \item In the new energy of the system,  the term  $\frac{1}{2}\int_{\omega}\left|\nabla^{2}\eta\right|^{2}dA$ in \eqref{eqn:energy-E-D-Eslip} is replaced with $\frac{1}{2}K\left(\eta\right)$ and thus the new energy, corresponding to \eqref{eqn:nonl-system-fsi} is
    \begin{equation}
      \tilde{E}\left(t\right):=\frac{1}{2}\int_{\Omega^{\eta}\left(t\right)}\left|\mathbf{u}\left(t,x\right)\right|^{2}dx+\frac{1}{2}\int_{\omega}\left|\partial_{t}\eta\right|^{2}dA+\frac{1}{2}K\left(\eta\right).
    \end{equation}
    
\item In the weak formulation the term $\int_{\omega}\nabla^{2}\eta:\nabla^{2}\xi dA$ is replaced by 
\begin{equation}\label{eqn:K-eta-xi}
    K\left(\eta,\xi\right):=\frac{h}{6}\int_{\omega}\mathcal{A}\mathbb{G}\left(\eta\right):\mathbb{G}^{\prime}\left(\eta\right)\xi dA+\frac{h^{3}}{48}\int_{\omega}\mathcal{A}\mathbb{R}^{\sharp}\left(\eta\right):\left(\mathbb{R}^{\sharp}\right)^{\prime}\left(\eta\right)\xi dA
\end{equation}
for $\xi \in H_{0}^{2}(\omega)$. By $\mathbb{G}^{\prime},\left(\mathbb{R}^{\sharp}\right)^{\prime}$ we denote the Fr\'{e}chet derivatives of $\mathbb{G},\mathbb{R}^{\sharp}$ respectively.
\end{itemize}
With this  we can now formulate the analogous result of Theorem~\ref{thm:main}, namely
\begin{theorem}\label{thm:nonl-shell-sect5}
 In the same hypothesis as in Theorem~\ref{thm:main}, the system \eqref{eqn:nonl-system-fsi} admits at least one weak solution, defined by Definition~\ref{def:weak-soln}.
\end{theorem}

\begin{proof} The proof is mostly based on the proof of Theorem~\ref{thm:main} and therefore follows the steps presented during Section~\ref{sec:proof-main}.
    We explain the analogies and the adjustments that need to be made.
\paragraph{The decoupled and linearized problem.} We refer to  Subsection~\ref{ssec:decoupled} for the formulation. The decoupling procedure is almost the same, with the most important modification concerns the linearization done in Subection~\ref{ssec:decoupled} as $K^\prime(\eta)$ is an additional nonlinearity. The key observation is that, since the main focus in on the $H^{2}$ estimates, we should linearize the term $K(\eta,\xi)$ in the weak formulation as follows: 
    the change of metric tensor is updated to its linearized version 
    \begin{equation}\label{eqn:chg-metric-lin}
        \mathbb{G}_{\delta}\left(\eta\right):=\left[\begin{array}{cc}
\left(R+\delta\right)\left(R+\eta\right)+\partial_{\theta}\delta\partial_{\theta}\eta-R^{2} & \frac{\partial_{\theta}\delta\partial_{z}\eta+\partial_{\theta}\eta\partial_{z}\delta}{2}\\
\frac{\partial_{\theta}\delta\partial_{z}\eta+\partial_{\theta}\eta\partial_{z}\delta}{2} & 1+\partial_{z}\delta\partial_{z}\eta
\end{array}\right]
    \end{equation}
and the linearized change of curvature tensor is given by

\begin{equation}\label{eqn:chg-curv-lin}
    \mathbb{R}_{\delta}^{\sharp}\left(\eta\right):=\left[\begin{array}{cc}
\left(1+\frac{\delta}{R}\right)\partial_{\theta\theta}\eta-\frac{1}{R}\left(R+\eta\right)\left(R+\delta\right)-\frac{2}{R}\partial_{\theta}\delta\partial_{\theta}\eta+R & \left(1+\frac{\delta}{R}\right)\partial_{\theta z}\eta-\frac{1}{R}\frac{\partial_{\theta}\delta\partial_{z}\eta+\partial_{\theta}\eta\partial_{z}\delta}{2}\\
\left(1+\frac{\delta}{R}\right)\partial_{\theta z}\eta-\frac{1}{R}\frac{\partial_{\theta}\delta\partial_{z}\eta+\partial_{\theta}\eta\partial_{z}\delta}{2} & \left(1+\frac{\delta}{R}\right)\partial_{zz}\eta
\end{array}\right]
\end{equation}
    We assemble them to define the Koiter energy (associated to the lineraized tensors) as  
    \begin{equation}
        K_{\delta}\left(\eta\right):=\frac{h}{6}\int_{\omega}\mathcal{A}\mathbb{G}_{\delta}\left(\eta\right):\mathbb{G}_{\delta}\left(\eta\right)dA+\frac{h^{3}}{48}\int_{\omega}\mathcal{A}\mathbb{R}_{\delta}^{\sharp}\left(\eta\right):\mathbb{R}_{\delta}^{\sharp}\left(\eta\right)dA.
    \end{equation}
We point out that we preserve the $H^{2}$ estimates of $\eta$ since the term $\gamma(\eta)=1+\frac{\eta}{R}$ appearing in \eqref{eqn:chg-curv-lin} is replaced by $\gamma(\delta)$ for $\delta$'s such that $\left\Vert \delta\right\Vert _{L_{t,x}^{\infty}}<R$ which ensure that $\gamma\left(\delta\right)>0$.
The remaining terms, containing lower-order derivatives of $\eta$ can be estimated by $\left\Vert \eta\right\Vert _{H_{x}^{2}}$ and the appropriate norms of $\delta$ --which is given and smooth as in Subsection~\ref{ssec:decoupled}.

With this observation, the construction of the decoupled and regularized solution follows in the same manner.
\paragraph{The fixed-point and the limiting process.} We recall Subsections~\ref{ssec:set-v-fixed-point} and ~\ref{ssec:eps-limit-passage}. The linearization as above is important as the set-valued Theorem~\ref{thm:kakutani} requires convexity. The fixed-point $\delta \mapsto \eta$ can be performed without additional difficulties. We recover $K_{\delta}\left(\eta\right)=K\left(\eta\right)$ and $K_{\delta}\left(\eta,\xi\right)=K\left(\eta,\xi\right)$, recall here \eqref{eqn:nonl-koiter} and \eqref{eqn:K-eta-xi}. Concerning the limiting procedure $\varepsilon \to 0$ as in Subsection~\ref{ssec:compactness}, in order to pass to the limit in the term $K\left(\eta_{\varepsilon},\xi\right)$ we see, following \cite[Section 6]{MS22} that the most tedious terms are the ones involving $\gamma\left(\eta_{\varepsilon} \right)\nabla^{2}\eta_{\varepsilon}$, as the limit in the lower order terms can be justified by an application of the Aubin-Lions lemma, using only energy estimates, see~\eqref{eqn:comp-emd-AL}. To pass to the limit in the terms containing $\left(1+\frac{\eta_{\varepsilon}}{R}\right)\nabla^{2}\eta_{\varepsilon}$ a crucial observation of \cite{MS22} was that the following additional estimate can be established:
\begin{equation}
    \sup_{\varepsilon>0}\left(\left\Vert \eta_{\varepsilon}\right\Vert _{L_{t}^{2}H_{x}^{2+s}}+\left\Vert \partial_{t}\eta_{\varepsilon}\right\Vert _{L_{t}^{2}H_{x}^{s}}\right)<\infty,\quad s\in\left(0,\frac{1}{2}\right).
\end{equation}
These improved estimates are  a consequence of the equations and the use of  appropriate test functions involving difference quotients of fractional type-- see \cite[Theorem 1.2]{MS22} and its proof.
These new estimates can be proved in our situation, too, using the extension operator $\mathcal{F}_{\eta}$ from Proposition~\ref{prop:extension-no-slip}. 
We have proposed in Proposition~\ref{prop:2nd-gradients-convergence} a more direct approach. In any case it holds that $\nabla^{2}\eta_{\varepsilon}\to\nabla^{2}\eta$ in $L^{2}_{t}L^{2}_{x}$ so $\gamma\left(\eta_{\varepsilon}\right)\nabla^{2}\eta_{\varepsilon}\to\gamma\left(\eta\right)\nabla^{2}\eta$ in the same space and  we obtain that $K\left(\eta_{\varepsilon},\xi\right)\to K\left(\eta,\xi\right)$. This limit passage is the final step in proving Theorem~\ref{thm:main} and also in the proof of Theorem~\ref{thm:main-nonlinear}.
\end{proof}

\paragraph{Acknowledgements.} A.R is supported by the Grant RYC2022-036183-I funded by 
MICIU/AEI/10.13039/501100011033 
and by ESF+. A.R, C.M have been partially supported by the Basque Government through the BERC 2022-2025 program and by the Spanish State Research Agency through BCAM Severo Ochoa CEX2021-001142-S and through project PID2023-146764NB-I00 funded by MICIU/AEI/10.13039/501100011033 and cofunded by the European Union.
\paragraph{Data Availability Statement.} Our manuscript has no available data.
\bibliographystyle{plain}
\bibliography{bibliography} 
\medskip
\textbf{Claudiu M\^{i}ndril\u{a}} \\
Basque Center for Applied Mathematics (BCAM), 
Alameda de Mazarredo 14, 48009 Bilbao, Spain \\
Email: \href{mailto:cmindrila@bcamath.org}{cmindrila@bcamath.org}

\vspace{1em}
\noindent
\textbf{Arnab Roy} \\
Basque Center for Applied Mathematics (BCAM), 
Alameda de Mazarredo 14, 48009 Bilbao, Spain \\
IKERBASQUE, Basque Foundation for Science, 
Plaza Euskadi 5, 48009 Bilbao, Bizkaia, Spain\\
Email: \href{mailto:aroy@bcamath.org}{aroy@bcamath.org}

\end{document}